\documentclass[final, 11pt,reqno]{amsart}
\usepackage{amssymb,amsmath,graphicx,amsfonts,euscript}
\usepackage{color}
\usepackage[pagewise]{lineno}
\usepackage{showkeys}
\usepackage{fancyhdr}
\usepackage[margin=1in]{geometry}
\usepackage{titletoc}
\usepackage[linktoc=page]{hyperref}
\hypersetup{colorlinks=false, pdfstartview=FitV, linkcolor=blue,citecolor=blue, urlcolor=blue}
\newtheorem{theorem}{\bf Theorem}[section]

\newtheorem{proposition}[theorem]{\bf Proposition}

\newtheorem{remark}{\bf Remark}

\newtheorem{lemma}[theorem]{\bf Lemma}

\def \div {\,{\rm div}\,}
\def \Re {\,{\rm Re}\,}

\newcommand{\beq}{\begin{equation}}
\newcommand{\eeq}{\end{equation}}
\newcommand{\ben}{\begin{eqnarray}}
\newcommand{\een}{\end{eqnarray}}
\newcommand{\beno}{\begin{eqnarray*}}
\newcommand{\eeno}{\end{eqnarray*}}

\numberwithin{equation}{section}
\subjclass[2010]{35Q30,  35Q35, 76E06}
\keywords{3-D Navier-Stokes equation, Planar helical flow,  Nonlinear stability}

\makeatother

\title[Nonlinear stability for the planar helical flow]{Nonlinear stability and transition threshold for the planar helical flow}

\author[Binbin Shi]{Binbin Shi}
\address{School of Mathematics and Statistics, Nanjing University of Science and Technology, Nanjing,210094, P.R. China.}
\email{shibb@njust.edu.cn}

\author[Yucheng Wang]{Yucheng Wang}
\address{School of Business Informatics and Mathematics, University of Mannheim, Mannheim, 68159, Germany.}
\email{yucheng.wang@uni-mannheim.de}

\begin{document}

\begin{abstract}

In this paper, we study the nonlinear stability for the 3-D planar helical flow $(\delta^2\sin(m_0 y),\delta^2\cos(m_0 y),0)$  on torus $\mathbb{T}^3=\{(x_1,x_2,y)\big|x_1,x_2\in \mathbb{T}_{2\pi}, y\in \mathbb{T}_{2\pi \delta}, \delta\geq1\}$ for high Reynolds number $Re$. We prove that if the initial velocity $U_0$ satisfies
$$
\left\|U_0-(\delta^2\sin(m_0 y),\delta^2\cos(m_0 y),0)\right\|_{X_0}\leq c_0 Re^{-7/4}
$$
for some $c_0>0$ independent of $Re$, then the solution of 3-D incompressible Navier-Stokes equation is global in time and does not transit away from the planar helical flow. Here $\delta>1, m_0=\delta^{-1}$ and the norm $\|\cdot\|_{X_0}$ is defined in \eqref{eq:1.8}. This is a nonlinear stability result for 3-D non-shear flow and the transition threshold is less than $7/4$.

\end{abstract}

\maketitle

\titlecontents{section}[0pt]{\vspace{0\baselineskip}}
{\thecontentslabel\quad}{}%
{\hspace{0em}\titlerule*[10pt]{$\cdot$}\contentspage}

\titlecontents{subsection}[1em]{\vspace{0\baselineskip}}
{\thecontentslabel\quad}{}%
{\hspace{0em}\titlerule*[10pt]{$\cdot$}\contentspage}

\titlecontents{subsubsection}[2em]{\vspace{0\baselineskip}}
{\thecontentslabel\quad}{}%
{\hspace{0em}\titlerule*[10pt]{$\cdot$}\contentspage}

\setcounter{tocdepth}{2}
\tableofcontents

\section{Introduction}\label{sec.1}

In this paper, we consider the three dimensional incompressible Navier-Stokes equation on torus $\mathbb{T}^3=\{(x_1,x_2,y)\big|x_1,x_2\in \mathbb{T}_{2\pi}, y\in \mathbb{T}_{2\pi \delta}, \delta\geq1\}$
\begin{equation}\label{eq:1.1}
\begin{cases}
\partial_t U-\nu\Delta U+U\cdot\nabla U+\nabla P=F,\\
\nabla\cdot U=0,\\
U(0,x_1,x_2,y)=U_0(x_1,x_2,y).
\end{cases}
\end{equation}
Here $\nu=Re^{-1}$ is the viscosity coefficient and $Re$ is the Reynolds number, $U(t,x_1,x_2,y)$ is the velocity field, $P(t,x_1,x_2,y)$ is the pressure, and $F$ is the external force given by
$$
F=\left(\nu\sin(m_0 y),\nu\cos(m_0 y),0\right),
$$
where $m_0=\delta^{-1}$. A non-shear steady solution of \eqref{eq:1.1} is given by
\begin{equation}\label{eq:1.2}
U^*=\left(\delta^2\sin (m_0 y),\delta^2\cos (m_0 y),0 \right),\ \ \ P^*=C,
\end{equation}
which is called the {\em planar helical flow}, see~\cite{FBW.2022}, where $C$ is a constant. In this paper, we want to understand the hydrodynamics stability of this flow at high Reynolds number~($\nu\rightarrow 0$).

\vskip .05in

\subsection{Background and previous works}

Beginning with Reynolds's famous experiment~\cite{Reynolds1883} in 1883, the hydrodynamics stability has been an active field in fluid mechanics, which is mainly concerned with the transition of motion from laminar to turbulent flow at high Reynolds number, see~\cite{DR.1982,SH.2001,TTRD.1993,Yaglom.2012}. It is well known that some laminar flows such as Couette flow, Poiseuille flow and Kolmogorov flow are linearly stable for any Reynolds numbers, which are called the shear flow. However, these flows could be unstable and transition to turbulence to even small perturbations when the Reynolds numbers are large, which is referred to as subcritical transition \cite{DHB.1992}. Up to now, the hydrodynamics stability theory still lacks a good understanding~\cite{Eckert.2010}. Based on the research on the hydrodynamics stability, Trefethe et al. first proposed the {\em transition threshold problem}: How much disturbance will lead to the instability of the flow and the dependence of disturbance on Reynolds number? see~\cite{TTRD.1993}. And mathematical version was formulated in~\cite{BGM.2019}:~{\em Given a norm $\|\cdot\|_{X}$, find a $\beta=\beta(X)$ such that}
$$
\begin{aligned}
&\left\|V_0\right\|_{X}\ll Re^{-\beta}\Rightarrow stability,\\
&\left\|V_0\right\|_{X}\gg Re^{-\beta}\Rightarrow instability.
\end{aligned}
$$
Here the exponent $\beta$ is referred to as the transition threshold and $V_0$ is initial disturbance. There are a lot of works devoted to estimating $\beta$ for some important laminar flows such as shear flow via the numerical or asymptotic analysis, see~\cite{Chapman.2002,LHR.1994,RSBH.1998} and references therein.

\vskip .05in

Recently, there is a lot of important progress on the transition threshold problem for some shear flow via rigorous mathematical analysis. In the case of two dimensions, many scholars have conducted research and made progress, we will not introduce specifically it here, we can see the Couette flow~\cite{BMV.2016,BVW.2018,CLWZ.2020,MZ.2022,WZ.2023}, the Poiseuille flow~\cite{CEW.2020,Zotto.2023,DL.2022} and the Kolmogorov flow~\cite{LX.2019,WZZ.2020}. As the fact that physical space is three dimensional case, the three dimensional transition threshold problems are important and the most concerning for scholars. In particular, for the 3-D Couette flow, Bedrossian et al. proved that the transition threshold is less than 1 for Gevrey class perturbation~\cite{BGM.2020} and the transition threshold is less than 3/2 for Sobolev perturbation~\cite{BGM.2017}. Wei et al. proved that the transition threshold is less than 1 for Sobolev perturbation~\cite{WZ.2021}. Chen et al. considered transition threshold problem of the 3-D Couette flow with non-slip boundary condition and showed that the transition threshold is less than 1 for Sobolev perturbation~\cite{CWZ.0001}. For the 3-D plane Poiseuille flow, Chen et al. considered transition threshold problem with non-slip boundary condition and showed that the transition threshold is less than 7/4 for Sobolev perturbation~\cite{CDLZ.0002}. For the 3-D Kolmogorov flow, Li et al. proved that the transition threshold is less than 7/4 for Sobolev perturbation~\cite{LWZ.2020}. Here the transition threshold should not be optimal, some more subtle analysis is needed to improve it. In addition, the nonlinear stability and transition threshold problem for other situations have also been studied by scholars, the related research can refer to~\cite{AHL.0003,BHIW.0004,DWZ.2021,Liss.2020,ZZ.2023}. While for the case of pipe flow~\cite{Kerswell.2005}, it is an ancient and unsolved problem for understanding hydrodynamic stability. A classical example of laminar flow is the 3-D pipe Poiseuille flow, Chen et al.~proved the linear stability at high Reynolds number regime~\cite{CWZ.2023}. And the transition threshold of this flow may be $\beta=1$ by experimental and numerical results, see~\cite{HJM.2003,MM.2007}. However, it is challenging to solve this conjecture.

\vskip .05in
In mathematics and physics, the Arnold-Beltrami-Childress ($ABC$) flow
$$
U_{ABC}(x_1,x_2,y)=\left(A\sin y+C\cos x_2, A\cos y+B\sin x_1, B\cos x_1+C\sin x_2\right),
$$
describes a kind of stationary flow of incompressible fluid with periodic boundary conditions, where $U_{ABC}$ is the velocity vector field, $A$, $B$ and $C$ are arbitrary constants. The $ABC$ flow was first discovered by Arnold \cite{Arnold.1965} as a class of three dimensional steady-state solutions of the Euler equations or the Navier-Stokes equations with external force per unit mass. The $ABC$ flow is known in fluid mechanics, which are 3-D simple model  and have aroused wide interest in nonlinear dynamics, hydrodynamics, and magnetohydrodynamics. For different constants A, B and C, the $ABC$ flow has been extensively studied in analysis and numerical, some specific research contents can be seen in the reference \cite{Childress.1970,DU.2018,DU.201802,DFGHMS.1986,GF.1987,HZD.1998,MXYZ.2016,QL.2023,ZKLH.1993}. In this paper, the planar helical flow is  a kind of $ABC$ flow if $A=\delta^2, B=C=0$ and periodic domain is $\mathbb{T}_{2\pi}\times\mathbb{T}_{2\pi}\times \mathbb{T}_{2\pi\delta}, \delta\geq1$. Thus, we also have good physical motivation to study planar helical flow.

\vskip .05in

To our knowledge, the study of stability and transition of laminar flow focuses on the shear flow case. We want to consider the nonlinear stability and transition threshold problems of planar helical flow in this paper, which is a non-shear incompressible flow, see~\eqref{eq:1.2}. In~\cite{FBW.2022}, we studied the enhanced dissipation of planar helical flow, and the enhanced dissipation rate is the same as that of the 3-D Kolmogorov flow~\cite{LWZ.2020}. However, the conditions of enhanced dissipation are different, the enhanced dissipation of shear flow occurs in the non-zero mode of $x_1$, while the enhanced dissipation of planar helical flow occurs in the non-zero mode of $x_1$ or $x_2$. Otherwise, the planar helical flow and shear flow have different forms of movement, the velocity of shear flow changes and the direction remains unchanged if $y$ changes, while the velocity remains unchanged and direction of planar helical flow change. Obviously, the planar helical flow is a laminar flow. Based on these, it is an interesting problem for considering the stability and transition of planar helical flow.

\vskip .05in

\subsection{Main problem and result}
In this paper, we study the nonlinear stability and transition threshold problems of planar helical flow via rigorous mathematical analysis. This question is motivated by works of Li et al.~\cite{LWZ.2020} and Feng et al.~\cite{FBW.2022}. In fact, since the planar helical flow in \eqref{eq:1.2} is a steady solution of 3-D incompressible Navier-Stokes equation \eqref{eq:1.1}, we consider the stability of solution to \eqref{eq:1.1} around the planar helical flow. To this end, we introduce the perturbation $V=(v_1,v_2,v_3)=U-U^*$, which satisfies from \eqref{eq:1.1} that
\begin{equation}\label{eq:1.3}
\begin{cases}
(\partial_t+\delta^2\sin (m_0 y)\partial_{x_1}+\delta^2\cos (m_0 y)\partial_{x_2}-\nu\Delta)v_1+\delta\cos(m_0 y)v_3=-\partial_{x_1}P-V\cdot\nabla v_1,\\
(\partial_t+\delta^2\sin (m_0 y)\partial_{x_1}+\delta^2\cos (m_0 y)\partial_{x_2}-\nu\Delta)v_2-\delta\sin(m_0 y)v_3=-\partial_{x_2}P-V\cdot\nabla v_2,\\
(\partial_t+\delta^2\sin (m_0 y)\partial_{x_1}+\delta^2\cos (m_0 y)\partial_{x_2}-\nu\Delta)v_3=-\partial_{y}P-V\cdot\nabla v_3,\\
V(0,x_1,x_2,y)=V_0(x_1,x_2,y),
\end{cases}
\end{equation}
where $\div V=0$ and $U^\ast$ is given in \eqref{eq:1.2}. Define
\begin{equation}\label{eq:1.4}
\omega_3=\partial_{x_1}v_2-\partial_{x_2}v_1,
\end{equation}
the \eqref{eq:1.3} becomes
\begin{equation}\label{eq:1.5}
\begin{cases}
(\partial_t+\delta^2\sin (m_0 y)\partial_{x_1}+\delta^2\cos (m_0 y)\partial_{x_2}-\nu\Delta)\Delta v_3\\
\ \ \ \ \ \ \ \ +\sin (m_0 y)\partial_{x_1}v_3+\cos (m_0 y)\partial_{x_2}v_3=-\partial_y\Delta p-\Delta(V\cdot\nabla v_3),\\
(\partial_t+\delta^2\sin (m_0 y)\partial_{x_1}+\delta^2\cos (m_0 y)\partial_{x_2}-\nu\Delta)\omega_3\\
\ \ \ \ \ \ \ \ -\delta\sin(m_0 y)\partial_{x_1}v_3-\delta\cos(m_0 y)\partial_{x_2}v_3=\nabla\cdot(-V\omega_3+\omega v_3),
\end{cases}
\end{equation}
where
\begin{equation}\label{eq:1.6}
\omega=\nabla\times V=(\omega_1,\omega_2,\omega_3), \ \ \ p=-\Delta^{-1}\left(\sum_{i,j=1}^{3}\partial_iv_j\partial_jv_i \right),
\end{equation}
and $\partial_1=\partial_{x_1}, \partial_2=\partial_{x_2}, \partial_3=\partial_{y}$. The idea of writing the coupled system to $\Delta v_3$ and $\omega_3$ in \eqref{eq:1.5} from  \eqref{eq:1.3} may go back to Kelvin's original paper \cite{Kelvin.1887} and the works of Li et al.~\cite{LWZ.2020}. In this paper, we mainly consider the perturbation equations \eqref{eq:1.3} and \eqref{eq:1.5}.

\vskip .05in

In \cite{FBW.2022}, we introduced the enhanced dissipation of planar helical flow. We know that the zero mode part is one dimension and non-zero part is three dimensions, and enhanced dissipation only occurs in the non-zero part of equation. The goal of this paper is to study the stability threshold of  planar helical flow, the enhanced dissipation of this flow plays an important role. Thus we need to analysis the zero mode and non-zero mode of \eqref{eq:1.3}. Similar to \cite{FBW.2022}, we define the zero mode and non-zero mode for $f(x_1,x_2,y)$, it is as follows
\begin{equation}\label{eq:1.7}
P_0f=\frac{1}{4\pi^2}\int_{\mathbb{T}^2}f(x_1,x_2,y)dx_1dx_2,\ \ \ f_{\neq}=P_{\neq}f=f-P_0f.
\end{equation}
We denote the Sobolev space $X_0(\mathbb{T}^3)$ with the norm
\begin{equation}\label{eq:1.8}
\big\|f\big\|^2_{X_0}=\big\|P_0f\big\|^2_{H^2}+\big\|\partial_{x_1}f\big\|^2_{H^2}+\big\|\partial_{x_2}f\big\|^2_{H^2}.
\end{equation}
Our main result is stated as follows
\begin{theorem}\label{thm:1.1}
Let $\delta>1$ and initial data $V_0\in X_0(\mathbb{T}^3)$, there exist $c_0,\epsilon,\nu_0\in (0,1)$ and a positive constant $C_0$ independent of $\nu$, such that if for any $\nu\in (0,\nu_0)$ and
$$
\left\|V_0\right\|_{X_0}\leq c_0\nu^{7/4}.
$$
Then the solution $V$ of \eqref{eq:1.3} is global in time and the following stability estimates
$$
\begin{aligned}
&\left\|v_3\right\|_{X_0}+e^{\epsilon\nu^{1/2}t}\left(\big\|\partial_{x_1}(\Delta v_3)\big\|_{L^2}+\big\|\partial_{x_2}(\Delta v_3)\big\|_{L^2}\right)\\
&\ \ \ \ +e^{\epsilon\nu^{1/2}t}\left(\big\|\partial_{x_1}\omega_3\big\|_{L^2}+\big\|\partial_{x_2}\omega_3\big\|_{L^2}\right)\leq C_0\left\|V_0\right\|_{X_0}
\end{aligned}
$$
and
$$
\left\|V\right\|_{X_0}\leq C_0\nu^{-1}\left\|V_0\right\|_{X_0}
$$
hold, where the $\omega_3$ is defined in \eqref{eq:1.4}.
\end{theorem}

Let us give some explanations for the Theorem \ref{thm:1.1}. It is a result for the transition threshold problem of non-shear flow and stability threshold $\beta\leq 7/4$. We believe that this result is not optimal. From the Theorem \ref{thm:1.1}, one has
$$
\big\|\partial_{x_1}(\Delta v_3)\big\|_{L^2}+\big\|\partial_{x_2}(\Delta v_3)\big\|_{L^2}
+\big\|\partial_{x_1}\omega_3\big\|_{L^2}+\big\|\partial_{x_2}\omega_3\big\|_{L^2}\lesssim e^{-\epsilon\nu^{1/2}t}\big\|V_0\big\|_{X_0},
$$
and since
\begin{equation}\label{eq:1.9}
\partial_{x_1}P_0=\partial_{x_2}P_0=0,\ \ \ P_{\neq}\partial_{x_1}=\partial_{x_1},\ \ \ P_{\neq}\partial_{x_2}=\partial_{x_2},
\end{equation}
then we know that the $\partial_{x_1}\Delta v_3$, $\partial_{x_2}\Delta v_3$, $\partial_{x_1}\omega_3$ and $\partial_{x_2}\omega_3$ are non-zero modes of \eqref{eq:1.5} and have enhanced dissipation decay. And we imply from
$$
\left\|V\right\|_{X_0}\lesssim\left\|V_0\right\|_{X_0}
$$
that the planar helical flow is nonlinear stability in $X_0$ sense if the perturbation is $o(\nu^{7/4})$.

\vskip .05in

\subsection{Ideas and sketch of the proof}
In the following, we briefly state our main ideas of the proof. Firstly, since the proof relies on the enhanced dissipation decay estimate for linearized equation of \eqref{eq:1.5}, it is as follows
\begin{equation}\label{eq:1.10}
\begin{cases}
(\partial_t+\mathcal{L})\Delta v_3=0,\\
(\partial_t+\mathcal{H})\omega_3=\delta\sin(m_0 y)\partial_{x_1}v_3+\delta\cos(m_0 y)\partial_{x_2}v_3,
\end{cases}
\end{equation}
the \eqref{eq:1.10} is linearized Navier-Stokes equation around planar helical flow. Here linearized operator
\begin{equation}\label{eq:1.11}
\mathcal{L}=\delta^2\sin (m_0 y)\partial_{x_1}+\delta^2\cos (m_0 y)\partial_{x_2}-\nu\Delta +\sin (m_0 y)\partial_{x_1}\Delta^{-1}+\cos (m_0 y)\partial_{x_2}\Delta^{-1},
\end{equation}
and
\begin{equation}\label{eq:1.12}
\mathcal{H}=\delta^2\sin (m_0 y)\partial_{x_1}+\delta^2\cos (m_0 y)\partial_{x_2}-\nu\Delta.
\end{equation}
To study the linearized Navier-Stokes equation \eqref{eq:1.10}, we first consider the resolvent estimate of linearized operator $\mathcal{L}$ and $\mathcal{H}$, where the $\mathcal{L}$ is nonlocal. By scaling and taking the Fourier transform in $x_1$ and $x_2$, the linearized operator $\mathcal{L}$ and $\mathcal{H}$ are reduced to
$$
\mathcal{L}_{k_1,k_2}=-\nu m^2_0\partial^2_{y}+i|k|\delta^2\sin (y+\alpha_k)\left(1-(\alpha^2-\partial^2_{y})^{-1}\right),
$$
and
$$
\mathcal{H}_{k_1,k_2}
=-\nu m^2_0\partial^2_{y}+i|k|\delta^2\sin (y+\alpha_k),
$$
where
$$
|k|=\sqrt{k_1^2+k_2^2},\ \ k=(k_1,k_2)\neq \mathbf{0}, \ \ \alpha^2=|k|^2\delta^2, \ \ \cos \alpha_k=\frac{k_1}{|k|},\ \ \sin \alpha_k=\frac{k_2}{|k|}.
$$
The resolvent estimate of operator $\mathcal{H}_{k_1,k_2}$ can refer to \cite{FBW.2022}. The resolvent estimate of operator $\mathcal{L}_{k_1,k_2}$ is difficult, since it is nonlocal and non self-adjoint. In this paper, we use the inner product $\langle\cdot,\cdot\rangle_{\ast}$
\begin{equation}\label{eq:1.13}
\langle u,w\rangle_{\ast}=\langle u,w-(\alpha^2-\partial^2_{y})^{-1}w\rangle,\ \ \ \alpha^2>1,
\end{equation}
and the following resolvent estimate
\begin{equation}\label{eq:1.14}
\left\|L_\lambda w\right\|_{L^2}\geq C|\gamma|^{1/2}\left(1-\alpha^{-2}\right)\|w\|_{L^2},
\end{equation}
to solve these difficulties, where
$$
L_\lambda w=i\frac{\gamma}{\nu}\left[(\sin y-\lambda)w+\sin y\varphi\right]-\nu\partial^2_yw,
$$
and $(\partial^2_y-\alpha^2)\varphi=w, \gamma,\alpha,\lambda\in \mathbb{R}$, $\nu>0$ and $|\gamma|\gg \nu^2, \alpha^2>1$, the detailed introduction of \eqref{eq:1.13} and  \eqref{eq:1.14} can be referred to \cite{LWZ.2020}. The specific discussion of resolvent estimate for linearized operator $\mathcal{H}_{k_1,k_2}$ and $\mathcal{L}_{k_1,k_2}$ can be seen in  Section \ref{sec.2} and Section \ref{sec.3}.

\vskip .05in

Motivated by \cite{GGN.2009,LWZ.2020}, we define the pseudospectral quantity $\Psi(H)$ for m-accretive operator $H$, see \eqref{eq:2.3} in Section \ref{sec.2}. The pseudospectral plays an important role in studying the hydrodynamics stability due to the non normality of linearized operator in stability problem of fluid mechanics, see \cite{LWZ.202002,LWZ.2020,Trefethen.1997,TTRD.1993}. And we obtain the pseudospectral bound of the operator $\mathcal{H}_{k_1,k_2}$ and $\mathcal{L}_{k_1,k_2}$ by resolvent estimate in Lemma \ref{lem:2.3} and \ref{lem:3.2}, as follows
$$
\Psi(\mathcal{H}_{k_1,k_2})\geq c_1|k|^{1/2}\nu^{1/2}, \ \ \ \ \Psi(\mathcal{L}_{k_1,k_2})\geq c_2|k|^{1/2}\nu^{1/2}(1-\alpha^{-2})^{3/2}.
$$

\vskip .05in

Next, we consider the enhanced dissipation of planar helical flow. For a m-accretive operator, Wei \cite{Wei.2021} established the semigroup estimate from the pseudospectral, see Lemma \ref{lem:2.1} in Section \ref{sec.2}. Based on the pseudospectral bound of the operator $\mathcal{H}_{k_1,k_2}$ and $\mathcal{L}_{k_1,k_2}$, we will show by Lemma \ref{lem:2.3} and Proposition \ref{prop:3.3} that
$$
\left\|e^{-t\mathcal{H}}P_{\neq}\right\|_{L^2\rightarrow L^2}\lesssim e^{-c\nu^{1/2}t-\nu t}, \ \ \
\left\|e^{-t\mathcal{L}}P_{\neq}\right\|_{L^2\rightarrow L^2}\lesssim e^{-c\nu^{1/2}t-\nu t},
$$
where
$$
c=\min\left\{c_1, c_2(1-\alpha^{-2})^{3/2}\right\}.
$$
However, these two estimates can not obtain the enhanced dissipation of linearized system \eqref{eq:1.10} since the source term $\delta\sin(m_0 y)\partial_{x_1}v_3$ and $\delta\cos(m_0 y)\partial_{x_2}v_3$ in the second equation of  \eqref{eq:1.10}. In this paper, we handle this problem by symmetry of planar helical flow. Specifically, by scaling and taking the Fourier transform in $x_1$ and $x_2$, the \eqref{eq:1.10} becomes
$$
\begin{cases}
\partial_t u+\left(\nu(k_1^2+k_2^2)+\mathcal{L}_{k_1,k_2}  \right)u=0,\\
\partial_t w+\left(\nu(k_1^2+k_2^2)+\mathcal{H}_{k_1,k_2}  \right)w=-i|k|\delta^3\sin (y+\alpha_k)(\alpha^2-\partial_y^2)^{-1}u,\\
u(0,y)=u_0(y),\ \ \ w(0,y)=w_0(y),
\end{cases}
$$
where $\alpha^2=(k_1^2+k_2^2)\delta^2$. In Proposition \ref{prop:3.6}, we will show that for $\alpha^2>1$, one has
$$
\left\|w(t)\right\|_{L^2}\lesssim e^{-at}\left(\left\|w_0\right\|_{L^2}+\left\|u_0\right\|_{L^2}\right),
$$
where $a=\nu(k_1^2+k_2^2)+c|k|^{1/2}\nu^{1/2}$. The linearized Navier-Stokes equation around 3-D Kolmogorov flow has similar problem, since this flow has no good symmetry, the source term was handled by wave operator methods, see \cite{LWZ.2020}. In this paper, we study the nonlinear stability threshold of \eqref{eq:1.3} in the case of $\delta>1(\alpha^2>1)$, while for the $\delta=1 (\alpha^2=1)$ is a challenge problem, thus we only study the linearized equation as a appendix in this paper, see Appendix \ref{sec.A}.

\vskip .05in

Finally, we study the nonlinear stability threshold of \eqref{eq:1.3}. First, we need to establish the enhanced dissipation decay estimate for the  following inhomogeneous system
$$
\begin{cases}
(\partial_t+\mathcal{L})\Delta v_3=\div f,\\
(\partial_t+\mathcal{H})\omega_3=\delta\sin(m_0 y)\partial_{x_1}v_3+\delta\cos(m_0 y)\partial_{x_2}v_3+\div g.
\end{cases}
$$
In Proposition \ref{prop:4.1} and \ref{prop:4.2}, we will obtain the following estimates
$$
\big\|\partial_{x_1}(\Delta v_3)\big\|^2_{X_{ed}}
\lesssim \big\|\partial_{x_1}(\Delta v_3)(0)\big\|^2_{ L^2}+\nu^{-1}\left\|e^{\epsilon\nu^{1/2}t}\partial_{x_1}f\right\|^2_{L^2 L^2},
$$

$$
\big\|\partial_{x_2}(\Delta v_3)\big\|^2_{X_{ed}}
\lesssim \big\|\partial_{x_2}(\Delta v_3)(0)\big\|^2_{ L^2}+\nu^{-1}\left\|e^{\epsilon\nu^{1/2}t}\partial_{x_2}f\right\|^2_{L^2 L^2},
$$

\vskip .02in

$$
\begin{aligned}
\big\|\partial_{x_1}\omega_3\big\|^2_{X_{ed}}
\lesssim& \big\|\partial_{x_1}(\Delta v_3)(0)\big\|^2_{L^2}+\big\|\partial_{x_1}\omega_3(0)\big\|^2_{L^2}\\
&+\nu^{-1}\left(\left\|e^{\epsilon\nu^{1/2}t}\partial_{x_1}f\right\|^2_{L^2L^2}
+\left\|e^{\epsilon\nu^{1/2}t}\partial_{x_1}g\right\|^2_{L^2L^2}   \right),
\end{aligned}
$$
and
$$
\begin{aligned}
\big\|\partial_{x_2}\omega_3\big\|^2_{X_{ed}}
\lesssim& \big\|\partial_{x_2}(\Delta v_3)(0)\big\|^2_{L^2}+\big\|\partial_{x_2}\omega_3(0)\big\|^2_{L^2}\\
&+\nu^{-1}\left(\left\|e^{\epsilon\nu^{1/2}t}\partial_{x_2}f\right\|^2_{L^2L^2}
+\left\|e^{\epsilon\nu^{1/2}t}\partial_{x_2}g\right\|^2_{L^2L^2}   \right),
\end{aligned}
$$
where the norm $\|\cdot\|_{X_{ed}}$ includes the enhanced dissipation decay, which is defined by
\begin{equation}\label{eq:1.15}
\big\|u\big\|^2_{X_{ed}}=\left\|e^{\epsilon\nu^{1/2}t}u\right\|^2_{L^\infty L^2}+\nu^{1/2}\left\|e^{\epsilon\nu^{1/2}t}u\right\|^2_{L^2 L^2}+\nu\left\|e^{\epsilon\nu^{1/2}t}\nabla u\right\|^2_{L^2 L^2},\ \ \ \epsilon\in (0,c).
\end{equation}
Based on these estimates, we introduce the following two energy functionals
$$
\begin{aligned}
\mathcal{E}_1(T)=&\sup_{0\leq t\leq T}\left[\big\|v_3\big\|_{X_0}+e^{\epsilon\nu^{1/2}t}\left(\big\|\partial_{x_1}(\Delta v_3)\big\|_{L^2}+\big\|\partial_{x_2}(\Delta v_3)\big\|_{L^2}\right)\right]\\
&+\sup_{0\leq t\leq T}\left[e^{\epsilon\nu^{1/2}t}\left(\big\|\partial_{x_1}\omega_3\big\|_{L^2}
+\big\|\partial_{x_2}\omega_3\big\|_{L^2}\right)\right],\\
\mathcal{E}_2(T)=&\sup_{0\leq t\leq T}\left\|V(t)\right\|_{X_0},
\end{aligned}
$$
and in Proposition \ref{prop:4.5}, we will show that if $\mathcal{E}_1(T)\leq \varepsilon_1\nu,\ \mathcal{E}_2(T)\leq \varepsilon_1\nu^{3/4}$, one has the following stability estimates
$$
\mathcal{E}_1(T)\leq C_1\left\|V_0\right\|_{X_0},\ \ \ \ \mathcal{E}_2(T)\leq C_1\nu^{-1}\left\|V_0\right\|_{X_0}.
$$
Thus we obtain global nonlinear stability by the smallness of $V_0$ and continuity argument.

\vskip .05in

The stability threshold of planar helical flow is determined by estimates of $P_0v_1$ and $P_0v_2$, which is non enhanced dissipation decay, it is as follows
$$
(\partial_t-\nu\Delta)\partial^2_yP_0v_1=\div\widetilde{F},\ \ \ (\partial_t-\nu\Delta)\partial^2_yP_0v_2=\div \widetilde{G},
$$
where
$$
\widetilde{F}=-\nabla\left( \delta\cos(m_0 y)P_0v_3\right)-\nabla P_0(V\cdot\nabla v_1),
$$
and
$$
\widetilde{G}=-\nabla\left(\delta\sin(m_0 y)P_0v_3\right)-\nabla P_0(V\cdot\nabla v_2).
$$
Here the $-\nabla\left(\delta\cos(m_0 y)P_0v_3\right)$ and $-\nabla\left(\delta\sin(m_0 y)P_0v_3\right)$
are bad terms, which give rise to the lift-up, see \cite{EP.1975,LWZ.2020}. Since the equation is one dimension, we use the energy method to deal with this effect.

\vskip .05in

In this paper, some ideas come from the study of nonlinear stability to 3-D Kolmogorov flow \cite{LWZ.2020}, but there are also differences in the methods of proof.~The stability threshold of 3-D Kolmogorov flow was determined by $P_0v_1$, which is two dimensions zero mode equation and it is non enhanced dissipation decay. In \cite{LWZ.2020}, the estimate of $P_0v_1$ is subtle, the authors overcome this difficulty by decomposing the equation of $P_0v_1$ into a steady equation and a heat equation with good source terms (see \cite[Theorem 6.10]{LWZ.2020}). In this paper, the stability threshold of planar helical flow is determined by $P_0v_1$ and $P_0v_2$, which are non enhanced dissipation decay. Here we only use the energy method to deal with them, the main reason is that the $P_0v_1$ and $P_0v_2$ are one dimension, the details can see Section \ref{sec.4}. In addition, the enhanced dissipation decay of linearized Navier-Stokes equation and the estimate of vorticity are difficult for the 3-D Kolmogorov flow, the authors developed wave operator method to deal with them (see \cite[Proposition 5.2, Proposition 6.5 and Appendix A.2]{LWZ.2020}). For the planar helical flow, we establish the enhanced dissipation decay of linearized Navier-Stokes equation and the estimate of vorticity by energy estimate and semigroup theory. Here we need to use the symmetry of planar helical flow and consider the stability threshold in the sense of the norm $X_0$ (see \eqref{eq:1.8}), the details can see Section \ref{sec.3} and Section \ref{sec.4}.

\vskip .05in

The rest of this paper is arranged as follows. In Section \ref{sec.2}, we introduce some preparations and useful lemmas. In Section \ref{sec.3}, we prove the enhanced dissipation effect of linearized Navier-Stokes equation around planar helical flow. In Section \ref{sec.4}, we establish the nonlinear stability and prove the Theorem \ref{thm:1.1}. In Appendix \ref{sec.A}, we introduce the linearized Navier-Stokes equation around planar helical flow in the case of $\delta=1$.

\vskip .05in

Throughout the paper, we use the standard notations to denote function spaces and use $C$ to denote a generic constant which may vary from line to line.

\vskip .05in

\section{Preliminaries}\label{sec.2}
In what follows, we provide some  notations and the auxiliary results. The details are as follows.

\subsection{Notations}
For quantities $X,Y$, if there exists a constant $C$ such that $X\leq CY$, we write $X\lesssim Y$. In this paper, we will use some norm to study nonlinear stability, such as $\|\cdot\|_{X_0}$ in \eqref{eq:1.8} and $\|\cdot\|_{X_{ed}}$ in \eqref{eq:1.15}. We will also use the norms $\|\cdot\|_{X_1}$, the definition is as follows
\begin{equation}\label{eq:2.1}
\big\|f\big\|^2_{X_1}=\big\|f\big\|^2_{L^\infty L^2}+\nu\big\|\nabla f\big\|^2_{L^2L^2}.
\end{equation}
There are some useful inequalities, which will be used in the proof. Such as
\begin{equation}\label{eq:2.2}
\big\|f\big\|_{H^2}\lesssim\big\|f\big\|_{X_0},\ \ \ \ \big\|f\big\|_{X_1}\lesssim\big\|f\big\|_{X_{ed}},\ \ \ \
\big\|f_{\neq}\big\|_{L^2}\lesssim\big\|\partial_{x_1}f\big\|_{L^2}+\big\|\partial_{x_2}f\big\|_{L^2}.
\end{equation}
In this paper, we denote
$$
\big\|f\big\|^2_{L^2 L^2}=\int_{0}^{T}\big\|f(t)\big\|^2_{L^2}dt,\ \ \ \|f\|^2_{L^2 H^s}=\int_{0}^{T}\|f(t)\|^2_{H^s}dt,
$$
and
$$
\big\|f\big\|^2_{L^\infty H^2}=\sup_{0\leq t\leq T}\big\|f(t)\big\|^2_{H^2},\ \ \ \big\|f\big\|^2_{L^\infty X_0}=\sup_{0\leq t\leq T}\big\|f(t)\big\|^2_{X_0},
$$
where $T$ is finite or infinite.

\subsection{The operator theory and energy inequalities}
We study the enhanced dissipation effect of incompressible flow by resolvent estimate, thus we need to introduce some the operator theory. Let $(\mathcal{X},\|\cdot\|)$ be a complex Hilbert space and let $H$ be a closed linear operator in $\mathcal{X}$ with domain $D(H)$. $H$ is m-accretive if the left open half-plane is contained in the resolvent set with
$$
(H+\lambda I)^{-1}\in \mathcal{B}(\mathcal{X})\,,\quad \|(H+\lambda I)^{-1}\|\leq (Re \lambda)^{-1},\ \  Re\lambda >0,
$$
where $\mathcal{B}(\mathcal{X})$ denotes the set of bounded linear operators on $\mathcal{X}$ with operator norm $\|\cdot\|$ and $I$ is the identity operator, see \cite{Kato.1966,Wei.2021}.

\vskip .05in

We denote $e^{-tH}$ is a  semigroup with $-H$ as generator and define pseudospectral bound
\begin{equation}\label{eq:2.3}
\Psi(H)=\inf\left\{\big\|(H-i\lambda I)f\big\|: f\in D(H), \lambda\in \mathbb{R}, \big\|f\big\|=1\right \}.
\end{equation}
The following result is the Gearchart-Pr\"{u}ss type theorem for m-accretive operators, see \cite{Wei.2021}.
\begin{lemma}\label{lem:2.1}
Let $H$ be an m-accretive operator in a Hilbert space $\mathcal{X}$. Then for any $t\geq 0$, we have
$$
\left\|e^{-tH}\right\|_{\mathcal{X}\rightarrow\mathcal{X}}\leq e^{-t\Psi(H)+\pi/2},
$$
where $\Psi(H)$ is defined in \eqref{eq:2.3}.
\end{lemma}

We study the nonlinearity stability of \eqref{eq:1.3}, some results of linear decay and nonlinear estimate are helpful for our proof. First, we introduce the following resolvent estimate, see \cite{LWZ.2020}.
\begin{lemma}\label{lem:2.2}
Given $0<\nu\leq1,|\gamma|\gg \nu^2$,and $\lambda\in \mathbb{R}$, there exists a constant $C>0$ independent of $\nu,\gamma,\lambda$, such that
$$
\left\|L_\lambda w\right\|_{L^2}\geq C|\gamma|^{1/2}\left(1-\alpha^{-2}\right)\left\|w\right\|_{L^2},
$$
where
$$
L_\lambda w=i\frac{\gamma}{\nu}\big[(\sin y-\lambda)w+\sin y\varphi\big]-\nu\partial^2_yw,
$$
and $(\partial^2_y-\alpha^2)\varphi=w$.
\end{lemma}
In fact, the operators $\mathcal{H}$ and $\mathcal{H}_{k_1,k_2}$ have been studied in \cite{FBW.2022}, the following lemma introduces the pseudospectral bound of $\mathcal{H}_{k_1,k_2}$ and semigroup estimate of $\mathcal{H}$.
\begin{lemma}\label{lem:2.3}
The operation $\mathcal{H}_{k_1,k_2}$ is m-accretive, if $\nu |k|^{-1}\ll1$, then there exists a positive constant $c_1$ independent of $\nu$ and $|k|$, such that
$$
\Psi(\mathcal{H}_{k_1,k_2})\geq c_1|k|^{1/2}\nu^{1/2},
$$
where $\Psi(\mathcal{H}_{k_1,k_2})$ is defined by \eqref{eq:2.3}. Moreover, the $e^{-t\mathcal{H}}$ is a semigroup with the operator $-\mathcal{H}$ as generator for any $0<c\leq c_1$, then
$$
\left\|e^{-t\mathcal{H}}P_{\neq}\right\|_{L^2\rightarrow L^2}\lesssim e^{-c\nu^{1/2}t-\nu t},
$$
where $P_{\neq}$ is defined in \eqref{eq:1.7}.
\end{lemma}

Similar to 3-D Kolmogorov flow (see \cite{LWZ.2020}), we can establish the following lemma, which is important for the decay estimate of the vorticity, it is
\begin{lemma}\label{lem:2.4}
Let $u$ and $F$ satisfy
$$
(\partial_t+\mathcal{H})u=\div F.
$$
If $P_0u=P_0F=0$, then we have
$$
\big\|u\big\|^2_{X_{ed}}\lesssim \big\|u(0)\big\|^2_{L^2}+\nu^{-1}\left\|e^{\epsilon\nu^{1/2}t}F\right\|^2_{L^2L^2},
$$
where $\mathcal{H}$ is defined in \eqref{eq:1.12} and $\|\cdot\|_{X_{ed}}$ is defined in \eqref{eq:1.15}.
\end{lemma}
Next, we introduce following estimate (see \cite{LWZ.2020}), it will be used in nonlinear estimate.
\begin{lemma}\label{lem:2.5}
For $i\in\{1,2\}$, it holds that
$$
\big\|f_1\partial_{x_i}f_2\big\|_{L^2}\lesssim \left(\big\|\partial_{x_i}f_1\big\|_{L^2}+\big\|f_1\big\|_{L^2}\right)\big\|\Delta f_2\big\|_{L^2}.
$$
\end{lemma}

\vskip .05in

\section{Decay estimate for linearized Navier-Stokes equation}\label{sec.3}

In this section, we will study the enhanced dissipation decay for linearized Navier-Stokes equation around planar helical flow, which is very important for considering the nonlinear stability. First, we need to establish the resolvent estimate of linearized operator.
\subsection{Resolvent estimate of linearized operator} In this paper, we consider the linearized Navier-Stokes equation around planar helical flow, we need to establish the resolvent estimate of linearized operator. First, we study the resolvent estimate of following linear operator
\begin{equation}\label{eq:3.1}
L_\lambda w=i\frac{\gamma}{\nu}\big[(\sin(y+\alpha_0)-\lambda)w+\sin (y+\alpha_0)\varphi\big]-\nu\partial^2_yw,\ \ \ y\in \mathbb{T}_{2\pi},
\end{equation}
where $(\partial^2_y-\alpha^2)\varphi=w, \gamma,\lambda\in \mathbb{R},\nu>0$, $|\alpha|>1$ and $\alpha_0$ is arbitrary. In \cite{LWZ.2020}, Li et al. considered the resolvent estimate of operator $L_\lambda$ in the case of $\alpha_0=0$,
we only need to modify simply to obtain the resolvent estimate of operator $L_\lambda$ in this paper.
\begin{proposition}\label{prop:3.1}
Let $0<\nu\leq1,|\gamma|\gg \nu^2$, $\lambda\in \mathbb{R}$ and $|\alpha|>1$, there exists a constant $C>0$ independent of $\nu,\gamma,\lambda$, such that
$$
\left\|L_\lambda w\right\|_{L^2}\geq C|\gamma|^{1/2}\left(1-\alpha^{-2}\right)\left\|w\right\|_{L^2},\ \ \ \forall~\alpha_0\in [0,2\pi),
$$
where the operator $L_\lambda$ is defined in \eqref{eq:3.1}.
\end{proposition}

\begin{proof}
For any $\alpha_0\in [0,2\pi)$, we define the translation operator $\mathcal{T}_{\alpha_0}$ as follows
$$
\mathcal{T}_{\alpha_0}f(y)=f(y-\alpha_0), \ \ \ f(y)\in L^2(\mathbb{T}_{2\pi}).
$$
Applying the operator $\mathcal{T}_{\alpha_0}$ to the $L_\lambda\omega$ and denoting
$$
\widetilde{w}(y)=w(y-\alpha_0),\ \ \ \widetilde{\varphi}(y)=\varphi(y-\alpha_0),
$$
then we have
\begin{equation}\label{eq:3.2}
\mathcal{T}_{\alpha_0}L_\lambda w=i\frac{\gamma}{\nu}\big[(\sin y-\lambda)\widetilde{w}+\sin y\widetilde{\varphi}\big]-\nu\partial^2_y\widetilde{w},\ \ \ y\in \mathbb{T}_{2\pi},
\end{equation}
where $(\partial^2_y-\alpha^2)\widetilde{\varphi}=\widetilde{w}$. Since $\|\mathcal{T}_{\alpha_0} f\|_{L^2}=\| f\|_{L^2}$, combining Lemma \ref{lem:2.2} and \eqref{eq:3.2}, one has
$$
\begin{aligned}
\left\|L_\lambda w\right\|_{L^2}
&=\left\|\mathcal{T}_{\alpha_0}L_\lambda w\right\|_{L^2}\\
&\geq\left\|i\frac{\gamma}{\nu}\big[(\sin y-\lambda)\widetilde{w}+\sin y\widetilde{\varphi}\big]-\nu\partial^2_y\widetilde{w}\right\|_{L^2}\\
&\geq C|\gamma|^{1/2}\left(1-\alpha^{-2}\right)\left\|\widetilde{w}\right\|_{L^2}\\
&=C|\gamma|^{1/2}\left(1-\alpha^{-2}\right)\left\|w\right\|_{L^2}.
\end{aligned}
$$
This completes the proof of Proposition \ref{prop:3.1}.
\end{proof}

Next, we consider the linearized operators $\mathcal{H}$ and $\mathcal{L}$, by scaling and taking the Fourier transform in $x_1$ and $x_2$, the linearized operators are reduced to
$$
\begin{aligned}
\widetilde{\mathcal{L}}_{k_1,k_2}&=i\delta^2k_1\sin y+i\delta^2k_2\cos y-\nu(-|k|^2+m_0^2\partial^2_{y})\\
 &\ \ \ +ik_1\sin y(-|k|^2+m_0^2\partial^2_{y})^{-1}+ik_2\cos y(-|k|^2+m_0^2\partial^2_{y})^{-1}\\
 &=i|k|\delta^2\left(\frac{k_1}{|k|}\sin y+\frac{k_2}{|k|}\cos y\right)-\nu(-|k|^2+m_0^2\partial^2_{y})\\
 &\ \ \ +i|k|\left(\frac{k_1}{|k|}\sin y+\frac{k_2}{|k|}\cos y\right)(-|k|^2+m_0^2\partial^2_{y})^{-1}\\
 &=i|k|\delta^2\sin (y+\alpha_k)-\nu(-|k|^2+m_0^2\partial^2_{y})+i|k|\sin (y+\alpha_k)(-|k|^2+m_0^2\partial^2_{y})^{-1}\\
 &=i|k|\delta^2\sin (y+\alpha_k)-\nu(-|k|^2+m_0^2\partial^2_{y})+i|k|\delta^2\sin (y+\alpha_k)(-|k|^2\delta^2+\delta^2m_0^2\partial^2_{y})^{-1}\\
 &=i|k|\delta^2\sin (y+\alpha_k)-\nu(-|k|^2+m_0^2\partial^2_{y})+i|k|\delta^2\sin (y+\alpha_k)(-\alpha^2+\partial^2_{y})^{-1}\\
 &=\nu|k|^2-\nu m_0^2\partial^2_{y}+i|k|\delta^2\sin (y+\alpha_k)\left(1-(\alpha^2-\partial^2_{y})^{-1}\right),
\end{aligned}
$$
and similarly, one has
$$
\begin{aligned}
\widetilde{\mathcal{H}}_{k_1,k_2}&=i\delta^2k_1\sin y+i\delta^2k_2\cos y-\nu(-|k|^2+m_0^2\partial^2_{y})\\
 &=i|k|\delta^2\left(\frac{k_1}{|k|}\sin y+\frac{k_2}{|k|}\cos y\right)-\nu(-|k|^2+m_0^2\partial^2_{y})\\
 &=\nu|k|^2-\nu m_0^2\partial^2_{y}+i|k|\delta^2\sin(y+\alpha_k),
\end{aligned}
$$
where
$$
k=(k_1,k_2)\neq{\bf 0},\ \ |k|=\sqrt{k_1^2+k_2^2},\ \ \alpha^2=|k|^2\delta^2, \ \ m_0=\delta^{-1}, \ \ \cos \alpha_k=\frac{k_1}{|k|},\ \  \sin \alpha_k=\frac{k_2}{|k|}.
$$
Since the $\nu|k|^2$ is a constant and has damping effect in the equation, thus we only consider the following operators
\begin{equation}\label{eq:3.3}
\mathcal{L}_{k_1,k_2}=-\nu m_0^2\partial^2_{y}+i|k|\delta^2\sin (y+\alpha_k)\left(1-(\alpha^2-\partial^2_{y})^{-1}\right),
\end{equation}
and
\begin{equation}\label{eq:3.4}
\mathcal{H}_{k_1,k_2}=-\nu m_0^2\partial^2_{y}+i|k|\delta^2\sin(y+\alpha_k).
\end{equation}
In \cite{FBW.2022}, we studied the operators $\mathcal{H}$ and $\mathcal{H}_{k_1,k_2}$, the pseudospectral bound $\Psi(\mathcal{H}_{k_1,k_2})$ and the semigroup estimate of $\mathcal{H}$ were established, the details also see Lemma \ref{lem:2.3} in Section \ref{sec.2}. Thus, we only consider the operators $\mathcal{L}$ and $\mathcal{L}_{k_1,k_2}$ in Section \ref{sec.3}.

\vskip .05in

Similar argument with \cite{LWZ.2020}, we only need to check that the operator $\mathcal{L}_{k_1,k_2}$ is accretive, which implies that the operator is m-accretive. Since the $\mathcal{L}_{k_1,k_2}$ is not self-adjoint, we define a new inner product $\langle\cdot,\cdot\rangle_{\ast}$ as
$$
\left\langle u,w\right\rangle_{\ast}=\left\langle u,w-(\alpha^2-\partial^2_{y})^{-1}w\right\rangle.
$$
Taking $X=L^2(\mathbb{T}_{2\pi})$ with the norm $\|u\|_{\ast}=\langle u,u\rangle^{1/2}_{\ast}$. Obviously, the operator $\mathcal{L}_{k_1,k_2}$ is self-adjoint in new inner product and the domain $D(\mathcal{L}_{k_1,k_2})=H^2(\mathbb{T}_{2\pi})$. If $\alpha^2>1$, we can get
\begin{equation}\label{eq:3.5}
\left(1-\alpha^{-2}\right)\big\|u\big\|^2_{L^2}\leq\big\|u\big\|^2_{\ast}\leq\big\|u\big\|^2_{L^2},
\end{equation}
then the norm $\|\cdot\|_{\ast}$ is equivalent to the $L^2$ norm. And it is easy to check that
$$
\Re\left\langle\mathcal{L}_{k_1,k_2}u,u\right\rangle_{\ast}=-\nu m_0^2\left\langle\partial^2_yu,u\right\rangle_{\ast}
=\nu  m_0^2\left\langle\partial_yu,\partial_y u-\partial_y(\alpha^2-\partial^2_{y})^{-1}u\right\rangle=\nu m_0^2\big\|\partial_y u\big\|^2_{\ast}\geq0,
$$
thus the operator $\mathcal{L}_{k_1,k_2}$ is m-accretive. We define
$$
\Psi(\mathcal{L}_{k_1,k_2})=\inf\left\{\big\|(\mathcal{L}_{k_1,k_2}-i\lambda I)u\big\|_{\ast}: u\in D(\mathcal{L}_{k_1,k_2}), \lambda\in \mathbb{R}, \big\|u\big\|_{\ast}=1 \right\},
$$
the following lemma gives the estimate of pseudospectral bound to the $\mathcal{L}_{k_1,k_2}$.
\begin{lemma}\label{lem:3.2}
If $0<\nu\ll1$ and $\alpha^2>1$, then
$$
\Psi(\mathcal{L}_{k_1,k_2})\geq c_2|k|^{1/2}\nu^{1/2}(1-\alpha^{-2})^{3/2}.
$$
\end{lemma}

\begin{proof}
For any $\lambda\in \mathbb{R}$,  let $\gamma=\nu|k|\delta^4$, $\lambda_1=m_0^2\lambda/|k|$, one has
$$
\begin{aligned}
\mathcal{L}_{k_1,k_2}-i\lambda I&=-\nu m_0^2\partial^2_{y}+i|k|\delta^2\sin (y+\alpha_k)\left(1-(\alpha^2-\partial^2_{y})^{-1}\right)-i\lambda I\\
&=m_0^2\left(-\nu\partial^2_{y}+i|k|\delta^4\sin (y+\alpha_k)+i|k|\delta^4\sin (y+\alpha_k)(\partial^2_{y}-\alpha^2)^{-1}-i\lambda m_0^{-2}\right)\\
&=m_0^2\left(-\nu\partial^2_{y}+i\frac{\gamma}{\nu}\sin (y+\alpha_k)+i\frac{\gamma}{\nu}\sin (y+\alpha_k)(\partial^2_{y}-\alpha^2)^{-1}-i\frac{\gamma}{\nu}\lambda_1\right)\\
&=m_0^2\left(i\frac{\gamma}{\nu}\big[(\sin (y+\alpha_k)-\lambda_1)+\sin (y+\alpha_k)(\partial^2_{y}-\alpha^2)^{-1}\big]-\nu\partial^2_{y}\right),
\end{aligned}
$$
then we have
$$
\left(\mathcal{L}_{k_1,k_2}-i\lambda\right)w=m_0^2\left(i\frac{\gamma}{\nu}\left((\sin (y+\alpha_k)-\lambda_1)w+\sin (y+\alpha_k)\varphi\right)-\nu\partial^2_{y}w\right),
$$
where $\varphi=(\partial^2_y-\alpha^2)^{-1}w$. Combining Proposition \ref{prop:3.1} and \eqref{eq:3.5}, we imply that
$$
\begin{aligned}
\big\|(\mathcal{L}_{k_1,k_2}-i\lambda)w\big\|^2_{\ast}
&\geq \left(1-\alpha^{-2}\right)\big\|(\mathcal{L}_{k_1,k_2}-i\lambda)w\big\|^2_{L^2}\\
&=\left(1-\alpha^{-2}\right)m_0^2\left\|i\frac{\gamma}{\nu}\big[(\sin (y+\alpha_k)-\lambda_1)w+\sin (y+\alpha_k)\varphi\big]-\nu\partial^2_{y}w\right\|^2_{L^2}\\
&\geq\left(1-\alpha^{-2}\right)m_0^2C^2|\gamma|\left(1-\alpha^{-2}\right)^2\big\|w\big\|^2_{L^2}\\
&\geq \left(Cm_0\right)^2|\gamma|\left(1-\alpha^{-2}\right)^3\big\|w\big\|^2_{\ast},
\end{aligned}
$$
and there exists $c_2\leq C\delta$, such that
$$
\left(Cm_0\right)^2|\gamma|\left(1-\alpha^{-2}\right)^3\geq c_2^2|k|\nu\left(1-\alpha^{-2}\right)^3.
$$
Then we have
$$
\big\|(\mathcal{L}_{k_1,k_2}-i\lambda)w\big\|_{\ast}\geq c_2|k|^{1/2}\nu^{1/2}\left(1-\alpha^{-2}\right)^{3/2}\big\|w\big\|_{\ast}.
$$
This completes the proof of Lemma \ref{lem:3.2}.
\end{proof}

\vskip .05in

We obtain the semigroup estimate from pseudospectral bound, which based on the Gearhart-Pr\"{u}ss type lemma, see Lemma \ref{lem:2.1} and also \cite{GGN.2009,LWZ.2020} for details. The following result is a semigroup estimate of $\mathcal{L}$.

\begin{proposition}\label{prop:3.3}
Let $\delta>1$, if $0<\nu\ll1$, then there exists a constant $c>0$ independent of $\nu$, such that
\begin{equation}\label{eq:3.6}
\left\|e^{-t\mathcal{L}}P_{\neq}\right\|_{L^2\rightarrow L^2}\lesssim e^{-c\nu^{1/2}t-\nu t},
\end{equation}
where the operator $\mathcal{L}$ is defined in \eqref{eq:1.11} and $P_{\neq}$ is defined in \eqref{eq:1.7}. Moreover, for any $f\in L^2(\mathbb{T}^3)$ and $c'\in (0,c)$, there exists a constant $C(c,c')$ such that
\begin{equation}\label{eq:3.7}
\nu\int_{0}^{\infty}e^{2c'\nu^{1/2}s}\left\|(\nabla e^{-s\mathcal{L}}f)_{\neq}\right\|^2_{L^2}ds\leq C(c,c')\big\|f_{\neq}\big\|^2_{L^2},
\end{equation}
where $(\nabla e^{-t\mathcal{L}}f)_{\neq}$ is defined by \eqref{eq:1.7}.
\end{proposition}

In order to prove Proposition \ref{prop:3.3},  we consider the equation
\begin{equation}\label{eq:3.8}
\partial_t u+\nu(k_1^2+k_2^2)u+\mathcal{L}_{k_1,k_2}u=0,\ \ u(0,y)=u_0(y),
\end{equation}
and give the following two lemmas.

\begin{lemma}\label{lem:3.4}
Let $\delta>1$ and $u(t,y)$ be the solution of \eqref{eq:3.8}. If $0<\nu\ll1$, then for $(k_1,k_2)\neq {\bf 0}$, $\alpha^2>1$, one has
$$
\big\|u(t)\big\|_{L^2}\leq Ce^{-c\nu^{1/2}|k|^{1/2}t-\nu(k_1^2+k_2^2)t}\big\|u_0\big\|_{L^2}.
$$
\end{lemma}

\begin{proof}
By semigroup theory, one has
\begin{equation}\label{eq:3.9}
u(t,y)=e^{-\nu(k_1^2+k_2^2)t}e^{-t\mathcal{L}_{k_1,k_2}}u_0(y),
\end{equation}
thus
\begin{equation}\label{eq:3.10}
\big\|u(t)\big\|_{L^2} \leq e^{-\nu(k_1^2+k_2^2)t}\left\|e^{-t\mathcal{L}_{k_1,k_2}}u_0\right\|_{L^2}.
\end{equation}
Combining Lemma \ref{lem:2.1} and \ref{lem:3.2}, we obtain
\begin{equation}\label{eq:3.11}
\left\|e^{-t\mathcal{L}_{k_1,k_2}}u_0\right\|_{\ast}\leq e^{-c_2t(1-\alpha^{-2})^{3/2}|k|^{1/2}\nu^{1/2}+\pi/2}\big\|u_0\big\|_{\ast}.
\end{equation}
Thus we deduce by \eqref{eq:3.5} and \eqref{eq:3.11} that
\begin{equation}\label{eq:3.12}
\begin{aligned}
\left\|e^{-t\mathcal{L}_{k_1,k_2}}u_0\right\|_{L^2}&\leq(1-\alpha^{-2})^{-1/2}
\left\|e^{-t\mathcal{L}_{k_1,k_2}}u_0\right\|_{\ast}\\
&\leq (1-\alpha^{-2})^{-1/2}e^{-c_2(1-\alpha^{-2})^{3/2}|k|^{1/2}\nu^{1/2}t+\pi/2}\big\|u_0\big\|_{\ast}\\
&\leq(1-\alpha^{-2})^{-1/2}e^{-c_2(1-\alpha^{-2})^{3/2}|k|^{1/2}\nu^{1/2}t+\pi/2}\big\|u_0\big\|_{L^2}\\
&\leq Ce^{-c|k|^{1/2}\nu^{1/2}t}\big\|u_0\big\|_{L^2},
\end{aligned}
\end{equation}
where
$$
C=\left(1-\alpha^{-2}\right)^{-1/2}e^{\pi/2}, \ \ \ c=\min\left\{c_1, c_2(1-\alpha^{-2})^{3/2}\right\}.
$$
Combining \eqref{eq:3.10} and \eqref{eq:3.12}, one has
$$
\big\|u(t)\big\|_{L^2}\leq Ce^{-c\nu^{1/2}|k|^{1/2}t-\nu(k_1^2+k_2^2)t}\big\|u_0\big\|_{L^2}.
$$
This completes the proof of Lemma \ref{lem:3.4}.
\end{proof}

Similar with the works of Li et al.~(see \cite{LWZ.2020}, Lemma 4.1), the following lemma can be established.

\begin{lemma}\label{lem:3.5}
Let $\delta>1$ and $u(t,y)$ is the solution of \eqref{eq:3.8}. If $0<\nu\ll1$, then for $(k_1,k_2)\neq {\bf 0}$, $\alpha^2>1$, one has
$$
\nu m_0^2\int_{0}^{t}\left(\alpha^2\big\|u(s)\big\|^2_{L^2}+\big\|\partial_yu(s)\big\|^2_{L^2}\right)ds
\lesssim\big\|u_0\big\|^2_{L^2},
$$
and
$$
\nu m_0^2\int_{0}^{t}e^{2as}\left(\alpha^2\big\|u(s)\big\|^2_{L^2}+\big\|\partial_yu(s)\big\|^2_{L^2}\right)ds
\lesssim\left(1+at\right)\big\|u_0\big\|^2_{L^2},
$$
where $a=c\nu^{1/2}|k|^{1/2}+\nu(k_1^2+k_2^2)$. Moreover, for any $c'\in (0,c)$, there exists a constant $C(c,c')>0$ such that
$$
\nu m_0^2\int_{0}^{t}e^{2c'\nu^{1/2}|k|^{1/2}s}\left(\alpha^2\big\|u(s)\big\|^2_{L^2}+\big\|\partial_yu(s)\big\|^2_{L^2}\right)ds
\leq C(c,c')\big\|u_0\big\|^2_{L^2}.
$$
\end{lemma}

Now, we give the proof of Proposition \ref{prop:3.3}.

\begin{proof}[The proof of Proposition {\rm \ref{prop:3.3}}]
Using the same discussion of Lemma \ref{lem:2.3} and combining with Lemma \ref{lem:3.4}, we can easily get the \eqref{eq:3.6}. Next, we prove the \eqref{eq:3.7}. Since
$$
e^{-t\mathcal{L}}f=\sum_{k_1,k_2\in \mathbb{Z}}e^{-\nu(k_1^2+k_2^2)t}e^{-t\mathcal{L}_{k_1,k_2}}\hat{f}_1(k_1,k_2,\cdot)(m_0 y)e^{i(k_1x_1+k_2x_2)},
$$
$$
(\partial_{x_1} e^{-t\mathcal{L}}f)_{\neq}=\sum_{(k_1,k_2)\neq {\bf 0}}ik_1e^{-\nu(k_1^2+k_2^2)t}e^{-t\mathcal{L}_{k_1,k_2}}\hat{f}_1(k_1,k_2,\cdot)(m_0 y)e^{i(k_1x_1+k_2x_2)},
$$
$$
(\partial_{x_2} e^{-t\mathcal{L}}f)_{\neq}=\sum_{(k_1,k_2)\neq {\bf 0}}ik_2e^{-\nu(k_1^2+k_2^2)t}e^{-t\mathcal{L}_{k_1,k_2}}\hat{f}_1(k_1,k_2,\cdot)(m_0 y)e^{i(k_1x_1+k_2x_2)},
$$
and
$$
(\partial_{y} e^{-t\mathcal{L}}f)_{\neq}=\sum_{(k_1,k_2)\neq {\bf 0}}m_0\partial_{y} e^{-\nu(k_1^2+k_2^2)t}e^{-t\mathcal{L}_{k_1,k_2}}\hat{f}_1(k_1,k_2,\cdot)(m_0 y)e^{i(k_1x_1+k_2x_2)},
$$
where $f_1(x_1,x_2,y)=f(x_1,x_2,\delta y)$. Then one has
\begin{equation}\label{eq:3.13}
\begin{aligned}
\left\|(\nabla e^{-t\mathcal{L}}f)_{\neq}\right\|^2_{L^2}&\leq4\pi^2\delta\sum_{(k_1,k_2)\neq {\bf 0}}(k^2_1+k^2_2)\left\|e^{-\nu(k_1^2+k_2^2)t}e^{-t\mathcal{L}_{k_1,k_2}}\hat{f}_1\right\|^2_{L^2}\\
&\ \ \ +4\pi^2m_0\sum_{(k_1,k_2)\neq {\bf 0}}\left\|\partial_ye^{-\nu(k_1^2+k_2^2)t}e^{-t\mathcal{L}_{k_1,k_2}}\hat{f}_1\right\|^2_{L^2}\\
&=4\pi^2m_0\sum_{(k_1,k_2)\neq {\bf 0}}(k^2_1+k^2_2)\delta^2\left\|e^{-\nu(k_1^2+k_2^2)t}e^{-t\mathcal{L}_{k_1,k_2}}\hat{f}_1\right\|^2_{L^2}\\
&\ \ \ +4\pi^2m_0\sum_{(k_1,k_2)\neq {\bf 0}}\left\|\partial_ye^{-\nu(k_1^2+k_2^2)t}e^{-t\mathcal{L}_{k_1,k_2}}\hat{f}_1\right\|^2_{L^2}\\
&=4\pi^2m_0\sum_{(k_1,k_2)\neq {\bf 0}}\alpha^2\left\|e^{-\nu(k_1^2+k_2^2)t}e^{-t\mathcal{L}_{k_1,k_2}}\hat{f}_1\right\|^2_{L^2}\\
&\ \ \ +4\pi^2m_0\sum_{(k_1,k_2)\neq {\bf 0}}\left\|\partial_ye^{-\nu(k_1^2+k_2^2)t}e^{-t\mathcal{L}_{k_1,k_2}}\hat{f}_1\right\|^2_{L^2}.
\end{aligned}
\end{equation}
Considering the equation
$$
\partial_t u+\nu(k_1^2+k_2^2)u+\mathcal{L}_{k_1,k_2}u=0, \ \ \ u(0,y)=\hat{f}_1(k_1,k_2,y),
$$
we deduce by \eqref{eq:3.9} and \eqref{eq:3.13} that
$$
\left\|(\nabla e^{-t\mathcal{L}}f)_{\neq}\right\|^2_{L^2}\leq4\pi^2m_0\sum_{(k_1,k_2)\neq {\bf 0}}\alpha^2\big\|u\big\|^2_{L^2}
+4\pi^2m_0\sum_{(k_1,k_2)\neq {\bf 0}}\big\|\partial_yu\big\|^2_{L^2}.
$$
By Lemma \ref{lem:3.5}, we have
$$
\begin{aligned}
&\nu\int_{0}^{\infty}e^{2c'\nu^{1/2}s}\left\|(\nabla e^{-s\mathcal{L}}f)_{\neq}\right\|^2_{L^2}ds\\
\leq&4\pi^2\nu m_0\sum_{(k_1,k_2)\neq {\bf 0}}\int_{0}^{\infty}e^{2c'\nu^{1/2}s}\left(\alpha^2\big\|u\big\|^2_{L^2}+\big\|\partial_yu\big\|^2_{L^2}\right)ds\\
=&4\pi^2\delta\sum_{(k_1,k_2)\neq {\bf 0}}\nu m^2_0\int_{0}^{\infty}e^{2c'\nu^{1/2}s}\left(\alpha^2\big\|u\big\|^2_{L^2}+\big\|\partial_yu\big\|^2_{L^2}\right)ds\\
\leq& 4\pi^2\delta\sum_{(k_1,k_2)\neq {\bf 0}} C(c,c')\big\|\hat{f}_1(k_1,k_2,\cdot)\big\|^2_{L^2}\\
\leq&  C(c,c')4\pi^2\delta\sum_{(k_1,k_2)\neq{\bf 0}}\big\|\hat{f}_1(k_1,k_2,\cdot)\big\|^2_{L^2}\\
=& C(c,c')\big\|f_{\neq}\big\|^2_{L^2}.
\end{aligned}
$$
This completes the proof of Proposition \ref{prop:3.3}.
\end{proof}

\begin{remark}
Similar argument with Proposition {\rm \ref{prop:3.3}} and Lemma {\rm \ref{lem:3.4}}, we can easily get that for $i=1,2$, one has
$$
\big\|\partial_{x_i}u(t)\big\|_{L^2}\leq Ce^{-c\nu^{1/2}|k|^{1/2}t-\nu(k_1^2+k_2^2)t}\big\|\partial_{x_i}u_0\big\|_{L^2},
$$
and
$$
\nu\int_{0}^{\infty}e^{2c'\nu^{1/2}s}\left\|(\nabla e^{-s\mathcal{L}}\partial_{x_i}f)_{\neq}\right\|^2_{L^2}ds\leq C(c,c')\big\|\partial_{x_i}f_{\neq}\big\|^2_{L^2}.
$$
\end{remark}

\subsection{Enhanced dissipation decay of linearized equation}
Here we show that the enhanced dissipation of linearized Navier-Stokes equation \eqref{eq:1.10}. we can easily obtain that the enhanced dissipation of $\Delta v_3$ to \eqref{eq:1.10} by Lemma \ref{lem:2.3}. Next, we consider the enhanced dissipation decay of $\omega_3$. Taking the Fourier transform in $x_1,x_2$, changing $y$ to $\delta y$, and we denote
$$
u(t,y)=\frac{1}{4\pi^2}\int_{\mathbb{T}^2}\Delta v_3(t,x_1,x_2,\delta y)e^{-ik_1x_1-ik_2x_2}dx_1dx_2,
$$
and
$$
w(t,y)=\frac{1}{4\pi^2}\int_{\mathbb{T}^2}\omega_3(t,x_1,x_2,\delta y)e^{-ik_1x_1-ik_2x_2}dx_1dx_2.
$$
Then the equation \eqref{eq:1.10} becomes
\begin{equation}\label{eq:3.14}
\begin{cases}
\partial_t u+\left(\nu(k_1^2+k_2^2)+\mathcal{L}_{k_1,k_2}  \right)u=0,\\
\partial_t w+\left(\nu(k_1^2+k_2^2)+\mathcal{H}_{k_1,k_2}  \right)w=-i|k|\delta^3\sin (y+\alpha_k)(\alpha^2-\partial_y^2)^{-1}u,\\
u(0,y)=u_0(y),\ \ \ w(0,y)=w_0(y),
\end{cases}
\end{equation}
where $\alpha^2=(k_1^2+k_2^2)\delta^2$, the $\mathcal{L}_{k_1,k_2}$ and $\mathcal{H}_{k_1,k_2}$ are defined in \eqref{eq:3.3} and \eqref{eq:3.4}. The following result proves the enhanced dissipation decay of $w$ in \eqref{eq:3.14}.

\begin{proposition}\label{prop:3.6}
Let $\delta>1$ and $(u,w)$ are the solutions of \eqref{eq:3.14}. If $0<\nu\ll1$, then for $(k_1,k_2)\neq {\bf 0}$, $\alpha^2>1$, one has
$$
\big\|w(t)\big\|_{L^2}\lesssim e^{-at}\left(\big\|w_0\big\|_{L^2}+\big\|u_0\big\|_{L^2}\right),
$$
where $a=\nu(k_1^2+k_2^2)+c|k|^{1/2}\nu^{1/2}$.
\end{proposition}

\begin{proof}
Using the definition of $\mathcal{L}_{k_1,k_2}$ in \eqref{eq:3.3}, the first equation of \eqref{eq:3.14} is written as
$$
\begin{aligned}
&\partial_t u+\left(\nu(k_1^2+k_2^2)+\mathcal{L}_{k_1,k_2}  \right)u\\
=&\partial_t u+\left(\nu(k_1^2+k_2^2)-\nu m_0^2\partial^2_{y}+i|k|\delta^2\sin (y+\alpha_k)\left(1-(\alpha^2-\partial^2_{y})^{-1}\right) \right)u\\
=&\partial_t u+\left(\nu(k_1^2+k_2^2)-\nu m^2_0\partial^2_{y}+i|k|\delta^2\sin (y+\alpha_k)\right)u
-i|k|\delta^2\sin (y+\alpha_k)(\alpha^2-\partial^2_{y})^{-1}u\\
=& \partial_t u+\left(\nu(k_1^2+k_2^2)+\mathcal{H}_{k_1,k_2}\right)u
-i|k|\delta^2\sin (y+\alpha_k)(\alpha^2-\partial^2_{y})^{-1}u=0,
\end{aligned}
$$
and from the second equation of \eqref{eq:3.14}, one has that
$$
\partial_t m_0w+\left(\nu(k_1^2+k_2^2)+\mathcal{H}_{k_1,k_2}\right)m_0w=-i|k|\delta^2\sin (y+\alpha_k)(\alpha^2-\partial_y^2)^{-1}u,
$$
thus we have
$$
\partial_t( m_0w+u)+\left(\nu(k_1^2+k_2^2)+\mathcal{H}_{k_1,k_2}\right)( m_0w+u)=0.
$$
Then we deduce by Lemma \ref{lem:2.1} and \ref{lem:2.3} that
\begin{equation}\label{eq:3.15}
\big\|u+m_0w\big\|_{L^2}\lesssim e^{-at}\left(m_0\big\|w_0\big\|_{L^2}+\big\|u_0\big\|_{L^2}\right),
\end{equation}
and combining Lemma \ref{lem:3.4} and \eqref{eq:3.15}, to obtain
$$
\begin{aligned}
\big\|w\big\|_{L^2}&\lesssim m_0^{-1}e^{-at}\left(m_0\big\|w_0\big\|_{L^2}+\big\|u_0\big\|_{L^2}\right)+m_0^{-1}e^{-at}\big\|u_0\big\|_{L^2}\\
&\lesssim e^{-at}\left(\big\|w_0\big\|_{L^2}+\big\|u_0\big\|_{L^2}\right).
\end{aligned}
$$
This completes the proof of Proposition \ref{prop:3.6}.
\end{proof}

\begin{remark} Obviously, for the $\omega_3$ in \eqref{eq:1.10}, we deduce by Proposition {\rm \ref{prop:3.6}} that
$$
\big\|(\omega_3)_{\neq}\big\|_{L^2}\lesssim e^{-c\nu^{1/2}t-\nu t}\left(\big\|\omega_3(0)\big\|_{L^2}+\big\|\Delta v_3(0)\big\|_{L^2}\right).
$$
And similar argument with Proposition {\rm \ref{prop:3.6}}, we can obtain that for $i=1,2$, one has
$$
\big\|\partial_{x_i}\omega_3\big\|_{L^2}\lesssim e^{-c\nu^{1/2}t-\nu t}\left(\big\|\partial_{x_i}\omega_3(0)\big\|_{L^2}+\big\|\partial_{x_i}\Delta v_3(0)\big\|_{L^2}\right).
$$
\end{remark}

\vskip .05in

\section{Nonlinear stability threshold}\label{sec.4}

In this section, we study the nonlinear stability of \eqref{eq:1.3}, and mainly consider the \eqref{eq:1.5}. The \eqref{eq:1.5} can be written as
\begin{equation}\label{eq:4.1}
\begin{cases}
(\partial_t+\mathcal{L})\Delta v_3=\div f,\\
(\partial_t+\mathcal{H})\omega_3=\delta\sin(m_0 y)\partial_{x_1}v_3+\delta\cos(m_0 y)\partial_{x_2}v_3+\div g,
\end{cases}
\end{equation}
where
\begin{equation}\label{eq:4.2}
f=-\nabla(V\cdot\nabla v_3)-(0,0,\Delta p),
\end{equation}
and
\begin{equation}\label{eq:4.3}
g=-V\omega_3+\omega v_3=-\omega_3(v_1,v_2,v_3)+v_3(\omega_1,\omega_2,\omega_3),
\end{equation}
and the $\mathcal{L}$ and $\mathcal{H}$ are defined in \eqref{eq:1.11} and \eqref{eq:1.12}, the $p$ and $\omega$ are defined in \eqref{eq:1.6}.

\vskip .05in

\subsection{The estimates of velocity and vorticity}
Here we consider the estimates of velocity $\Delta v_3$ and vorticity $\omega_3$. We establish the time-space estimate for non-zero mode parts, which include enhanced dissipation decay. Considering the first equation in \eqref{eq:4.1}, one has the following estimates of velocity.

\begin{proposition}\label{prop:4.1}
Let $\Delta v_3$ be the solution of \eqref{eq:4.1}, then
$$
\big\|\partial_{x_1}(\Delta v_3)\big\|^2_{X_{ed}}
\lesssim \big\|\partial_{x_1}(\Delta v_3)(0)\big\|^2_{ L^2}+\nu^{-1}\left\|e^{\epsilon\nu^{1/2}t}\partial_{x_1}f\right\|^2_{L^2 L^2},
$$
and
$$
\big\|\partial_{x_2}(\Delta v_3)\big\|^2_{X_{ed}}
\lesssim \big\|\partial_{x_2}(\Delta v_3)(0)\big\|^2_{ L^2}+\nu^{-1}\left\|e^{\epsilon\nu^{1/2}t}\partial_{x_2}f\right\|^2_{L^2 L^2},
$$
where $\|\cdot\|_{X_{ed}}$ is defined in \eqref{eq:1.15}.
\end{proposition}

\begin{proof}
In order to estimate $\partial_{x_1}(\Delta v_3)$, we apply $\partial_{x_1}$ to first equation of \eqref{eq:4.1} and decompose resulting equation into
\begin{equation}\label{eq:4.4}
\begin{cases}
(\partial_t+\mathcal{L})\partial_{x_1}(\Delta v^{NL}_3)=\div \partial_{x_1}f,\\
(\Delta v^{NL}_3)(0)=0,
\end{cases}
\end{equation}
and
\begin{equation}\label{eq:4.5}
\begin{cases}
(\partial_t+\mathcal{L})\partial_{x_1}(\Delta v^{L}_3)=0,\\
(\Delta v^{L}_3)(0)=(\Delta v_3)(0),
\end{cases}
\end{equation}
where $v_3=v_3^{L}+v_3^{NL}$. We estimate $\partial_{x_1}(\Delta v^{NL}_3)$ and $\partial_{x_1}(\Delta v^{L}_3)$ separately, the detail is as follows

\vskip .05in

\noindent {\bf Step 1. The basic estimate of $\partial_{x_1}(\Delta v^{NL}_3)$.}
Considering the equation \eqref{eq:4.4}, one gets
\begin{equation}\label{eq:4.6}
\begin{aligned}
&\frac{1}{2}\frac{d}{dt}\left\|\partial_{x_1}(\Delta v^{NL}_3)\right\|^2_{L^2}+\nu\left\|\nabla\partial_{x_1}(\Delta v^{NL}_3)\right\|^2_{L^2}\\
& \ +\Re\left\langle\sin (m_0 y)\partial^2_{x_1}(v^{NL}_3),\partial_{x_1}(\Delta v^{NL}_3)\right\rangle+\Re\left\langle\cos (m_0 y)\partial_{x_2}\partial_{x_1}(v^{NL}_3),\partial_{x_1}(\Delta v^{NL}_3)\right\rangle\\
=&\Re\left\langle \div \partial_{x_1}f, \partial_{x_1}(\Delta v^{NL}_3)\right\rangle,
\end{aligned}
\end{equation}
and
\begin{equation}\label{eq:4.7}
\begin{aligned}
&\frac{1}{2}\frac{d}{dt}\left\|\partial_{x_1}(\nabla v^{NL}_3)\right\|^2_{L^2}+\nu\left\|\partial_{x_1}(\Delta v^{NL}_3)\right\|^2_{L^2}\\
&\ -\Re\left\langle\delta^2\sin (m_0 y)\partial^2_{x_1}(\Delta v^{NL}_3),\partial_{x_1}(v^{NL}_3)\right\rangle-\Re\left\langle\delta^2\cos (m_0 y)\partial_{x_2}\partial_{x_1}(\Delta v^{NL}_3),\partial_{x_1}(v^{NL}_3)\right\rangle\!\!\!\!\\
=&\Re\left\langle \div \partial_{x_1}f, -\partial_{x_1}(v^{NL}_3)\right\rangle.
\end{aligned}
\end{equation}
We deduce from \eqref{eq:4.7} that
\begin{equation}\label{eq:4.8}
\begin{aligned}
&-\frac{1}{2}\frac{d}{dt}\left(m_0^2\left\|\partial_{x_1}(\nabla v^{NL}_3)\right\|^2_{L^2}\right)-\nu m_0^2\left\|\partial_{x_1}(\Delta v^{NL}_3)\right\|^2_{L^2}\\
&+\Re\left\langle\sin (m_0 y)\partial^2_{x_1}(\Delta v^{NL}_3),\partial_{x_1}(v^{NL}_3)\right\rangle+\Re\left\langle\cos (m_0 y)\partial_{x_2}\partial_{x_1}(\Delta v^{NL}_3),\partial_{x_1}(v^{NL}_3)\right\rangle\\
=&-m_0^2\Re\left\langle \div \partial_{x_1}f, -\partial_{x_1}(v^{NL}_3)\right\rangle.
\end{aligned}
\end{equation}
Since
\begin{equation}\label{eq:4.9}
\left\|\partial_{x_1}(\Delta v^{NL}_3)\right\|^2_{L^2}\geq \left\|\partial_{x_1}(\nabla v^{NL}_3)\right\|^2_{L^2},\ \ \
\left\|\partial_{x_1}(\nabla\Delta v^{NL}_3)\right\|^2_{L^2}\geq \left\|\partial_{x_1}(\Delta v^{NL}_3)\right\|^2_{L^2},
\end{equation}
\vskip .01in
$$
\left\langle\sin (m_0 y)\partial^2_{x_1}(v^{NL}_3),\partial_{x_1}(\Delta v^{NL}_3)\right\rangle
=-\left\langle\sin (m_0 y)\partial^2_{x_1}(\Delta v^{NL}_3),\partial_{x_1}(v^{NL}_3)\right\rangle,
$$
and
$$
\left\langle\cos (m_0 y)\partial_{x_2}\partial_{x_1}(v^{NL}_3),\partial_{x_1}(\Delta v^{NL}_3)\right\rangle
=-\left\langle\cos (m_0 y)\partial_{x_2}\partial_{x_1}(\Delta v^{NL}_3),\partial_{x_1}(v^{NL}_3)\right\rangle.
$$
Then we obtain by \eqref{eq:4.6} and \eqref{eq:4.8} that
\begin{equation}\label{eq:4.10}
\begin{aligned}
&\frac{1}{2}\frac{d}{dt}\left(\left\|\partial_{x_1}(\Delta v^{NL}_3)\right\|^2_{L^2}-m_0^2\left\|\partial_{x_1}(\nabla v^{NL}_3)\right\|^2_{L^2} \right)\\
&\ \ \ \ \ \ \ +\frac{\nu}{2}\left(1-m_0^2\right)\left\|\partial_{x_1}(\nabla\Delta v^{NL}_3)\right\|^2_{L^2}\lesssim\nu^{-1}\big\|\partial_{x_1}f\big\|^2_{L^2},
\end{aligned}
\end{equation}
where we use
$$
\begin{aligned}
&\Re\left\langle \div \partial_{x_1}f, \partial_{x_1}(\Delta v^{NL}_3)\right\rangle-m_0^2\Re\left\langle \div \partial_{x_1}f, -\partial_{x_1}(v^{NL}_3)\right\rangle\\
=&-\Re\left\langle \partial_{x_1}f, \partial_{x_1}(\nabla\Delta v^{NL}_3)\right\rangle-m_0^2\Re\left\langle \partial_{x_1}f, \partial_{x_1}(\nabla v^{NL}_3)\right\rangle\\
\leq& C\nu^{-1}(1-m_0^2)^{-1}\left\|\partial_{x_1}f\right\|^2_{L^2}+\frac{\nu(1-m_0^2)}{2}\left\|\partial_{x_1}(\nabla\Delta v^{NL}_3)\right\|^2_{L^2}.
\end{aligned}
$$
Combining \eqref{eq:4.9}, \eqref{eq:4.10} and $(1-m_0^2)^{-1}\lesssim 1$, one gets
\begin{equation}\label{eq:4.11}
 \left\|\partial_{x_1}(\Delta v^{NL}_3)(t)\right\|^2_{L^2}+\nu\int_{0}^{t}\left\|\partial_{x_1}(\nabla\Delta v^{NL}_3)(s)\right\|^2_{L^2}ds \lesssim\nu^{-1}\int_{0}^{t}\big\|\partial_{x_1}f(s)\big\|^2_{L^2}ds.
\end{equation}
Therefore, we obtain by \eqref{eq:4.4} and \eqref{eq:4.11} that
\begin{equation}\label{eq:4.12}
\left\|\int_{0}^{t}e^{-(t-s)\mathcal{L}}\div \partial_{x_1}f(s)ds\right\|^2_{L^2}=\left\|\partial_{x_1}(\Delta v^{NL}_3)(t)\right\|^2_{L^2}
\lesssim\nu^{-1}\int_{0}^{t}\big\|\partial_{x_1}f(s)\big\|^2_{L^2}ds.
\end{equation}

\noindent {\bf Step 2. The estimate of $\big\|\partial_{x_1}(\Delta v^{NL}_3)\big\|_{X_{ed}}$.}  Taking $\tau_0=\nu^{-1/2}$,  for any $k\in \mathbb{N}^{+}$, we define
$$
I_k=\nu^{-1}\int_{(k-1)\tau_0}^{k\tau_0}\big\|\partial_{x_1}f(s)\big\|^2_{L^2}ds.
$$
By \eqref{eq:4.12}, we know that for any $0\leq t<\tau_0$, one has
\begin{equation}\label{eq:4.13}
\left\|\partial_{x_1}(\Delta v^{NL}_3)(t)\right\|^2_{L^2}\lesssim\nu^{-1}\int_{0}^{\tau_0}\big\|\partial_{x_1}f(s)\big\|^2_{L^2}ds=I_1.
\end{equation}
For any $l\tau_0\leq t<(l+1)\tau_0$, $l\geq 1$ is an integer, one has
$$
\begin{aligned}
\partial_{x_1}(\Delta v^{NL}_3)(t)&=\int_{0}^{t}e^{-(t-s)\mathcal{L}}\div \partial_{x_1}f(s)ds\\
&=\sum_{k=1}^{l}\int_{(k-1)\tau_0}^{k\tau_0}e^{-(t-s)\mathcal{L}}\div \partial_{x_1}f(s)ds
+\int_{l\tau_0}^{t}e^{-(t-s)\mathcal{L}}\div \partial_{x_1}f(s)ds.
\end{aligned}
$$
Combining Proposition \ref{prop:3.3} and similar argument with \eqref{eq:4.12}, for $k=1,2,\cdots, l$, one gets
$$
\begin{aligned}
&\left\|\int_{(k-1)\tau_0}^{k\tau_0}e^{-(t-s)\mathcal{L}}\div \partial_{x_1}f(s)ds\right\|^2_{L^2}\\
=&\left\|e^{-(t-k\tau_0)\mathcal{L}}\int_{(k-1)\tau_0}^{k\tau_0}e^{-(k\tau_0-s)\mathcal{L}}\div \partial_{x_1}f(s)ds\right\|^2_{L^2}\\
\lesssim& ~e^{-2c\nu^{1/2}(t-k\tau_0)}I_k,
\end{aligned}
$$
and
$$
\left\|\int_{l\tau_0}^{t}e^{-(t-s)\mathcal{L}}\div \partial_{x_1}f(s)ds\right\|^2_{L^2}
\lesssim\nu^{-1}\int_{l\tau_0}^{(l+1)\tau_0}\big\|\partial_{x_1}f(s)\big\|^2_{L^2}ds=I_{l+1}.
$$
Then for any $l\tau_0\leq t<(l+1)\tau_0$, we obtain
\begin{equation}\label{eq:4.14}
\big\|\partial_{x_1}(\Delta v^{NL}_3)(t)\big\|^2_{L^2}
\lesssim \sum_{k=1}^{l}e^{-2c\nu^{1/2}(t-k\tau_0)}I_k+I_{l+1}\lesssim\sum_{k=1}^{l+1}e^{-2c(l+1-k)}I_{k},
\end{equation}
and combining \eqref{eq:4.13}, we know that for any $l\tau_0\leq t<(l+1)\tau_0$, $l\geq 0$ is an integer, the \eqref{eq:4.14} can also be established. Since $\tau_0=\nu^{-1/2}$, then for $\epsilon\in (0,c)$, we can easily get
$$
\begin{aligned}
\sum_{k=1}^{\infty}e^{2\epsilon k}I_k
&=\sum_{k=1}^{\infty}e^{2\epsilon k}\nu^{-1}\int_{(k-1)\tau_0}^{k\tau_0}\big\|\partial_{x_1}f(s)\big\|^2_{L^2}ds\\
&\lesssim\sum_{k=1}^{\infty}\nu^{-1}\int_{(k-1)\tau_0}^{k\tau_0}e^{2\epsilon\nu^{1/2}s}
\big\|\partial_{x_1}f(s)\big\|^2_{L^2}ds\\
&=\nu^{-1}\int_{0}^{\infty}e^{2\epsilon\nu^{1/2}s}\big\|\partial_{x_1}f(s)\big\|^2_{L^2}ds.
\end{aligned}
$$
Then we deduce by \eqref{eq:4.14} that
$$
\begin{aligned}
\|\partial_{x_1}(\Delta v^{NL}_3)(t)\|_{L^2}
&\lesssim\sum_{k=1}^{l+1}e^{-c(l+1-k)}I^{1/2}_k\\
&\lesssim \left(\sum_{k=1}^{l+1}e^{-(c+\epsilon)(l+1-k)}I_k\right)^{1/2}\left(\sum_{k=1}^{l+1}e^{-(c-\epsilon)(l+1-k)}\right)^{1/2}\\
&\lesssim \left(\sum_{k=1}^{l+1}e^{-(c+\epsilon)(l+1-k)}I_k\right)^{1/2}.
\end{aligned}
$$
Therefore, we have
\begin{equation}\label{eq:4.15}
\begin{aligned}
&\nu^{1/2}\int_{0}^{\infty}e^{2\epsilon\nu^{1/2}t}\left\|\partial_{x_1}(\Delta v^{NL}_3)(t)\right\|^2_{L^2}dt\\
\lesssim&~\nu^{1/2}\sum_{l=0}^{\infty}e^{2\epsilon(l+1)}\int_{l\tau_0}^{(l+1)\tau_0}\left\|\partial_{x_1}(\Delta v^{NL}_3)(t)\right\|^2_{L^2}dt\\
\lesssim&~\nu^{1/2}\sum_{l=0}^{\infty}e^{2\epsilon(l+1)}\tau_0\sum_{k=1}^{l+2}e^{-(c+\epsilon)(l+2-k)}I_k\\
=&\sum_{k=1}^{\infty}e^{2\epsilon k}I_k\sum_{l=k-2}^{\infty}e^{-(c-\epsilon)(l+2-k)}\\
\lesssim & \sum_{k=1}^{\infty}e^{2\epsilon k}I_k\lesssim \nu^{-1}\left\|e^{\epsilon\nu^{1/2}t}\partial_{x_1}f\right\|^2_{L^2L^2}.
\end{aligned}
\end{equation}
We get by \eqref{eq:4.10} that
$$
\begin{aligned}
&\frac{1}{2}\frac{d}{dt}\left(e^{2\epsilon\nu^{1/2}t}\big(\left\|\partial_{x_1}(\Delta v^{NL}_3)\right\|^2_{L^2}-m_0^2\left\|\partial_{x_1}(\nabla v^{NL}_3)\right\|^2_{L^2} \big)\right)\\
&+\frac{\nu}{2}\big(1-m_0^2\big)e^{2\epsilon\nu^{1/2}t}\left\|\partial_{x_1}(\nabla\Delta v^{NL}_3)\right\|^2_{L^2}\\
\lesssim& ~\nu^{-1}e^{2\epsilon\nu^{1/2}t}\big\|\partial_{x_1}f\big\|^2_{L^2}+\nu^{1/2}e^{2\epsilon\nu^{1/2}t}
\left(\left\|\partial_{x_1}(\Delta v^{NL}_3)\right\|^2_{L^2}-m_0^2\left\|\partial_{x_1}(\nabla v^{NL}_3)\right\|^2_{L^2} \right),
\end{aligned}
$$
then combining \eqref{eq:4.4}, \eqref{eq:4.9} and \eqref{eq:4.15}, we have
\begin{equation}\label{eq:4.16}
\begin{aligned}
&e^{2\epsilon\nu^{1/2}t}\big(1-m_0^2\big)\left\|\partial_{x_1}(\Delta v^{NL}_3)\right\|^2_{L^2}+\nu\big(1-m_0^2\big)\int_{0}^{t}e^{2\epsilon\nu^{1/2}s}\left\|\partial_{x_1}(\nabla\Delta v^{NL}_3)\right\|^2_{L^2}ds\\
\lesssim& \nu^{-1}\int_{0}^{t}e^{2\epsilon\nu^{1/2}s}\big\|\partial_{x_1}f\big\|^2_{L^2}ds
+\nu^{1/2}\int_{0}^{t}e^{2\epsilon\nu^{1/2}s}\big\|\partial_{x_1}(\Delta v^{NL}_3)\big\|^2_{L^2}ds\\
\lesssim& \nu^{-1}\int_{0}^{t}e^{2\epsilon\nu^{1/2}s}\big\|\partial_{x_1}f\big\|^2_{L^2}ds.
\end{aligned}
\end{equation}
Thus, we deduce by \eqref{eq:4.15} and \eqref{eq:4.16} that
\begin{equation}\label{eq:4.17}
\begin{aligned}
\big\|\partial_{x_1}(\Delta v^{NL}_3)\big\|^2_{X_{ed}}=&\left\|e^{\epsilon\nu^{1/2}t}\partial_{x_1}(\Delta v^{NL}_3)\right\|^2_{L^\infty L^2}+\nu^{1/2}\left\|e^{\epsilon\nu^{1/2}t}\partial_{x_1}(\Delta v^{NL}_3)\right\|^2_{L^2 L^2}\\
&+\nu\left\|e^{\epsilon\nu^{1/2}t}\partial_{x_1}(\nabla\Delta v^{NL}_3)\right\|^2_{L^2 L^2}\\
\lesssim &\nu^{-1}\left\|e^{\epsilon\nu^{1/2}t}\partial_{x_1}f\right\|^2_{L^2 L^2}.
\end{aligned}
\end{equation}

\noindent {\bf Step 3. The estimate of $\left\|\partial_{x_1}(\Delta v_3)\right\|_{X_{ed}}$.} First, we consider linear equation \eqref{eq:4.5}, the solution $\partial_{x_1}(\Delta v^{L}_3)(t)$ can be written as
$$
\partial_{x_1}(\Delta v^{L}_3)(t)=e^{-t\mathcal{L}}\partial_{x_1}(\Delta v_3)(0).
$$
Then we deduce by Proposition \ref{prop:3.3} that
\begin{equation}\label{eq:4.18}
\begin{aligned}
\big\|\partial_{x_1}(\Delta v^{L}_3)\big\|^2_{X_{ed}}=&\left\|e^{\epsilon\nu^{1/2}t}\partial_{x_1}(\Delta v^{L}_3)\right\|^2_{L^\infty L^2}+\nu^{1/2}\left\|e^{\epsilon\nu^{1/2}t}\partial_{x_1}(\Delta v^{L}_3)\right\|^2_{L^2 L^2}\\
&+\nu\left\|e^{\epsilon\nu^{1/2}t}\partial_{x_1}(\nabla\Delta v^{L}_3)\right\|^2_{L^2 L^2}\\
\lesssim &\big\|\partial_{x_1}(\Delta v^{L}_3)(0)\big\|^2_{ L^2}.
\end{aligned}
\end{equation}
Since
$$
\partial_{x_1}(\Delta v_3)=\partial_{x_1}(\Delta v^{NL}_3)+\partial_{x_1}(\Delta v^{L}_3),
$$
then combining \eqref{eq:4.17} and \eqref{eq:4.18}, we obtain
$$
\begin{aligned}
\big\|\partial_{x_1}(\Delta v_3)\big\|^2_{X_{ed}}=&\left\|e^{\epsilon\nu^{1/2}t}\partial_{x_1}(\Delta v_3)\right\|^2_{L^\infty L^2}+\nu^{1/2}\left\|e^{\epsilon\nu^{1/2}t}\partial_{x_1}(\Delta v_3)\right\|^2_{L^2 L^2}\\
&+\nu\left\|e^{\epsilon\nu^{1/2}t}\partial_{x_1}(\nabla\Delta v_3)\right\|^2_{L^2 L^2}\\
\lesssim &\big\|\partial_{x_1}(\Delta v_3)(0)\big\|^2_{ L^2}+\nu^{-1}\left\|e^{\epsilon\nu^{1/2}t}\partial_{x_1}f\right\|^2_{L^2 L^2}.
\end{aligned}
$$
Similarly, one has
$$
\begin{aligned}
\big\|\partial_{x_2}(\Delta v_3)\big\|^2_{X_{ed}}
\lesssim \big\|\partial_{x_2}(\Delta v_3)(0)\big\|^2_{ L^2}+\nu^{-1}\left\|e^{\epsilon\nu^{1/2}t}\partial_{x_2}f\right\|^2_{L^2 L^2}.
\end{aligned}
$$
This completes the proof of Proposition \ref{prop:4.1}.
\end{proof}

\vskip .05in

For the second equation in \eqref{eq:4.1}, one has the following estimate of vorticity.
\begin{proposition}\label{prop:4.2}
Let $\omega_3$ be the solution of \eqref{eq:4.1}, then
$$
\begin{aligned}
\big\|\partial_{x_1}\omega_3\big\|^2_{X_{ed}}
\lesssim &\big\|\partial_{x_1}(\Delta v_3)(0)\big\|^2_{L^2}+\big\|\partial_{x_1}\omega_3(0)\big\|^2_{L^2}\\
&+\nu^{-1}\left(\left\|e^{\epsilon\nu^{1/2}t}\partial_{x_1}f\right\|^2_{L^2L^2}
+\left\|e^{\epsilon\nu^{1/2}t}\partial_{x_1}g\right\|^2_{L^2L^2}\right),
\end{aligned}
$$
and
$$
\begin{aligned}
\big\|\partial_{x_2}\omega_3\big\|^2_{X_{ed}}
\lesssim &\big\|\partial_{x_2}(\Delta v_3)(0)\big\|^2_{L^2}+\big\|\partial_{x_2}\omega_3(0)\big\|^2_{L^2}\\
&+\nu^{-1}\left(\left\|e^{\epsilon\nu^{1/2}t}\partial_{x_2}f\right\|^2_{L^2L^2}
+\left\|e^{\epsilon\nu^{1/2}t}\partial_{x_2}g\right\|^2_{L^2L^2}\right),
\end{aligned}
$$
where $\|\cdot\|_{X_{ed}}$ is defined in \eqref{eq:1.15}.
\end{proposition}

\begin{proof}
Applying the operator $\partial_{x_1}$ to \eqref{eq:4.1}, we obtain
$$
(\partial_t+\mathcal{L})\partial_{x_1}\Delta v_3=
(\partial_t+\mathcal{H})\partial_{x_1}\Delta v_3\\
+\sin (m_0 y)\partial_{x_1}\partial_{x_1}v_3+\cos (m_0 y)\partial_{x_2}\partial_{x_1}v_3=\div \partial_{x_1} f,
$$
and
$$
(\partial_t+\mathcal{H})\partial_{x_1}\omega_3=\delta\sin(m_0 y)\partial_{x_1}\partial_{x_1}v_3+\delta\cos(m_0 y)\partial_{x_2}\partial_{x_1}v_3+\div \partial_{x_1}g.
$$
Then we have
$$
(\partial_t+\mathcal{H})(\partial_{x_1}\Delta v_3+m_0\partial_{x_1}\omega_3)=\div \partial_{x_1} f+m_0\div \partial_{x_1}g.
$$
By Lemma \ref{lem:2.4} and Proposition \ref{prop:4.1}, one has
$$
\begin{aligned}
\big\|\partial_{x_1}\omega_3\big\|^2_{X_{ed}}
\lesssim &m_0^{-1}\big\|\partial_{x_1}(\Delta v_3)(0)\big\|^2_{L^2}+\big\|\partial_{x_1}\omega_3(0)\big\|^2_{L^2}\\
&+\nu^{-1}m_0^{-1}\left(\left\|e^{\epsilon\nu^{1/2}t}\partial_{x_1}f\right\|^2_{L^2L^2}
+m_0\left\|e^{\epsilon\nu^{1/2}t}\partial_{x_1}g\right\|^2_{L^2L^2}   \right)\\
\lesssim &\big\|\partial_{x_1}(\Delta v_3)(0)\big\|^2_{L^2}+\big\|\partial_{x_1}\omega_3(0)\big\|^2_{L^2}\\
&+\nu^{-1}\left(\left\|e^{\epsilon\nu^{1/2}t}\partial_{x_1}f\right\|^2_{L^2L^2}
+\left\|e^{\epsilon\nu^{1/2}t}\partial_{x_1}g\right\|^2_{L^2L^2}\right).
\end{aligned}
$$
Similarly, one has
$$
\begin{aligned}
\big\|\partial_{x_2}\omega_3\big\|^2_{X_{ed}}
\lesssim &\big\|\partial_{x_2}(\Delta v_3)(0)\big\|^2_{L^2}+\big\|\partial_{x_2}\omega_3(0)\big\|^2_{L^2}\\
&+\nu^{-1}\left(\left\|e^{\epsilon\nu^{1/2}t}\partial_{x_2}f\right\|^2_{L^2L^2}
+\left\|e^{\epsilon\nu^{1/2}t}\partial_{x_2}g\right\|^2_{L^2L^2}\right).
\end{aligned}
$$
This completes the proof of Proposition \ref{prop:4.2}.
\end{proof}

\vskip .05in

\subsection{The estimates of nonlinear terms} Here we get some the estimates of nonlinear terms, first, we need the following lemmas to give the estimate of $v_3$ and $\omega_3$.

\begin{lemma}\label{lem:4.3}
It holds that for $i,j\in \{1,2\}$,
\begin{equation}\label{eq:4.19}
\begin{aligned}
&\left\|\partial_{x_i}\partial_{x_j}(v_i)_{\neq}\right\|_{L^2}+\left\|\partial_{x_j}(v_i)_{\neq}\right\|_{L^2}+
\left\|\partial_{x_i}\partial_{x_j}V\right\|_{L^2}\\
\lesssim&\big\|\partial_{x_1}\omega_3\big\|_{L^2}+\big\|\partial_{x_2}\omega_3\big\|_{L^2}+\big\|(\Delta v_3)_{\neq}\big\|_{L^2},
\end{aligned}
\end{equation}
\vskip .05in
\begin{equation}\label{eq:4.20}
\begin{aligned}
&\left\|\partial_{x_i}\nabla (v_i)_{\neq}\right\|_{L^2}+\left\|\nabla (v_i)_{\neq}\right\|_{L^2}+\big\|\partial_{x_i}\omega\big\|_{L^2}\\
\lesssim&\big\|(\Delta v_3)_{\neq}\big\|_{L^2}+\big\|\partial_{x_1}\nabla\omega_3\big\|_{L^2}+\big\|\partial_{x_2}\nabla\omega_3\big\|_{L^2},
\end{aligned}
\end{equation}
and for $i,j\in \{1,2\}$, one has
\begin{equation}\label{eq:4.21}
\begin{aligned}
&\left\|\partial_{x_i}\partial_{x_j}\nabla(v_i)_{\neq}\right\|_{L^2}
+\left\|\partial_{x_j}\nabla(v_i)_{\neq}\right\|_{L^2}\\
\lesssim&\big\|\partial_{x_1}(\Delta v_3)_{\neq}\big\|_{L^2}+\big\|\partial_{x_2}(\Delta v_3)_{\neq}\big\|_{L^2}+\big\|\partial_{x_1}\nabla\omega_3\big\|_{L^2}+\big\|\partial_{x_2}\nabla\omega_3\big\|_{L^2}.
\end{aligned}
\end{equation}
\end{lemma}

\begin{proof}
Recall that $\omega_3=\partial_{x_1}v_2-\partial_{x_2}v_1$ and $\div V=0$, one gets
$$
v_1=-\left(\partial^2_{x_1}+\partial^2_{x_2}\right)^{-1}\left(\partial_{x_2}\omega_3+\partial_{x_1}\partial_y v_3 \right),
$$
and
$$
v_2=\left(\partial^2_{x_1}+\partial^2_{x_2}\right)^{-1}\left(\partial_{x_1}\omega_3-\partial_{x_2}\partial_y v_3\right).
$$

\vskip .05in

\noindent {\bf Step 1. The estimate of \eqref{eq:4.19}.} Since
$$
\sum_{i,j\in \{1,2\}}\big\|\partial_{x_i}\partial_{x_j}(v_i)_{\neq}\big\|_{L^2}
=\left\|\partial^2_{x_1}(v_1)_{\neq}\right\|_{L^2}+\big\|\partial_{x_1}\partial_{x_2}(v_1)_{\neq}\big\|_{L^2}
+\big\|\partial_{x_1}\partial_{x_2}(v_2)_{\neq}\big\|_{L^2}+\big\|\partial^2_{x_2}(v_2)_{\neq}\big\|_{L^2},
$$
$$
\sum_{i,j\in \{1,2\}}\big\|\partial_{x_j}(v_i)_{\neq}\big\|_{L^2}
=\big\|\partial_{x_1}(v_1)_{\neq}\big\|_{L^2}+\big\|\partial_{x_1}(v_2)_{\neq}\big\|_{L^2}
+\big\|\partial_{x_2}(v_1)_{\neq}\big\|_{L^2}
+\big\|\partial_{x_2}(v_2)_{\neq}\big\|_{L^2},
$$
and
$$
\begin{aligned}
\sum_{i,j\in \{1,2\}}\big\|\partial_{x_i}\partial_{x_j}V\big\|_{L^2}
=&\big\|\partial^2_{x_1}v_1\big\|_{L^2}+2\big\|\partial_{x_1}\partial_{x_2}v_1\big\|_{L^2}
+\big\|\partial^2_{x_2}v_1\big\|_{L^2}
+\big\|\partial^2_{x_1}v_2\big\|_{L^2}+2\big\|\partial_{x_1}\partial_{x_2}v_2\big\|_{L^2}\\
&+\big\|\partial^2_{x_2}v_2\big\|_{L^2}+\big\|\partial^2_{x_1}v_3\big\|_{L^2}+2\big\|\partial_{x_1}\partial_{x_2}v_3\big\|_{L^2}
+\big\|\partial^2_{x_2}v_3\big\|_{L^2}.
\end{aligned}
$$
Denoting
$$
I_1=\sum_{i,j\in \{1,2\}}\big\|\partial_{x_i}\partial_{x_j}(v_i)_{\neq}\big\|_{L^2}+\big\|\partial_{x_j}(v_i)_{\neq}\big\|_{L^2}
+\big\|\partial_{x_i}\partial_{x_j}V\big\|_{L^2},
$$
then combining the expressions of $v_1$, $v_2$ and \eqref{eq:1.9}, we have
$$
\begin{aligned}
I_1\lesssim&\big\|\partial^2_{x_1}v_1\big\|_{L^2}+\big\|\partial_{x_1}\partial_{x_2}v_1\big\|_{L^2}
+\big\|\partial^2_{x_2}v_1\big\|_{L^2}
+\big\|\partial_{x_1}v_1\big\|_{L^2}+\big\|\partial_{x_2}v_1\big\|_{L^2}+\big\|\partial^2_{x_1}v_2\big\|_{L^2}
+\big\|\partial_{x_1}\partial_{x_2}v_2\big\|_{L^2}\\
&+\big\|\partial^2_{x_2}v_2\big\|_{L^2}
+\big\|\partial_{x_1}v_2\big\|_{L^2}+\big\|\partial_{x_2}v_2\big\|_{L^2}+\big\|\partial^2_{x_1}v_3\big\|_{L^2}
+\big\|\partial_{x_1}\partial_{x_2}v_3\big\|_{L^2}+\big\|\partial^2_{x_2}v_3\big\|_{L^2}\\
\lesssim&\big\|\partial_{x_1}\omega_3\big\|_{L^2}+\big\|\partial_{x_2}\partial_y v_3\big\|_{L^2}+\big\|\partial_{x_2} \omega_3\big\|_{L^2}+\big\|\partial_{x_1}\partial_y v_3\big\|_{L^2}+\big\|\partial^2_{x_1}v_3\big\|_{L^2}\\
&+\big\|\partial_{x_1}\partial_{x_2}v_3\big\|_{L^2}+\big\|\partial^2_{x_2}v_3\big\|_{L^2}\\
\lesssim&\big\|\partial_{x_1}\omega_3\big\|_{L^2}+\big\|\partial_{x_2}\omega_3\big\|_{L^2}+\big\|\partial_{x_2}\partial_y (v_3)_{\neq}\big\|_{L^2}+\big\|\partial_{x_1}\partial_y (v_3)_{\neq}\big\|_{L^2}\\
&+\big\|\partial_{x_1}\partial_{x_2}(v_3)_{\neq}\big\|_{L^2}+\big\|\partial^2_{x_1}(v_3)_{\neq}\big\|_{L^2}
+\big\|\partial^2_{x_2}(v_3)_{\neq}\big\|_{L^2}\\
\lesssim& \big\|\partial_{x_1}\omega_3\big\|_{L^2}+\big\|\partial_{x_2}\omega_3\big\|_{L^2}+\big\|(\Delta v_3)_{\neq}\big\|_{L^2}.
\end{aligned}
$$

\vskip .05in
\noindent {\bf Step 2. The estimate of \eqref{eq:4.20}.} Similarly, denoting
$$
I_2=\sum_{i\in\{1,2\}}\big\|\partial_{x_i}\nabla (v_i)_{\neq}\big\|_{L^2}+\big\|\nabla (v_i)_{\neq}\big\|_{L^2}+\big\|\partial_{x_i}\omega\big\|_{L^2},
$$
then we have
$$
\begin{aligned}
I_2\lesssim &\big\|\partial_{x_1}\nabla (v_1)_{\neq}\big\|_{L^2}+\big\|\partial_{x_2}\nabla (v_2)_{\neq}\big\|_{L^2}+\big\|\nabla (v_1)_{\neq}\big\|_{L^2}+\big\|\nabla (v_2)_{\neq}\big\|_{L^2}\\
&+\big\|\partial_{x_1}\omega_1\big\|_{L^2}+\big\|\partial_{x_2}\omega_1\big\|_{L^2}
+\big\|\partial_{x_1}\omega_2\big\|_{L^2}+\big\|\partial_{x_2}\omega_2\big\|_{L^2}
+\big\|\partial_{x_1}\omega_3\big\|_{L^2}+\big\|\partial_{x_2}\omega_3\big\|_{L^2}\\
\lesssim&\big\|\partial_{x_1}\nabla (v_1)_{\neq}\big\|_{L^2}+\big\|\partial_{x_2}\nabla (v_1)_{\neq}\big\|_{L^2}
+\big\|\partial_{x_1}\nabla (v_2)_{\neq}\big\|_{L^2}\\
&+\big\|\partial_{x_2}\nabla (v_2)_{\neq}\big\|_{L^2}
+\big\|\partial_{x_1}\nabla V\big\|_{L^2}+\big\|\partial_{x_2}\nabla V\big\|_{L^2}\\
\lesssim&\big\|\partial_{x_1}\nabla (v_1)_{\neq}\big\|_{L^2}+\big\|\partial_{x_2}\nabla (v_1)_{\neq}\big\|_{L^2}
+\big\|\partial_{x_1}\nabla (v_2)_{\neq}\big\|_{L^2}\\
&+\big\|\partial_{x_2}\nabla (v_2)_{\neq}\big\|_{L^2}
+\big\|\partial_{x_1}\nabla (v_3)_{\neq}\big\|_{L^2}+\big\|\partial_{x_2}\nabla (v_3)_{\neq}\big\|_{L^2}\\
\lesssim& \big\| \nabla(\omega_3)_{\neq}\big\|_{L^2}+\big\|\partial_y \nabla (v_3)_{\neq}\big\|_{L^2}
+\big\|\partial_{x_1}\nabla (v_3)_{\neq}\big\|_{L^2}+\big\|\partial_{x_2}\nabla (v_3)_{\neq}\big\|_{L^2}\\
\lesssim&\big\|(\Delta v_3)_{\neq}\big\|_{L^2}+\big\|\partial_{x_1}\nabla\omega_3\big\|_{L^2}+\big\|\partial_{x_2}\nabla\omega_3\big\|_{L^2}.
\end{aligned}
$$

\vskip .05in

\noindent {\bf Step 3. The estimate of \eqref{eq:4.21}.} We denote
$$
I_3=\sum_{i,j\in\{1,2\}}\big\|\partial_{x_i}\partial_{x_j}\nabla(v_i)_{\neq}\big\|_{L^2}
+\big\|\partial_{x_j}\nabla(v_i)_{\neq}\big\|_{L^2},
$$
and by the expressions of $v_1$, $v_2$, \eqref{eq:1.9} and \eqref{eq:4.19}, we have
$$
\begin{aligned}
I_3=&\big\|\partial^2_{x_1}\nabla(v_1)_{\neq}\big\|_{L^2}+\big\|\partial_{x_1}\partial_{x_2}\nabla(v_1)_{\neq}\big\|_{L^2}
+\big\|\partial_{x_1}\nabla(v_1)_{\neq}\big\|_{L^2}+\big\|\partial_{x_2}\nabla(v_1)_{\neq}\big\|_{L^2}\\
&+\big\|\partial_{x_1}\partial_{x_2}\nabla(v_2)_{\neq}\big\|_{L^2}+\big\|\partial^2_{x_2}\nabla(v_2)_{\neq}\big\|_{L^2}
+\big\|\partial_{x_1}\nabla(v_2)_{\neq}\big\|_{L^2}+\big\|\partial_{x_2}\nabla(v_2)_{\neq}\big\|_{L^2}\\
\lesssim&\big\|\partial_{x_1}\nabla \omega_3\big\|_{L^2}+\big\|\partial_{x_2}\nabla \omega_3\big\|_{L^2}
+\big\|\partial_{x_1}\nabla\partial_y v_3\big\|_{L^2}+\big\|\partial_{x_2}\nabla\partial_y v_3\big\|_{L^2}\\
\lesssim&\big\|\partial_{x_1}\nabla \omega_3\big\|_{L^2}+\big\|\partial_{x_2}\nabla \omega_3\big\|_{L^2}
+\big\|\partial_{x_1}(\Delta v_3)_{\neq}\big\|_{L^2}+\big\|\partial_{x_2}(\Delta v_3)_{\neq}\big\|_{L^2}.
\end{aligned}
$$
This completes the proof of Lemma \ref{lem:4.3}.
\end{proof}

\begin{remark}
Using \eqref{eq:1.9}, we know that $\partial_{x_i}(v_j)_{\neq}=\partial_{x_i}v_j,i=1,2$. In the discussion of Lemma {\rm \ref{lem:4.3}} and later section, both symbols are used for the convenience of proof.
\end{remark}

Next, we obtain some estimates for nonlinear term, the details are as follows.

\begin{lemma}\label{lem:4.4}
For $i\in \{1,2\}$, it holds that
$$
\begin{aligned}
&\big\|\partial_{x_i}(\Delta p)\big\|_{L^2}+\big\|\partial_{x_i}\nabla(V\cdot\nabla v_3)\big\|_{L^2}
+\big\|\partial_{x_i}(\omega_3V)\big\|_{L^2}+\big\|\partial_{x_i}(\omega v_3)\big\|_{L^2}\\
\lesssim&\big(\big\|\partial_{x_1}(\Delta v_3)_{\neq}\big\|_{L^2}+\big\|\partial_{x_2}(\Delta v_3)_{\neq}\big\|_{L^2}+\big\|\partial_{x_1}\omega_3\big\|_{L^2}+\big\|\partial_{x_2}\omega_3\big\|_{L^2}\big)\\
&\ \ \cdot\big(\big\|V\big\|_{H^2}+\big\|\partial_{x_1}V\big\|_{H^2}+\big\|\partial_{x_2}V\big\|_{H^2}\big)\\
&+\big(\big\|\nabla\partial_{x_1}\omega_3\big\|_{L^2}+\big\|\nabla\partial_{x_2}\omega_3\big\|_{L^2}\big)
\big(\big\|v_3\big\|_{H^2}+\big\|\partial_{x_1}v_3\big\|_{H^2}+\big\|\partial_{x_2}v_3\big\|_{H^2}\big)\\
&+\big(\big\|\partial_{x_1}\omega_3\big\|_{L^2}+\big\|\partial_{x_2}\omega_3\big\|_{L^2}+\big\|(\Delta v_3)_{\neq}\big\|_{L^2}\big)\\
&\cdot\big(\big\|\nabla\Delta v_3\big\|_{L^2}+\big\|\partial_{x_1}\nabla\Delta v_3\big\|_{L^2}+\big\|\partial_{x_2}\nabla\Delta v_3\big\|_{L^2}\big),
\end{aligned}
$$

\vskip .05in

$$
\begin{aligned}
\big\|\Delta p\big\|_{L^2}+\big\|\nabla(V\cdot\nabla v_3)\big\|_{L^2}\lesssim
&\big(\big\|\nabla\partial_{x_1}\omega_3\big\|_{L^2}+\big\|\nabla\partial_{x_2}\omega_3\big\|_{L^2}
+\big\|\nabla\Delta v_3\big\|_{L^2}\big)\big\|\Delta v_3\big\|_{L^2}\\
&+\big(\big\|\partial_{x_1}\omega_3\big\|_{L^2}+\big\|\partial_{x_2}\omega_3\big\|_{L^2}+\big\|(\Delta v_3)_{\neq}\big\|_{L^2})\\
&\cdot\big(\big\|V\big\|_{H^2}+\big\|\nabla\Delta v_3\big\|_{L^2}\big),
\end{aligned}
$$

\vskip .05in

$$
\big\|\nabla(V\cdot\nabla V)_{\neq}\big\|_{L^2}+\big\|V_{\neq}\cdot\nabla V_{\neq}\big\|_{H^1}\lesssim\big\|V\big\|_{H^2}\big\|\Delta V_{\neq}\big\|_{L^2},
$$
and
$$
\big\|\partial_{x_i}\nabla(V\cdot\nabla V)\big\|_{L^2}+\big\|\partial_{x_i}\nabla(V_{\neq}\cdot\nabla V_{\neq})\big\|_{L^2}
\lesssim \big\|\partial_{x_i}V\big\|_{H^2}\big\|\Delta V_{\neq}\big\|_{L^2}+\big\|V\big\|_{H^2}\big\|\partial_{x_i}\Delta V_{\neq}\big\|_{L^2}.
$$
\end{lemma}

\begin{proof}
We finish the proof by following five steps.

\vskip .05in

\noindent {\bf Step 1. The estimate of pressure $p$.} Recall that
\begin{equation}\label{eq:4.22}
\Delta p=-\sum_{i,j=1}^{3}\partial_iv_j\partial_jv_i
=-\sum_{i,j=1}^{3}\partial_iv_j\partial_j(v_i)_{\neq}-\sum_{i,j=1}^{3}\partial_iv_j\partial_jP_0v_i,
\end{equation}
where $\partial_i=\partial_{x_i}, i=1,2$, $\partial_3=\partial_y$. For $i,j\in \{1,2\}$, by Lemma \ref{lem:2.5} and \ref{lem:4.3}, one gets
$$
\begin{aligned}
\big\|\partial_iv_j\partial_j(v_i)_{\neq}\big\|_{L^2}&\lesssim \big(\big\|\partial_i\partial_j(v_i)_{\neq}\big\|_{L^2}+\big\|\partial_j(v_i)_{\neq}\big\|_{L^2}\big)\big\|\Delta v_j\big\|_{L^2}\\
&\lesssim \big(\big\|(\Delta v_3)_{\neq}\big\|_{L^2}+\big\|\partial_{x_1}\omega_3\big\|_{L^2}+\big\|\partial_{x_2}\omega_3\big\|_{L^2}\big)
\big\|V\big\|_{H^2}.
\end{aligned}
$$
For $i\in \{1,2\}, j=3$, by Lemma \ref{lem:2.5} and \ref{lem:4.3}, one has
$$
\begin{aligned}
\big\|\partial_iv_j\partial_j(v_i)_{\neq}\big\|_{L^2}&\lesssim \big(\big\|\partial_i\partial_j(v_i)_{\neq}\big\|_{L^2}+\big\|\partial_j(v_i)_{\neq}\big\|_{L^2}\big)\big\|\Delta v_j\big\|_{L^2}\\
&\lesssim \big(\big\|(\Delta v_3)_{\neq}\big\|_{L^2}+\big\|\nabla\partial_{x_1}\omega_3\big\|_{L^2}+\big\|\nabla\partial_{x_2}\omega_3\big\|_{L^2}\big)
\big\|\Delta v_3\big\|_{L^2}.
\end{aligned}
$$
For $i=3$, we imply that
$$
\big\|\partial_iv_j\partial_j(v_i)_{\neq}\big\|_{L^2}\lesssim\big\|\partial_iv_j\big\|_{L^4}
\big\|\partial_j(v_i)_{\neq}\big\|_{L^4}
\lesssim \big\|V\big\|_{H^2}\big\|(\Delta v_3)_{\neq}\big\|_{L^2}.
$$
Then we have
\begin{equation}\label{eq:4.23}
\begin{aligned}
\sum_{i,j=1}^{3}\big\|\partial_iv_j\partial_j(v_i)_{\neq}\big\|_{L^2}
\lesssim&\big(\big\|(\Delta v_3)_{\neq}\big\|_{L^2}+\big\|\partial_{x_1}\omega_3\big\|_{L^2}+\big\|\partial_{x_2}\omega_3\big\|_{L^2}\big)
\big\|V\big\|_{H^2}\\
&+\big(\big\|\nabla\partial_{x_1}\omega_3\big\|_{L^2}+\big\|\nabla\partial_{x_2}\omega_3\big\|_{L^2}\big)\big\|\Delta v_3\big\|_{L^2}.
\end{aligned}
\end{equation}
Since
$$
\sum_{i,j=1}^{3}\partial_iv_j\partial_jP_0v_i=\sum_{i,j=1}^{3}P_0\left( \partial_iv_j\partial_jP_0v_i\right)+\sum_{i,j=1}^{3}\partial_i(v_j)_{\neq}\partial_jP_0v_i.
$$
Similar to \eqref{eq:4.23}, one has
\begin{equation}\label{eq:4.24}
\begin{aligned}
\sum_{i,j=1}^{3}\big\|(\partial_iv_j\partial_jP_0v_i)_{\neq}\big\|_{L^2}
&=\sum_{i,j=1}^{3}\big\|\partial_i(v_j)_{\neq}\partial_jP_0v_i\big\|_{L^2}\\
&\lesssim\big(\big\|(\Delta v_3)_{\neq}\big\|_{L^2}+\big\|\partial_{x_1}\omega_3\big\|_{L^2}+\big\|\partial_{x_2}\omega_3\big\|_{L^2}\big)
\big\|P_0V\big\|_{H^2}\\
&\ \ \ +\big(\big\|\nabla\partial_{x_1}\omega_3\big\|_{L^2}+\big\|\nabla\partial_{x_2}\omega_3\big\|_{L^2}\big)
\big\|P_0v_3\big\|_{H^2}\\
&\lesssim\big(\big\|(\Delta v_3)_{\neq}\big\|_{L^2}+\big\|\partial_{x_1}\omega_3\big\|_{L^2}+\big\|\partial_{x_2}\omega_3\big\|_{L^2}\big)
\big\|V\big\|_{H^2}\\
&\ \ \ +\big(\big\|\nabla\partial_{x_1}\omega_3\big\|_{L^2}+\big\|\nabla\partial_{x_2}\omega_3\big\|_{L^2}\big)
\big\|v_3\big\|_{H^2},
\end{aligned}
\end{equation}
and
\begin{equation}\label{eq:4.25}
\sum_{i,j=1}^3\big\|P_0\left( \partial_iv_j\partial_jP_0v_i\right)\big\|_{L^2}\lesssim \big\|\partial_yP_0v_3\big\|^2_{L^4}
\lesssim \big\|\Delta v_3\big\|_{L^2}\big\|\nabla\Delta v_3\big\|_{L^2}.
\end{equation}
Combining \eqref{eq:4.22}-\eqref{eq:4.25}, to obtain
\begin{equation}\label{eq:4.26}
\begin{aligned}
\big\|\Delta p\big\|_{L^2}\lesssim
&\big(\big\|(\Delta v_3)_{\neq}\big\|_{L^2}+\big\|\partial_{x_1}\omega_3\big\|_{L^2}+\big\|\partial_{x_2}\omega_3\big\|_{L^2}\big)
\big\|V\big\|_{H^2}\\
&+\big(\big\|\nabla\partial_{x_1}\omega_3\big\|_{L^2}+\big\|\nabla\partial_{x_2}\omega_3\big\|_{L^2}+\big\|\nabla\Delta v_3\big\|_{L^2}\big)\big\|\Delta v_3\big\|_{L^2}.
\end{aligned}
\end{equation}
Next, we consider
\begin{equation}\label{eq:4.27}
\begin{aligned}
\partial_{x_1}\Delta p
&=-\sum_{i,j=1}^{3}\partial_{x_1}(\partial_iv_j\partial_jv_i)\\
&=-\sum_{i,j=1}^{3}\partial_i\partial_{x_1}v_j\partial_j(v_i)_{\neq}
-\sum_{i,j=1}^{3}\partial_iv_j\partial_j\partial_{x_1}(v_i)_{\neq}
-\sum_{i,j=1}^{3}\partial_i\partial_{x_1}v_j\partial_jP_0v_i.
\end{aligned}
\end{equation}
For $i,j\in \{1,2\}$, we deduce by Lemma \ref{lem:2.5} and \ref{lem:4.3} that
$$
\begin{aligned}
\big\|\partial_i\partial_{x_1}v_j\partial_j(v_i)_{\neq}\big\|_{L^2}&\lesssim \big(\big\|\partial_i\partial_j(v_i)_{\neq}\big\|_{L^2}+\big\|\partial_j(v_i)_{\neq}\big\|_{L^2}\big)
\big\|\partial_{x_1}\Delta v_j\big\|_{L^2}\\
&\lesssim \big(\big\|(\Delta v_3)_{\neq}\big\|_{L^2}+\big\|\partial_{x_1}\omega_3\big\|_{L^2}+\big\|\partial_{x_2}\omega_3\big\|_{L^2}\big)
\big\|\partial_{x_1}V\big\|_{H^2},
\end{aligned}
$$
and
$$
\begin{aligned}
\big\|\partial_iv_j\partial_j\partial_{x_1}(v_i)_{\neq}\big\|_{L^2}
&\lesssim \big(\big\|(\Delta v_3)_{\neq}\big\|_{L^2}+\big\|\partial_{x_1}\omega_3\big\|_{L^2}+\big\|\partial_{x_2}\omega_3\big\|_{L^2}\big)
\big\|\partial_{x_1}V\big\|_{H^2}.
\end{aligned}
$$
For $i\in \{1,2\}, j=3$, by Lemma \ref{lem:2.5} and \ref{lem:4.3}, one has
$$
\begin{aligned}
\big\|\partial_i\partial_{x_1}v_j\partial_j(v_i)_{\neq}\big\|_{L^2}&\lesssim \big(\big\|\partial_i\partial_j(v_i)_{\neq}\big\|_{L^2}+\big\|\partial_j(v_i)_{\neq}\big\|_{L^2}\big)
\big\|\partial_{x_1}\Delta v_j\big\|_{L^2}\\
&\lesssim \big(\big\|(\Delta v_3)_{\neq}\big\|_{L^2}+\big\|\nabla\partial_{x_1}\omega_3\big\|_{L^2}+\big\|\nabla\partial_{x_2}\omega_3\big\|_{L^2}\big)
\big\|\partial_{x_1}\Delta v_3\big\|_{L^2},
\end{aligned}
$$

$$
\begin{aligned}
\big\|\partial_iv_j\partial_j\partial_{x_1}(v_i)_{\neq}\big\|_{L^2}&\lesssim \big(\big\|\partial_i\partial_j\partial_{x_1}(v_i)_{\neq}\big\|_{L^2}+\big\|\partial_j\partial_{x_1}(v_i)_{\neq}\big\|_{L^2}\big)
\big\|\Delta v_j\big\|_{L^2}\\
&\lesssim\big(\big\|\partial_i\partial_{x_1}\nabla(v_i)_{\neq}\big\|_{L^2}+\big\|\partial_{x_1}\nabla(v_i)_{\neq}\big\|_{L^2}\big)
\big\|\Delta v_j\big\|_{L^2}\\
&\lesssim \big(\big\|\partial_{x_1}(\Delta v_3)_{\neq}\big\|_{L^2}+\big\|\partial_{x_2}(\Delta v_3)_{\neq}\big\|_{L^2}+\big\|\nabla\partial_{x_1}\omega_3\big\|_{L^2}+\big\|\nabla\partial_{x_2}\omega_3\big\|_{L^2}\big)
\big\|\Delta v_3\big\|_{L^2},
\end{aligned}
$$
and
$$
\big\|\partial_i\partial_{x_1}v_j\partial_jP_0v_i\big\|_{L^2}\lesssim \big\|(\Delta v_3)_{\neq}\big\|_{L^2}\big\|P_0 V\big\|_{H^2}
\lesssim \big\|(\Delta v_3)_{\neq}\big\|_{L^2}\big\|V\big\|_{H^2}.
$$
For $i=3$, we deduce by energy inequalities that
$$
\big\|\partial_i\partial_{x_1}v_j\partial_j(v_i)_{\neq}\big\|_{L^2}
\lesssim\big\|\partial_i\partial_{x_1}v_j\big\|_{L^4}\big\|\partial_j(v_i)_{\neq}\big\|_{L^4}
\lesssim \big\|\partial_{x_1}V\big\|_{H^2}\big\|(\Delta v_3)_{\neq}\big\|_{L^2},
$$

$$
\big\|\partial_iv_j\partial_j\partial_{x_1}(v_i)_{\neq}\big\|_{L^2}
\lesssim\big\|\partial_iv_j\big\|_{L^4}\big\|\partial_j\partial_{x_1}(v_i)_{\neq}\big\|_{L^4}
\lesssim \big\|V\big\|_{H^2}\big\|\partial_{x_1}(\Delta v_3)_{\neq}\big\|_{L^2},
$$
and
$$
\big\|\partial_i\partial_{x_1}v_j\partial_jP_0v_i\big\|_{L^2}
\lesssim\big(\big\|\partial_{x_1}(\Delta v_3)_{\neq}\big\|_{L^2}+\big\|\partial_{x_2}(\Delta v_3)_{\neq}\big\|_{L^2}+\big\|\nabla\partial_{x_1}\omega_3\big\|_{L^2}+\big\|\nabla\partial_{x_2}\omega_3\big\|_{L^2}\big)
\big\|\Delta v_3\big\|_{L^2}.
$$
Then we have from \eqref{eq:4.27} that
\begin{equation}\label{eq:4.28}
\begin{aligned}
&\big\|\partial_{x_1}\Delta p\big\|_{L^2}\\
\lesssim&\sum_{i,j=1}^{3}\big\|\partial_i\partial_{x_1}v_j\partial_j(v_i)_{\neq}\big\|_{L^2}
+\sum_{i,j=1}^{3}\big\|\partial_iv_j\partial_j\partial_{x_1}(v_i)_{\neq}\big\|_{L^2}
+\sum_{i,j=1}^{3}\big\|\partial_i\partial_{x_1}v_j\partial_jP_0v_i\big\|_{L^2}\\
\lesssim& \big(\big\|\partial_{x_1}(\Delta v_3)_{\neq}\big\|_{L^2}+\big\|\partial_{x_2}(\Delta v_3)_{\neq}\big\|_{L^2}+\big\|\nabla\partial_{x_1}\omega_3\big\|_{L^2}+\big\|\nabla\partial_{x_2}\omega_3\big\|_{L^2}\big)
\big\|\Delta v_3\big\|_{L^2}\\
&+\big(\big\|(\Delta v_3)_{\neq}\big\|_{L^2}+\big\|\partial_{x_1}\omega_3\big\|_{L^2}+\big\|\partial_{x_2}\omega_3\big\|_{L^2}\big)
\big\|\partial_{x_1}V\big\|_{H^2}+\big\|V\big\|_{H^2}\big\|\partial_{x_1}(\Delta v_3)_{\neq}\big\|_{L^2}\\
&+\big\|\partial_{x_1}V\big\|_{H^2}\big\|(\Delta v_3)_{\neq}\big\|_{L^2}+\big(\big\|(\Delta v_3)_{\neq}\big\|_{L^2}+\big\|\nabla\partial_{x_1}\omega_3\big\|_{L^2}+\big\|\nabla\partial_{x_2}\omega_3\big\|_{L^2}\big)
\big\|\partial_{x_1}\Delta v_3\big\|_{L^2}\\
&+\big\|(\Delta v_3)_{\neq}\big\|_{L^2}\big\|V\big\|_{H^2}\\
\lesssim&\big(\big\|\partial_{x_1}(\Delta v_3)_{\neq}\big\|_{L^2}+\big\|\partial_{x_2}(\Delta v_3)_{\neq}\big\|_{L^2}+\big\|\partial_{x_1}\omega_3\big\|_{L^2}+\big\|\partial_{x_2}\omega_3\big\|_{L^2}\big)
\cdot\big(\big\|V\big\|_{H^2}+\big\|\partial_{x_1}V\big\|_{H^2}\big)\!\!\!\\
&+\big(\big\|\nabla\partial_{x_1}\omega_3\big\|_{L^2}+\big\|\nabla\partial_{x_2}\omega_3\big\|_{L^2}\big)\cdot
\big(\big\|\Delta v_3\big\|_{L^2}+\big\|\partial_{x_1}\Delta v_3\big\|_{L^2}\big).
\end{aligned}
\end{equation}
Similarly, we also have
\begin{equation}\label{eq:4.29}
\begin{aligned}
\big\|\partial_{x_2}\Delta p\big\|_{L^2}\lesssim&\big(\big\|\partial_{x_1}(\Delta v_3)_{\neq}\big\|_{L^2}+\big\|\partial_{x_2}(\Delta v_3)_{\neq}\big\|_{L^2}+\big\|\partial_{x_1}\omega_3\big\|_{L^2}+\big\|\partial_{x_2}\omega_3\big\|_{L^2}\big)\\
&\cdot\big(\big\|V\big\|_{H^2}+\big\|\partial_{x_2}V\big\|_{H^2}\big)\\
&+\big(\big\|\nabla\partial_{x_1}\omega_3\big\|_{L^2}+\big\|\nabla\partial_{x_2}\omega_3\big\|_{L^2}\big)
\cdot\big(\big\|\Delta v_3\big\|_{L^2}+\big\|\partial_{x_2}\Delta v_3\big\|_{L^2}\big).
\end{aligned}
\end{equation}
\vskip .05in

\noindent {\bf Step 2. The estimate of $\nabla(V\cdot\nabla v_3)$.}  We write
\begin{equation}\label{eq:4.30}
\nabla(V\cdot\nabla v_3)=\sum_{j=1}^3\nabla\left(v_j\partial_j v_3 \right)
=\sum_{j=1}^3\nabla\left((v_j)_{\neq}\partial_j v_3 \right)+\sum_{j=1}^3\nabla\left(P_0v_j\partial_j v_3 \right).
\end{equation}
For $j\in \{1,2\}$, by Lemma  \ref{lem:2.5} and \ref{lem:4.3}, one gets
$$
\begin{aligned}
\big\|\nabla\left((v_j)_{\neq}\partial_j v_3 \right)\big\|_{L^2}
&\lesssim \big\|\nabla(v_j)_{\neq}\partial_j v_3 \big\|_{L^2}+\big\|(v_j)_{\neq}\nabla\partial_j v_3 \big\|_{L^2}\\
&\lesssim \big(\big\|\partial_j\nabla(v_j)_{\neq}\big\|_{L^2}+\big\|\nabla(v_j)_{\neq}\big\|_{L^2} \big)
\big\|\Delta v_3\big\|_{L^2}\\
&\ \ \ +\big(\big\|\partial_j(v_j)_{\neq}\big\|_{L^2}+\big\|(v_j)_{\neq}\big\|_{L^2}\big)\big\|\nabla\Delta v_3\big\|_{L^2}\\
&\lesssim \big(\big\|(\Delta v_3)_{\neq}\big\|_{L^2}+\big\|\partial_{x_1}\nabla\omega_3\big\|_{L^2}+\big\|\partial_{x_2}\nabla\omega_3\big\|_{L^2}\big)
\big\|\Delta v_3\big\|_{L^2}\\
&\ \ \ +\big(\big\|\partial_{x_1}\omega_3\big\|_{L^2}+\big\|\partial_{x_2}\omega_3\big\|_{L^2}+\big\|(\Delta v_3)_{\neq}\big\|_{L^2}\big)\big\|\nabla\Delta v_3\big\|_{L^2}.
\end{aligned}
$$
For $j=3$, one has that
$$
\begin{aligned}
\big\|\nabla\left((v_j)_{\neq}\partial_j v_3 \right)\big\|_{L^2}
&\lesssim \big\|\nabla(v_j)_{\neq}\partial_j v_3\big\|_{L^2}+\big\|(v_j)_{\neq}\partial_j\nabla v_3\big\|_{L^2}\\
&\lesssim \big\|(v_j)_{\neq}\|_{H^2}\big\|\partial_jv_3\big\|_{H^1}
\lesssim \big\|(\Delta v_3)_{\neq}\big\|_{L^2}\big\|\Delta v_3\big\|_{L^2}.
\end{aligned}
$$
Then we show that
\begin{equation}\label{eq:4.31}
\begin{aligned}
\sum_{j=1}^3\big\|\nabla\left((v_j)_{\neq}\partial_j v_3 \right)\big\|_{L^2}
\lesssim&\big(\big\|\partial_{x_1}\omega_3\big\|_{L^2}+\big\|\partial_{x_2}\omega_3\big\|_{L^2}+\big\|(\Delta v_3)_{\neq}\big\|_{L^2}\big)\big\|\nabla\Delta v_3\big\|_{L^2}\\
&+\big(\big\|\nabla\partial_{x_1}\omega_3\big\|_{L^2}+\big\|\nabla\partial_{x_2}\omega_3\big\|_{L^2}\big)
\big\|v_3\big\|_{H^2}.
\end{aligned}
\end{equation}
Since
$$
\sum_{j=1}^3\nabla\left(P_0v_j\partial_j v_3 \right)
=\sum_{j=1}^3\nabla\left(P_0v_j\partial_j (v_3)_{\neq}\right)+\partial_y\left(P_0v_3\partial_y P_0v_3 \right),
$$
where we use the $\partial_{x_1}P_0=\partial_{x_2}P_0=0$. By energy inequalities, one has
$$
\big\|\nabla\left(P_0v_j\partial_j (v_3)_{\neq} \right)\big\|_{L^2}
\lesssim \big\|P_0v_j\big\|_{H^2}\big\|\partial_j (v_3)_{\neq}\big\|_{H^1}
\lesssim \big\|V\big\|_{H^2}\big\|(\Delta v_3)_{\neq}\big\|_{L^2},
$$
and
$$
\big\|\partial_y\left(P_0v_3\partial_y  P_0v_3 \right)\big\|_{L^2}\lesssim \big\|P_0v_3\big\|_{H^1}\big\|\partial_y  P_0v_3 \big\|_{H^2}
\lesssim \big\|v_3\big\|_{H^2}\big\|\nabla\Delta v_3\big\|_{L^2}.
$$
Then we obtain
\begin{equation}\label{eq:4.32}
\sum_{j=1}^3\big\|\nabla\left(P_0v_j\partial_j v_3 \right)\big\|_{L^2}
\lesssim\big\|V\big\|_{H^2}\big\|(\Delta v_3)_{\neq}\big\|_{L^2}+\big\|v_3\big\|_{H^2}\big\|\nabla\Delta v_3\big\|_{L^2}.
\end{equation}
Combining \eqref{eq:4.30}-\eqref{eq:4.32}, we have
\begin{equation}\label{eq:4.33}
\begin{aligned}
\big\|\nabla(V\cdot\nabla v_3)\big\|_{L^2}
\lesssim&\big(\big\|\partial_{x_1}\omega_3\big\|_{L^2}+\big\|\partial_{x_2}\omega_3\big\|_{L^2}+\big\|(\Delta v_3)_{\neq}\big\|_{L^2}\big)\big\|\nabla\Delta v_3\big\|_{L^2}\\
&+\big(\big\|\nabla\partial_{x_1}\omega_3\big\|_{L^2}+\big\|\nabla\partial_{x_2}\omega_3\big\|_{L^2}\big)
\big\|v_3\big\|_{H^2}\\
&+\big\|V\big\|_{H^2}\big\|(\Delta v_3)_{\neq}\big\|_{L^2}+\big\|v_3\big\|_{H^2}\big\|\nabla\Delta v_3\big\|_{L^2}\\
\lesssim&\big(\big\|\partial_{x_1}\omega_3\big\|_{L^2}+\big\|\partial_{x_2}\omega_3\big\|_{L^2}+\big\|(\Delta v_3)_{\neq}\big\|_{L^2}\big)\big(\big\|V\big\|_{H^2}+\big\|\nabla\Delta v_3\big\|_{L^2}\big)\\
&+\big(\big\|\nabla\partial_{x_1}\omega_3\big\|_{L^2}+\big\|\nabla\partial_{x_2}\omega_3\big\|_{L^2}\big)
\big\|v_3\big\|_{H^2}+\big\|v_3\big\|_{H^2}\big\|\nabla\Delta v_3\big\|_{L^2}.
\end{aligned}
\end{equation}
Next, we consider
\begin{equation}\label{eq:4.34}
\partial_{x_1}\nabla(V\cdot\nabla v_3)=\sum_{j=1}^3\nabla\partial_{x_1}\left((v_j)_{\neq}\partial_j v_3 \right)+\sum_{j=1}^3\nabla\left(P_0v_j\partial_j\partial_{x_1} (v_3)_{\neq}\right).
\end{equation}
Since
$$
\begin{aligned}
J_1=&\sum_{j=1}^3\big\|\nabla\partial_{x_1}\left((v_j)_{\neq}\partial_j v_3 \right)\big\|_{L^2}\\
\lesssim&\sum_{j=1}^3\big\|\nabla\partial_{x_1}(v_j)_{\neq}\partial_j v_3\big\|_{L^2}+\sum_{j=1}^3\big\|\partial_{x_1}(v_j)_{\neq}\partial_j\nabla v_3\big\|_{L^2}\\
&+\sum_{j=1}^3\big\|\nabla(v_j)_{\neq}\partial_j \partial_{x_1}v_3\big\|_{L^2}+\sum_{j=1}^3\big\|(v_j)_{\neq}\partial_j \nabla\partial_{x_1}v_3\big\|_{L^2}.
\end{aligned}
$$
For $j=1,2$, we obtain by Lemma \ref{lem:2.5} and \ref{lem:4.3} that
$$
\begin{aligned}
\big\|\nabla\partial_{x_1}(v_j)_{\neq}\partial_j v_3\big\|_{L^2}
&\lesssim\big(\big\|\partial_j\nabla \partial_{x_1}(v_j)_{\neq}\big\|_{L^2}+\big\|\nabla \partial_{x_1}(v_j)_{\neq}\big\|_{L^2}\big)\big\|\Delta v_3\big\|_{L^2}\\
&\lesssim\big(\big\|\partial_{x_1}(\Delta v_3)_{\neq}\big\|_{L^2}+\big\|\partial_{x_2}(\Delta v_3)_{\neq}\big\|_{L^2}+\big\|\nabla\partial_{x_1}\omega_3\big\|_{L^2}+\big\|\nabla\partial_{x_2}\omega_3\big\|_{L^2}\big)
\big\|\Delta v_3\big\|_{L^2},
\end{aligned}
$$

\vskip .05in

$$
\begin{aligned}
\big\|\partial_{x_1}(v_j)_{\neq}\partial_j\nabla v_3\big\|_{L^2}
&\lesssim\big(\big\|\partial_j \partial_{x_1}(v_j)_{\neq}\big\|_{L^2}+\big\|\partial_{x_1}(v_j)_{\neq}\big\|_{L^2}\big)
\big\|\nabla \Delta v_3\big\|_{L^2}\\
&\lesssim\big(\big\|\partial_{x_1}\omega_3\big\|_{L^2}+\big\|\partial_{x_2}\omega_3\big\|_{L^2}+\big\|(\Delta v_3)_{\neq}\big\|_{L^2}\big)\big\|\nabla \Delta v_3\big\|_{L^2},
\end{aligned}
$$

\vskip .05in

$$
\begin{aligned}
\big\|\nabla(v_j)_{\neq}\partial_j \partial_{x_1}v_3\big\|_{L^2}
&\lesssim\big(\big\|\partial_j\nabla (v_j)_{\neq}\big\|_{L^2}+\big\|\nabla (v_j)_{\neq}\big\|_{L^2}\big)\big\|\partial_{x_1}\Delta v_3\big\|_{L^2}\\
&\lesssim\big(\big\|(\Delta v_3)_{\neq}\big\|_{L^2}+\big\|\nabla\partial_{x_1}\omega_3\big\|_{L^2}+\big\|\nabla\partial_{x_2}\omega_3\big\|_{L^2}\big)
\big\|\partial_{x_1}\Delta v_3\big\|_{L^2},
\end{aligned}
$$

\vskip .05in

$$
\begin{aligned}
\big\|(v_j)_{\neq}\partial_j \nabla\partial_{x_1}v_3\big\|_{L^2}
&\lesssim\big(\big\|\partial_j(v_j)_{\neq}\big\|_{L^2}+\big\|(v_j)_{\neq}\big\|_{L^2}\big)
\big\| \partial_{x_1}\nabla\Delta v_3\big\|_{L^2}\\
&\lesssim\big(\big\|\partial_{x_1}\omega_3\big\|_{L^2}+\big\|\partial_{x_2}\omega_3\big\|_{L^2}
+\big\|(\Delta v_3)_{\neq}\big\|_{L^2}\big)\big\| \partial_{x_1}\nabla\Delta v_3\big\|_{L^2}.
\end{aligned}
$$
For $j=3$, one gets
$$
\big\|\nabla\partial_{x_1}(v_j)_{\neq}\partial_j v_3\big\|_{L^2}\lesssim \big\|\partial_{x_1}(v_j)_{\neq}\big\|_{H^2}\big\|\partial_j v_3\big\|_{H^1}\lesssim \big\|\partial_{x_1}(\Delta v_3)_{\neq}\big\|_{L^2}\big\|\Delta v_3\big\|_{L^2},
$$

$$
\big\|\partial_{x_1}(v_j)_{\neq}\partial_j\nabla v_3\big\|_{L^2}
\lesssim\big\|\partial_{x_1}(v_j)_{\neq}\big\|_{L^\infty}\big\|\partial_j\nabla v_3\big\|_{L^2}
\lesssim \big\|\partial_{x_1}(\Delta v_3)_{\neq}\big\|_{L^2}\big\|\Delta v_3\big\|_{L^2},
$$

$$
\big\|\nabla(v_j)_{\neq}\partial_j \partial_{x_1}v_3\big\|_{L^2}
\lesssim\big\|(v_j)_{\neq}\big\|_{H^2}\big\|\partial_j \partial_{x_1}v_3\big\|_{H^1}
\lesssim\big\|(\Delta v_3)_{\neq}\big\|_{L^2}\big\|\partial_{x_1}\Delta v_3\big\|_{L^2},
$$

$$
\big\|(v_j)_{\neq}\partial_j \nabla\partial_{x_1}v_3\big\|_{L^2}
\lesssim\big\|(v_j)_{\neq}\big\|_{L^\infty}\big\|\partial_j \nabla\partial_{x_1}v_3\big\|_{L^2}
\lesssim\big\|(\Delta v_3)_{\neq}\big\|_{L^2}\big\|\partial_{x_1}\Delta v_3\big\|_{L^2}.
$$
Then we have
\begin{equation}\label{eq:4.35}
\begin{aligned}
J_1\lesssim&\big(\big\|\partial_{x_1}(\Delta v_3)_{\neq}\big\|_{L^2}+\big\|\partial_{x_2}(\Delta v_3)_{\neq}\big\|_{L^2}+\big\|\nabla\partial_{x_1}\omega_3\big\|_{L^2}+\big\|\nabla\partial_{x_2}\omega_3\big\|_{L^2}\big)
\big\|\Delta v_3\big\|_{L^2}\\
&+\big(\big\|\partial_{x_1}\omega_3\big\|_{L^2}+\big\|\partial_{x_2}\omega_3\big\|_{L^2}
+\big\|(\Delta v_3)_{\neq}\big\|_{L^2}\big)
\big\|\nabla \Delta v_3\big\|_{L^2}\\
&+\big(\big\|(\Delta v_3)_{\neq}\big\|_{L^2}+\big\|\nabla\partial_{x_1}\omega_3\big\|_{L^2}+\big\|\nabla\partial_{x_2}\omega_3\big\|_{L^2}\big)
\big\|\partial_{x_1}\Delta v_3\big\|_{L^2}\\
&+\big(\big\|\partial_{x_1}\omega_3\big\|_{L^2}+\big\|\partial_{x_2}\omega_3\big\|_{L^2}+\big\|(\Delta v_3)_{\neq}\big\|_{L^2}\big)\big\|\partial_{x_1}\nabla\Delta v_3\big\|_{L^2}\\
&+\big\|\partial_{x_1}(\Delta v_3)_{\neq}\big\|_{L^2}\big\|\Delta v_3\big\|_{L^2}+\big\|(\Delta v_3)_{\neq}\big\|_{L^2}\big\|\partial_{x_1}\Delta v_3\big\|_{L^2}.
\end{aligned}
\end{equation}
Similarly, one has
\begin{equation}\label{eq:4.36}
\begin{aligned}
J_2&=\sum_{j=1}^3\big\|\nabla\left(P_0v_j\partial_j\partial_{x_1} (v_3)_{\neq}\right)\big\|_{L^2}\\
&\lesssim\sum_{j=1}^3\big\|\nabla P_0v_j\partial_j \partial_{x_1}(v_3)_{\neq} \big\|_{L^2}+\sum_{j=1}^3\big\|P_0v_j\partial_j \nabla\partial_{x_1}(v_3)_{\neq} \big\|_{L^2}\\
&\lesssim \big\|P_0v_j\|_{H^2}\|\partial_j \partial_{x_1}(v_3)_{\neq}\big\|_{H^1}\lesssim \big\|V\big\|_{H^2}\big\|\partial_{x_1}(\Delta v_3)_{\neq}\big\|_{L^2}.
\end{aligned}
\end{equation}
Combining \eqref{eq:4.34}-\eqref{eq:4.36}, we have
\begin{equation}\label{eq:4.37}
\begin{aligned}
\big\|\partial_{x_1}\nabla(V\cdot\nabla v_3)\big\|_{L^2}\lesssim&\big(\big\|\partial_{x_1}(\Delta v_3)_{\neq}\big\|_{L^2}+\big\|\partial_{x_2}(\Delta v_3)_{\neq}\big\|_{L^2}+\big\|\nabla\partial_{x_1}\omega_3\big\|_{L^2}\\
&+\big\|\nabla\partial_{x_2}\omega_3\big\|_{L^2}\big)\big\|\Delta v_3\big\|_{L^2}+\big\|V\big\|_{H^2}\big\|\partial_{x_1}(\Delta v_3)_{\neq}\big\|_{L^2}\\
&+\big(\big\|\partial_{x_1}\omega_3\big\|_{L^2}+\big\|\partial_{x_2}\omega_3\big\|_{L^2}+\big\|(\Delta v_3)_{\neq}\big\|_{L^2}\big)\big\|\nabla \Delta v_3\big\|_{L^2}\\
&+\big(\big\|(\Delta v_3)_{\neq}\big\|_{L^2}+\big\|\nabla\partial_{x_1}\omega_3\big\|_{L^2}+\big\|\nabla\partial_{x_2}\omega_3\big\|_{L^2}\big)
\big\|\partial_{x_1}\Delta v_3\big\|_{L^2}\\
&+\big(\big\|\partial_{x_1}\omega_3\big\|_{L^2}+\big\|\partial_{x_2}\omega_3\big\|_{L^2}+\big\|(\Delta v_3)_{\neq}\big\|_{L^2}\big)\big\|\partial_{x_1}\nabla\Delta v_3\big\|_{L^2}\\
&+\big\|\partial_{x_1}(\Delta v_3)_{\neq}\big\|_{L^2}\big\|\Delta v_3\big\|_{L^2}+\big\|(\Delta v_3)_{\neq}\big\|_{L^2}\big\|\partial_{x_1}\Delta v_3\big\|_{L^2}\\
\lesssim&\big(\big\|\partial_{x_1}(\Delta v_3)_{\neq}\big\|_{L^2}+\big\|\partial_{x_2}(\Delta v_3)_{\neq}\big\|_{L^2}\big)\big(\big\|V\big\|_{H^2}+\big\|\partial_{x_1}V\big\|_{H^2}\big)\\
&+\big(\big\|\nabla\partial_{x_1}\omega_3\big\|_{L^2}+\big\|\nabla\partial_{x_2}\omega_3\big\|_{L^2}\big)
\big(\big\|v_3\big\|_{H^2}+\big\|\partial_{x_1}v_3\big\|_{H^2}\big)\\
&+\big(\big\|\partial_{x_1}\omega_3\big\|_{L^2}+\big\|\partial_{x_2}\omega_3\big\|_{L^2}+\big\|(\Delta v_3)_{\neq}\big\|_{L^2}\big)\\
&\cdot\big(\big\|\nabla\Delta v_3\big\|_{L^2}+\big\|\partial_{x_1}\nabla\Delta v_3\big\|_{L^2}\big).
\end{aligned}
\end{equation}
Similarly, we also have
\begin{equation}\label{eq:4.38}
\begin{aligned}
\big\|\partial_{x_2}\nabla(V\cdot\nabla v_3)\big\|_{L^2}
\lesssim&\big(\big\|\partial_{x_1}(\Delta v_3)_{\neq}\big\|_{L^2}+\big\|\partial_{x_2}(\Delta v_3)_{\neq}\big\|_{L^2}\big)\big(\big\|V\big\|_{H^2}+\big\|\partial_{x_2}V\big\|_{H^2}\big)\\
&+\big(\big\|\nabla\partial_{x_1}\omega_3\big\|_{L^2}+\big\|\nabla\partial_{x_2}\omega_3\big\|_{L^2}\big)
\big(\big\|v_3\big\|_{H^2}
+\big\|\partial_{x_2}v_3\big\|_{H^2}\big)\\
&+\big(\big\|\partial_{x_1}\omega_3\big\|_{L^2}+\big\|\partial_{x_2}\omega_3\big\|_{L^2}+\big\|(\Delta v_3)_{\neq}\big\|_{L^2}\big)\\
&\cdot\big(\big\|\nabla\Delta v_3\big\|_{L^2}+\big\|\partial_{x_2}\nabla\Delta v_3\big\|_{L^2}\big).
\end{aligned}
\end{equation}
\vskip .05in

\noindent {\bf Step 3. The estimate of $\partial_{x_1}(\omega_3V)$ and $\partial_{x_2}(\omega_3V)$.} By the definition of $\omega_3$, one has
$$
\partial_{x_1}(V\omega_3)=\partial_{x_1}V(\partial_{x_1}v_2-\partial_{x_2}v_1)+V\partial_{x_1}\omega_3,
$$
and by Lemma \ref{lem:2.5} and \ref{lem:4.3}, we obtain
\begin{equation}\label{eq:4.39}
\begin{aligned}
\big\|\partial_{x_1}(V\omega_3)\big\|_{L^2}&\lesssim\big\|\partial_{x_1}V\partial_{x_1}v_2\big\|_{L^2}
+\big\|\partial_{x_1}V\partial_{x_2}v_1\big\|_{L^2}+\big\|V\partial_{x_1}\omega_3\big\|_{L^2}\\
&\lesssim\big(\big\|\partial_{x_1}V\big\|_{L^2}+\big\|\partial^2_{x_1}V\big\|_{L^2}\big)\big\|\Delta v_2\big\|_{L^2}\\
&\ \ \ \ +\big(\big\|\partial_{x_1}V\big\|_{L^2}+\big\|\partial_{x_1}\partial_{x_2}V\big\|_{L^2}\big)\big\|\Delta v_1\big\|_{L^2}
+\big\|V\big\|_{L^\infty}\big\|\partial_{x_1}\omega_3\big\|_{L^2}\\
&\lesssim\big(\big\|\partial_{x_1}V\big\|_{L^2}+\big\|\partial^2_{x_1}V\big\|_{L^2}
+\big\|\partial_{x_1}\partial_{x_2}V\big\|_{L^2}\big)\\
&\ \ \ \cdot\big\|\Delta V\big\|_{L^2}
+\big\|\Delta V\big\|_{L^2}\big\|\partial_{x_1}\omega_3\big\|_{L^2}\!\!\!\!\!\\
&\lesssim \big(\big\|\partial^2_{x_1}V\big\|_{L^2}+\big\|\partial_{x_1}\partial_{x_2}V\big\|_{L^2}\big)
\big\|\Delta V\big\|_{L^2}+\big\|\Delta V\big\|_{L^2}\big\|\partial_{x_1}\omega_3\big\|_{L^2}\!\!\!\!\!\\
&\lesssim \big(\big\|\partial^2_{x_1}V\big\|_{L^2}+\big\|\partial_{x_1}\partial_{x_2}V\big\|_{L^2}
+\big\|\partial_{x_1}\omega_3\big\|_{L^2}\big)
\big\|\Delta V\big\|_{L^2}\\
&\lesssim\big(\big\|\partial_{x_1}\omega_3\big\|_{L^2}+\big\|\partial_{x_2}\omega_3\big\|_{L^2}+\big\|(\Delta v_3)_{\neq}\big\|_{L^2}\big)\big\|\Delta V\big\|_{L^2}\\
&\lesssim\big(\big\|\partial_{x_1}(\Delta v_3)_{\neq}\big\|_{L^2}+\big\|\partial_{x_2}(\Delta v_3)_{\neq}\big\|_{L^2}+\big\|\partial_{x_1}\omega_3\big\|_{L^2}
+\big\|\partial_{x_2}\omega_3\big\|_{L^2}\big)\big\|V\big\|_{H^2}.\!\!\!\!\!\!\!
\end{aligned}
\end{equation}
Similarly, one has
\begin{equation}\label{eq:4.40}
\begin{aligned}
\big\|\partial_{x_2}(V\omega_3)\big\|_{L^2}
\lesssim&\big(\big\|\partial_{x_1}(\Delta v_3)_{\neq}\big\|_{L^2}+\big\|\partial_{x_2}(\Delta v_3)_{\neq}\big\|_{L^2}\\
&+\big\|\partial_{x_1}\omega_3\big\|_{L^2}+\big\|\partial_{x_2}\omega_3\big\|_{L^2}\big)
\big\|V\big\|_{H^2}.\!\!\!\!
\end{aligned}
\end{equation}

\noindent {\bf Step 4. The estimate of $\partial_{x_1}(\omega v_3)$ and $\partial_{x_2}(\omega v_3)$.}  Since
$$
\partial_{x_1}(\omega v_3)=\partial_{x_1}\omega v_3+\omega \partial_{x_1}v_3,\ \ \ \partial_{x_1}v_3=\partial_{x_1}(v_3)_{\neq},
$$
then we deduce by Lemma \ref{lem:4.3} that
\begin{equation}\label{eq:4.41}
\begin{aligned}
\big\|\partial_{x_1}(\omega v_3)\big\|_{L^2}&\lesssim\big\|\partial_{x_1}\omega v_3\big\|_{L^2}+\big\|\omega \partial_{x_1}(v_3)_{\neq}\big\|_{L^2}\\
&\lesssim\big\|\partial_{x_1}\omega\big\|_{L^2}
\big\| v_3\big\|_{L^\infty}+\big\|\omega\big\|_{H^1}\big\|\partial_{x_1}(v_3)_{\neq}\big\|_{H^1}\\
&\lesssim \big(\big\|(\Delta v_3)_{\neq}\big\|_{L^2}+\big\|\partial_{x_1}\nabla\omega_3\big\|_{L^2}+\big\|\partial_{x_2}\nabla\omega_3\big\|_{L^2}\big)\\
&\ \ \cdot \big\|\Delta v_3\big\|_{L^2}
+\big\|V\big\|_{H^2}\big\|(v_3)_{\neq}\big\|_{H^2}\\
& \lesssim \big\|(\Delta v_3)_{\neq}\big\|_{L^2}\big\|V\big\|_{H^2}+\big(\big\|\partial_{x_1}\nabla\omega_3\big\|_{L^2}
+\big\|\partial_{x_2}\nabla\omega_3\big\|_{L^2}\big)\big\|v_3\big\|_{H^2}.
\end{aligned}
\end{equation}
Similarly, one has
\begin{equation}\label{eq:4.42}
\big\|\partial_{x_2}(\omega v_3)\big\|_{L^2}
 \lesssim \big\|(\Delta v_3)_{\neq}\big\|_{L^2}\big\|V\big\|_{H^2}+\big(\big\|\partial_{x_1}\nabla\omega_3\big\|_{L^2}
+\big\|\partial_{x_2}\nabla\omega_3\big\|_{L^2}\big)\big\|v_3\big\|_{H^2}.
\end{equation}

\vskip .05in

\noindent {\bf Step 5. The estimate of $\nabla(V\cdot\nabla V)$.} Since
$$
(V\cdot\nabla V)_{\neq}=(V\cdot\nabla V_{\neq})_{\neq}+V_{\neq}\cdot\nabla P_0V,
$$
then we get
\begin{equation}\label{eq:4.43}
\begin{aligned}
\big\|\nabla(V\cdot\nabla V)_{\neq}\big\|_{L^2}
&\lesssim\big\|\nabla(V\cdot\nabla V_{\neq})_{\neq}\big\|_{L^2}+\big\|\nabla(V_{\neq}\cdot\nabla P_0V)\big\|_{L^2}\\
&\lesssim \big\|V\big\|_{H^2}\big\|\nabla V_{\neq}\big\|_{H^1}+\big\|V_{\neq}\big\|_{H^2}\big\|\nabla P_0V\big\|_{H^1}\\
&\lesssim \big\|V\big\|_{H^2}\big\|\Delta V_{\neq}\big\|_{L^2}.
\end{aligned}
\end{equation}
Similarly, one has
\begin{equation}\label{eq:4.44}
\big\|V_{\neq}\cdot\nabla V_{\neq}\big\|_{H^1}\lesssim
\big\|V_{\neq}\big\|_{H^2}\big\|\Delta V_{\neq}\big\|_{L^2}\lesssim\big\|V\big\|_{H^2}\big\|\Delta V_{\neq}\big\|_{L^2}.
\end{equation}
By the similar argument with \eqref{eq:4.43} and \eqref{eq:4.44}, we obtain
\begin{equation}\label{eq:4.45}
\begin{aligned}
\big\|\partial_{x_1}\nabla(V\cdot\nabla V)\big\|_{L^2}
\lesssim&\big\|\nabla(\partial_{x_1}V\cdot\nabla V_{\neq})_{\neq}\big\|_{L^2}\\
&+\big\|\nabla(\partial_{x_1}V\cdot\nabla P_0V)\big\|_{L^2}
+\big\|\nabla(V\cdot\nabla \partial_{x_1}V)_{\neq}\big\|_{L^2}\\
\lesssim &\big\|\partial_{x_1}V\|_{H^2}\|\Delta V_{\neq}\big\|_{L^2}+\big\|V\|_{H^2}\|\partial_{x_1}\Delta V_{\neq}\big\|_{L^2},
\end{aligned}
\end{equation}
and
\begin{equation}\label{eq:4.46}
\begin{aligned}
\big\|\partial_{x_1}\nabla(V_{\neq}\cdot\nabla V_{\neq})\big\|_{L^2}
\lesssim \big\|\partial_{x_1}V\big\|_{H^2}\big\|\Delta V_{\neq}\big\|_{L^2}.
\end{aligned}
\end{equation}
Similarly, one also gets
\begin{equation}\label{eq:4.47}
\big\|\partial_{x_2}\nabla(V\cdot\nabla V)\big\|_{L^2}
\lesssim
\big\|\partial_{x_2}V\big\|_{H^2}\big\|\Delta V_{\neq}\big\|_{L^2}+\big\|V\big\|_{H^2}\big\|\partial_{x_2}\Delta V_{\neq}\big\|_{L^2},
\end{equation}
and
\begin{equation}\label{eq:4.48}
\big\|\partial_{x_2}\nabla(V_{\neq}\cdot\nabla V_{\neq})\big\|_{L^2}
\lesssim \big\|\partial_{x_2}V\big\|_{H^2}\big\|\Delta V_{\neq}\big\|_{L^2}.
\end{equation}

\vskip .05in

Now, we use the above estimates to complete the proof of lemma. Combining \eqref{eq:4.28}, \eqref{eq:4.29} and \eqref{eq:4.37}-\eqref{eq:4.42}, one has
$$
\begin{aligned}
&\sum_{i=1}^2\big(\big\|\partial_{x_i}(\Delta p)\big\|_{L^2}+\big\|\partial_{x_i}\nabla(V\cdot\nabla v_3)\big\|_{L^2}
+\big\|\partial_{x_i}(\omega_3V)\big\|_{L^2}+\big\|\partial_{x_i}(\omega v_3)\big\|_{L^2}\big)\\
\lesssim&\big(\big\|\partial_{x_1}(\Delta v_3)_{\neq}\big\|_{L^2}+\big\|\partial_{x_2}(\Delta v_3)_{\neq}\big\|_{L^2}+\big\|\partial_{x_1}\omega_3\big\|_{L^2}+\big\|\partial_{x_2}\omega_3\big\|_{L^2}\big)\\
&\cdot\big(\big\|V\big\|_{H^2}+\big\|\partial_{x_1}V\big\|_{H^2}+\big\|\partial_{x_2}V\big\|_{H^2}\big)\\
&+\big(\big\|\nabla\partial_{x_1}\omega_3\big\|_{L^2}+\big\|\nabla\partial_{x_2}\omega_3\big\|_{L^2}\big)
\big(\big\|v_3\big\|_{H^2}+\big\|\partial_{x_1}v_3\big\|_{H^2}+\big\|\partial_{x_2}v_3\big\|_{H^2}\big)\\
&+\big(\big\|\partial_{x_1}\omega_3\big\|_{L^2}+\big\|\partial_{x_2}\omega_3\big\|_{L^2}+\big\|(\Delta v_3)_{\neq}\big\|_{L^2})\\
&\cdot\big(\big\|\nabla\Delta v_3\big\|_{L^2}+\big\|\partial_{x_1}\nabla\Delta v_3\big\|_{L^2}+\big\|\partial_{x_2}\nabla\Delta v_3\big\|_{L^2}\big).
\end{aligned}
$$
Combining \eqref{eq:4.26} and \eqref{eq:4.33}, to obtain
$$
\begin{aligned}
\big\|\Delta p\big\|_{L^2}+\big\|\nabla(V\cdot\nabla v_3)\big\|_{L^2}\lesssim
&\big(\big\|\nabla\partial_{x_1}\omega_3\big\|_{L^2}+\big\|\nabla\partial_{x_2}\omega_3\big\|_{L^2}
+\big\|\nabla\Delta v_3\big\|_{L^2}\big)\big\|\Delta v_3\big\|_{L^2}\\
&+\big(\big\|\partial_{x_1}\omega_3\big\|_{L^2}+\big\|\partial_{x_2}\omega_3\big\|_{L^2}
+\big\|(\Delta v_3)_{\neq}\big\|_{L^2}\big)\\
&\cdot\big(\big\|V\big\|_{H^2}+\big\|\nabla\Delta v_3\big\|_{L^2}\big).
\end{aligned}
$$
Similarly, combining \eqref{eq:4.43}-\eqref{eq:4.48}, we have
$$
\big\|\nabla(V\cdot\nabla V)_{\neq}\big\|_{L^2}+\big\|V_{\neq}\cdot\nabla V_{\neq}\big\|_{H^1}\lesssim\big\|V\big\|_{H^2}\big\|\Delta V_{\neq}\big\|_{L^2},
$$

$$
\begin{aligned}
&\big\|\partial_{x_1}\nabla(V\cdot\nabla V)\big\|_{L^2}+\big\|\partial_{x_1}\nabla(V_{\neq}\cdot\nabla V_{\neq})\big\|_{L^2}\\
\lesssim &\big\|\partial_{x_1}V\big\|_{H^2}\big\|\Delta V_{\neq}\big\|_{L^2}+\big\|V\big\|_{H^2}\big\|\partial_{x_1}\Delta V_{\neq}\big\|_{L^2},
\end{aligned}
$$
and
$$
\begin{aligned}
&\big\|\partial_{x_2}\nabla(V\cdot\nabla V)\big\|_{L^2}+\big\|\partial_{x_2}\nabla(V_{\neq}\cdot\nabla V_{\neq})\big\|_{L^2}\\
\lesssim &\big\|\partial_{x_2}V\big\|_{H^2}\big\|\Delta V_{\neq}\big\|_{L^2}+\big\|V\big\|_{H^2}\big\|\partial_{x_2}\Delta V_{\neq}\big\|_{L^2}.
\end{aligned}
$$
This completes the proof of Lemma \ref{lem:4.4}.
\end{proof}

\subsection{Global nonlinear stability}
Here we establish the nonlinear stability and give the threshold for \eqref{eq:1.1} around planar helical flow. In fact, we consider the perturbation equation \eqref{eq:1.3}, see also \eqref{eq:1.5} or \eqref{eq:4.1}. We denote
$$
\begin{aligned}
\mathcal{M}_1=&\big\|P_0\Delta v_3\big\|^2_{X_1}+\big\|\partial_{x_1}\Delta v_3\big\|^2_{X_1}+\big\|\partial_{x_2}\Delta v_3\big\|^2_{X_1}
+\nu^{-1}\left\|e^{\epsilon\nu^{1/2}t}\partial_{x_1}\Delta p\right\|^2_{L^2 L^2}\\
&+\nu^{-1}\left\|e^{\epsilon\nu^{1/2}t}\partial_{x_2}\Delta p\right\|^2_{L^2 L^2}
+\big\|\partial_{x_1}(\Delta v_3)\big\|^2_{X_{ed}}\\
&+\big\|\partial_{x_2}(\Delta v_3)\big\|^2_{X_{ed}}+\big\|\partial_{x_1}\omega_3\big\|^2_{X_{ed}}
+\big\|\partial_{x_2}\omega_3\big\|^2_{X_{ed}},\\[2MM]
\mathcal{M}_2=&\big\|\partial_{x_1}\Delta V\big\|^2_{X_{ed}}+\big\|\partial_{x_2}\Delta V\big\|^2_{X_{ed}},
\end{aligned}
$$
and let us recall the following notations
$$
\begin{aligned}
\mathcal{E}_1(T)=&\sup_{0\leq t\leq T}\left[\big\|v_3\big\|_{X_0}+e^{\epsilon\nu^{1/2}t}\left(\big\|\partial_{x_1}(\Delta v_3)\big\|_{L^2}+\big\|\partial_{x_2}(\Delta v_3)\big\|_{L^2}\right)\right]\\
&+\sup_{0\leq t\leq T}\left[e^{\epsilon\nu^{1/2}t}\left(\big\|\partial_{x_1}\omega_3\big\|_{L^2}
+\big\|\partial_{x_2}\omega_3\big\|_{L^2}\right)\right],\\[2MM]
\mathcal{E}_2(T)=&\sup_{0\leq t\leq T}\big\|V(t)\big\|_{X_0}.
\end{aligned}
$$
Here the norm  $\|\cdot\|_{X_0}, \|\cdot\|_{X_{ed}}$ and $\|\cdot\|_{X_1}$ are defined in \eqref{eq:1.8}, \eqref{eq:1.15} and \eqref{eq:2.1}.

\vskip .05in

Next, we establish the following estimates based on the smallness assumption of $\mathcal{E}_1$ and $\mathcal{E}_2$,
which are very important to obtain global stability by continuity method.
\begin{proposition}\label{prop:4.5}
Let $\delta>1$, there exist positive constant $\varepsilon_1\in (0,1)$ and $C_0$, such that for any $0<\nu\ll1, T>0$, if the solution $V\in C([0,T],X_0(\mathbb{T}^3))$ and $\nabla V\in L^2([0,T],X_0(\mathbb{T}^3))$ of \eqref{eq:1.3} satisfies
$$
\mathcal{E}_1(T)\leq\varepsilon_1\nu,\ \ \ \ \mathcal{E}_2(T)\leq\varepsilon_1\nu^{3/4}.
$$
Then we have

$$
\mathcal{M}_1\lesssim \big\|V_0\big\|^2_{X_0},\ \ \ \mathcal{M}_2\lesssim\nu^{-3/2}\big\|V_0\big\|^2_{X_0},
$$
and
$$
\mathcal{E}_1(T)\leq C_0\big\|V_0\big\|_{X_0},\ \ \ \ \mathcal{E}_2(T)\leq C_0\nu^{-1}\big\|V_0\big\|_{X_0}.
$$
\end{proposition}

\begin{proof}
We finish the proof by following four steps.

\vskip .05in

\noindent {\bf Step 1. The estimate of $\mathcal{M}_1$.} Considering the \eqref{eq:4.1}-\eqref{eq:4.3}, we can deduce by Proposition \ref{prop:4.1} and \ref{prop:4.2} that
\begin{equation}\label{eq:4.49}
\begin{aligned}
&\big\|\partial_{x_1}(\Delta v_3)\big\|^2_{X_{ed}}+\big\|\partial_{x_2}(\Delta v_3)\big\|^2_{X_{ed}}
+\big\|\partial_{x_1}\omega_3\big\|^2_{X_{ed}}+\big\|\partial_{x_2}\omega_3\big\|^2_{X_{ed}}\\
&+\nu^{-1}\left\|e^{\epsilon\nu^{1/2}t}\partial_{x_1}\Delta p\right\|^2_{L^2 L^2}+\nu^{-1}\left\|e^{\epsilon\nu^{1/2}t}\partial_{x_2}\Delta p\right\|^2_{L^2 L^2}\\
\lesssim&\big\|\partial_{x_1}(\Delta v_3)(0)\big\|^2_{L^2}+\big\|\partial_{x_2}(\Delta v_3)(0)\big\|^2_{L^2}
+\big\|\partial_{x_1}\omega_3(0)\big\|^2_{L^2}+\big\|\partial_{x_2}\omega_3(0)\big\|^2_{L^2}\\
&+\nu^{-1}\left(\left\|e^{\epsilon\nu^{1/2}t}\partial_{x_1}f\right\|^2_{L^2L^2}
+\left\|e^{\epsilon\nu^{1/2}t}\partial_{x_1}g\right\|^2_{L^2L^2}\right)
+\nu^{-1}\left\|e^{\epsilon\nu^{1/2}t}\partial_{x_1}\Delta p\right\|^2_{L^2 L^2}\\
&+\nu^{-1}\left(\left\|e^{\epsilon\nu^{1/2}t}\partial_{x_2}f\right\|^2_{L^2L^2}
+\left\|e^{\epsilon\nu^{1/2}t}\partial_{x_2}g\right\|^2_{L^2L^2}\right)
+\nu^{-1}\left\|e^{\epsilon\nu^{1/2}t}\partial_{x_2}\Delta p\right\|^2_{L^2 L^2}.
\end{aligned}
\end{equation}
Combining \eqref{eq:4.2}, \eqref{eq:4.3} and Lemma \ref{lem:4.4}, we obtain
\begin{equation}\label{eq:4.50}
\begin{aligned}
&\left\|e^{\epsilon\nu^{1/2}t}\partial_{x_1}f\right\|^2_{L^2L^2}
+\left\|e^{\epsilon\nu^{1/2}t}\partial_{x_1}g\right\|^2_{L^2L^2}
+\left\|e^{\epsilon\nu^{1/2}t}\partial_{x_2}f\right\|^2_{L^2L^2}
+\left\|e^{\epsilon\nu^{1/2}t}\partial_{x_2}g\right\|^2_{L^2L^2}\\
&+\left\|e^{\epsilon\nu^{1/2}t}\partial_{x_1}\Delta p\right\|^2_{L^2 L^2}
+\left\|e^{\epsilon\nu^{1/2}t}\partial_{x_2}\Delta p\right\|^2_{L^2 L^2}\\
\lesssim&\sum_{i=1}^2\big(\left\|e^{\epsilon\nu^{1/2}t}\partial_{x_i}(\Delta p)\right\|^2_{L^2L^2}+\left\|e^{\epsilon\nu^{1/2}t}\partial_{x_i}\nabla(V\cdot\nabla v_3)\right\|^2_{L^2L^2}\\
&+\left\|e^{\epsilon\nu^{1/2}t}\partial_{x_i}(\omega_3V)\right\|^2_{L^2L^2}
+\left\|e^{\epsilon\nu^{1/2}t}\partial_{x_i}(\omega v_3)\right\|^2_{L^2L^2}\big)\\
\lesssim&\int_0^T e^{2\epsilon\nu^{1/2}t}\left(\big\|\partial_{x_1}(\Delta v_3)_{\neq}\big\|_{L^2}+\big\|\partial_{x_2}(\Delta v_3)_{\neq}\big\|_{L^2}+\big\|\partial_{x_1}\omega_3\big\|_{L^2}+\big\|\partial_{x_2}\omega_3\big\|_{L^2}\right)^2\\
&\cdot\big(\big\|V\big\|_{H^2}+\big\|\partial_{x_1}V\big\|_{H^2}+\big\|\partial_{x_2}V\big\|_{H^2}\big)^2dt\\
&+\int_0^{T} e^{2\epsilon\nu^{1/2}t}\big(\big\|\nabla\partial_{x_1}\omega_3\big\|_{L^2}+\big\|\nabla\partial_{x_2}\omega_3\big\|_{L^2}\big)^2
\big(\big\|v_3\big\|_{H^2}
+\big\|\partial_{x_1}v_3\big\|_{H^2}+\big\|\partial_{x_2}v_3\big\|_{H^2}\big)^2dt\\
&+\int_0^T e^{2\epsilon\nu^{1/2}t}\big(\big\|\partial_{x_1}\omega_3\big\|_{L^2}+\big\|\partial_{x_2}\omega_3\big\|_{L^2}
+\big\|(\Delta v_3)_{\neq}\big\|_{L^2}\big)^2\\
&\cdot\big(\big\|\nabla\Delta v_3\big\|_{L^2}+\big\|\partial_{x_1}\nabla\Delta v_3\big\|_{L^2}+\big\|\partial_{x_2}\nabla\Delta v_3\big\|_{L^2}\big)^2dt.
\end{aligned}
\end{equation}
We write $\mathcal{E}_1=\mathcal{E}_1(T), \mathcal{E}_2=\mathcal{E}_2(T)$, and combining the definition $\|\cdot\|_{X_0}$, $\|\cdot\|_{X_{ed}}$, $\mathcal{E}_1$ and $\mathcal{E}_2$, one gets 
$$
\begin{aligned}
&\int_0^T e^{2\epsilon\nu^{1/2}t}\big(\big\|\partial_{x_1}(\Delta v_3)_{\neq}\big\|_{L^2}+\big\|\partial_{x_2}(\Delta v_3)_{\neq}\big\|_{L^2}+\big\|\partial_{x_1}\omega_3\big\|_{L^2}+\big\|\partial_{x_2}\omega_3\big\|_{L^2}\big)^2\\
&\cdot\big(\big\|V\big\|_{H^2}+\big\|\partial_{x_1}V\big\|_{H^2}+\big\|\partial_{x_2}V\big\|_{H^2}\big)^2dt\\
&+\int_0^T e^{2\epsilon\nu^{1/2}t}\big(\big\|\nabla\partial_{x_1}\omega_3\big\|_{L^2}+\big\|\nabla\partial_{x_2}\omega_3\big\|_{L^2}\big)^2
\big(\big\|v_3\big\|_{H^2}
+\big\|\partial_{x_1}v_3\big\|_{H^2}+\big\|\partial_{x_2}v_3\big\|_{H^2}\big)^2dt\\
 \lesssim &\left(\left\|e^{\epsilon\nu^{1/2}t}\partial_{x_1}(\Delta v_3)_{\neq}\right\|^2_{L^2L^2}+\left\|e^{\epsilon\nu^{1/2}t}\partial_{x_2}(\Delta v_3)_{\neq}\right\|^2_{L^2L^2}\right)\big\|V\big\|^2_{L^\infty X_0}\\
 &+\left(\left\|e^{\epsilon\nu^{1/2}t}\partial_{x_1}\omega_3\right\|^2_{L^2L^2}
 +\left\|e^{\epsilon\nu^{1/2}t}\partial_{x_2}\omega_3\right\|^2_{L^2L^2}\right)\big\|V\big\|^2_{L^\infty X_0}\\
 &+\left(\left\|e^{\epsilon\nu^{1/2}t}\nabla\partial_{x_1}\omega_3\right\|^2_{L^2L^2}
 +\left\|e^{\epsilon\nu^{1/2}t}\nabla\partial_{x_2}\omega_3\right\|^2_{L^2L^2}\right)\big\|v_3\big\|^2_{L^\infty X_0}\\
\lesssim &\nu^{-1/2}\big(\big\|\partial_{x_1}(\Delta v_3)_{\neq}\big\|^2_{X_{ed}}+\big\|\partial_{x_2}(\Delta v_3)_{\neq}\big\|^2_{X_{ed}}+\big\|\partial_{x_1}\omega_3\big\|^2_{X_{ed}}
+\big\|\partial_{x_2}\omega_3\big\|^2_{X_{ed}}\big)\mathcal{E}^2_2\\
&+\nu^{-1}\big(\big\|\partial_{x_1}\omega_3\big\|^2_{X_{ed}}
+\big\|\partial_{x_2}\omega_3\big\|^2_{X_{ed}}\big)\mathcal{E}^2_1,
\end{aligned}
$$
and
$$
\begin{aligned}
&\int_0^T e^{2\epsilon\nu^{1/2}t}\big(\big\|\partial_{x_1}\omega_3\big\|_{L^2}+\big\|\partial_{x_2}\omega_3\big\|_{L^2}
+\big\|(\Delta v_3)_{\neq}\big\|_{L^2}\big)^2\\
& \ \cdot\big(\big\|\nabla\Delta v_3\big\|_{L^2}+\big\|\partial_{x_1}\nabla\Delta v_3\big\|_{L^2}+\big\|\partial_{x_2}\nabla\Delta v_3\big\|_{L^2}\big)^2dt\\
\lesssim&\big(\left\|e^{\epsilon\nu^{1/2}t}\partial_{x_1}\omega_3\right\|^2_{L^\infty L^2}+\left\|e^{\epsilon\nu^{1/2}t}\partial_{x_2}\omega_3\right\|^2_{L^\infty L^2}\\
&+\left\|e^{\epsilon\nu^{1/2}t}\partial_{x_1}(\Delta v_3)_{\neq}\right\|^2_{L^\infty L^2}
+\left\|e^{\epsilon\nu^{1/2}t}\partial_{x_2}(\Delta v_3)_{\neq}\right\|^2_{L^\infty L^2}\big)\\
&\ \ \cdot\big(\big\|P_0\nabla\Delta v_3\big\|^2_{L^2L^2}+\big\|\partial_{x_1}\nabla\Delta v_3\big\|^2_{L^2L^2}+\big\|\partial_{x_2}\nabla\Delta v_3\big\|^2_{L^2L^2}\big)\\
\lesssim& \nu^{-1}\big(\big\|P_0\Delta v_3\big\|^2_{X_1}+\big\|\partial_{x_1}\Delta v_3\big\|^2_{X_1}+\big\|\partial_{x_2}\Delta v_3\big\|^2_{X_1}\big)\mathcal{E}^2_1.
\end{aligned}
$$
Obviously, we know that
$$
\big\|\partial_{x_1}(\Delta v_3)_{\neq}(0)\big\|^2_{L^2}+\big\|\partial_{x_2}(\Delta v_3)_{\neq}(0)\big\|^2_{L^2}
+\big\|\partial_{x_1}\omega_3(0)\big\|^2_{L^2}+\big\|\partial_{x_2}\omega_3(0)\big\|^2_{L^2}\lesssim \big\|V_0\big\|^2_{X_0}.
$$
Then combining \eqref{eq:4.49} and \eqref{eq:4.50}, we have
\begin{equation}\label{eq:4.51}
\begin{aligned}
&\big\|\partial_{x_1}(\Delta v_3)_{\neq}\big\|^2_{X_{ed}}+\big\|\partial_{x_2}(\Delta v_3)_{\neq}\big\|^2_{X_{ed}}
+\big\|\partial_{x_1}\omega_3\big\|^2_{X_{ed}}+\big\|\partial_{x_2}\omega_3\big\|^2_{X_{ed}}\\
&+\nu^{-1}\left\|e^{\epsilon\nu^{1/2}t}\partial_{x_1}\Delta p\right\|^2_{L^2 L^2}+\nu^{-1}\left\|e^{\epsilon\nu^{1/2}t}\partial_{x_2}\Delta p\right\|^2_{L^2 L^2}\\
\leq&C\big\|V_0\big\|^2_{X_0}+C\nu^{-2}\big(\big\|P_0\Delta v_3\big\|^2_{X_1}+\big\|\partial_{x_1}\Delta v_3\big\|^2_{X_1}+\big\|\partial_{x_2}\Delta v_3\big\|^2_{X_1}\big)\mathcal{E}^2_1\\
&+C\nu^{-3/2}\big(\big\|\partial_{x_1}(\Delta v_3)_{\neq}\big\|^2_{X_{ed}}+\big\|\partial_{x_2}(\Delta v_3)_{\neq}\big\|^2_{X_{ed}}+\big\|\partial_{x_1}\omega_3\big\|^2_{X_{ed}}
+\big\|\partial_{x_2}\omega_3\big\|^2_{X_{ed}}\big)\mathcal{E}^2_2\\
&+C\nu^{-2}\big(\big\|\partial_{x_1}\omega_3\big\|^2_{X_{ed}}+\big\|\partial_{x_2}\omega_3\big\|^2_{X_{ed}}\big)\mathcal{E}^2_1.
\end{aligned}
\end{equation}
Next, we consider
$$
\big\|P_0\Delta v_3\big\|^2_{X_1}+\big\|\partial_{x_1}\Delta v_3\big\|^2_{X_1}+\big\|\partial_{x_2}\Delta v_3\big\|^2_{X_1}.
$$
Applying $P_0$ to the first equation of \eqref{eq:4.1}, we can easily get
$$
(\partial_t-\nu \Delta)P_0\Delta v_3=\div P_0 f,
$$
then we deduce by energy estimate that
$$
\big\|P_0\Delta v_3\big\|^2_{X_1}\lesssim\big\|P_0\Delta v_3(0)\big\|^2_{L^2}+\nu^{-1}\big\|P_0f\big\|^2_{L^2L^2}.
$$
Combining \eqref{eq:4.2} and Lemma \ref{lem:4.4}, one has
$$
\begin{aligned}
\big\|f\big\|^2_{L^2L^2}&\lesssim \big\|\nabla(V\cdot\nabla v_3)\big\|^2_{L^2L^2}
+\big\|\Delta p\big\|^2_{L^2L^2}\\
&\lesssim\int_0^{T}\big(\big\|\nabla\partial_{x_1}\omega_3\big\|^2_{L^2}+\big\|\nabla\partial_{x_2}\omega_3\big\|^2_{L^2}
+\big\|\nabla\Delta v_3\big\|^2_{L^2}\big)\big\|\Delta v_3\big\|^2_{L^2}dt\\
&\ \ \ +\int_0^{T}\big(\big\|\partial_{x_1}\omega_3\big\|^2_{L^2}+\big\|\partial_{x_2}\omega_3\big\|^2_{L^2}
+\big\|(\Delta v_3)_{\neq}\big\|^2_{L^2}\big)\cdot\big(\big\|V\big\|^2_{H^2}+\big\|\nabla\Delta v_3\big\|^2_{L^2}\big)dt.
\end{aligned}
$$
By the definition of $\|\cdot\|_{X_0}$, $\|\cdot\|_{X_{ed}}$, $\mathcal{E}_1$ and $\mathcal{E}_2$, one gets
$$
\begin{aligned}
&\int_0^{T}\big(\big\|\nabla\partial_{x_1}\omega_3\big\|^2_{L^2}+\big\|\nabla\partial_{x_2}\omega_3\big\|^2_{L^2}
+\big\|\nabla\Delta v_3\big\|^2_{L^2}\big)\big\|\Delta v_3\big\|^2_{L^2}dt\\
\lesssim&\big(\big\|\nabla\partial_{x_1}\omega_3\big\|^2_{L^2L^2}+\big\|\nabla\partial_{x_2}\omega_3\big\|^2_{L^2L^2}
+\big\|\nabla\Delta v_3\big\|^2_{L^2L^2}\big)\big\|v_3\big\|^2_{L^\infty X_0}\\
\lesssim&\nu^{-1}\big(\big\|\partial_{x_1}\omega_3\big\|^2_{X_{ed}}+\big\|\partial_{x_2}\omega_3\big\|^2_{X_{ed}}
+\big\|P_0\Delta v_3\big\|^2_{X_1}+\big\|\partial_{x_1}\Delta v_3\big\|^2_{X_1}+\big\|\partial_{x_2}\Delta v_3\big\|^2_{X_1}\big)\mathcal{E}^2_1,
\end{aligned}
$$
and
$$
\begin{aligned}
&\int_0^{T}\big(\big\|\partial_{x_1}\omega_3\big\|^2_{L^2}+\big\|\partial_{x_2}\omega_3\big\|^2_{L^2}
+\big\|(\Delta v_3)_{\neq}\big\|^2_{L^2}\big)\cdot\big(\big\|V\big\|^2_{H^2}+\big\|\nabla\Delta v_3\big\|^2_{L^2}\big)dt\\
\lesssim&\big(\big\|\partial_{x_1}\omega_3\big\|^2_{L^2L^2}+\big\|\partial_{x_2}\omega_3\big\|^2_{L^2L^2}
+\big\|\partial_{x_1}(\Delta v_3)_{\neq}\big\|^2_{L^2L^2}+\big\|\partial_{x_2}(\Delta v_3)_{\neq}\big\|^2_{L^2L^2}\big)\big\|V\big\|^2_{L^\infty X_0}\\
&+\big(\big\|\partial_{x_1}\omega_3\big\|^2_{L^\infty L^2}+\big\|\partial_{x_2}\omega_3\big\|^2_{L^\infty L^2}
+\big\|\partial_{x_1}(\Delta v_3)_{\neq}\big\|^2_{L^\infty L^2}+\big\|\partial_{x_2}(\Delta v_3)_{\neq}\big\|^2_{L^\infty L^2}\big)\big\|\nabla\Delta v_3\big\|^2_{L^2 L^2}\\
\lesssim &\nu^{-1/2}\big(\big\|\partial_{x_1}(\Delta v_3)_{\neq}\big\|^2_{X_{ed}}+\big\|\partial_{x_2}(\Delta v_3)_{\neq}\big\|^2_{X_{ed}}
+\big\|\partial_{x_1}\omega_3\big\|^2_{X_{ed}}+\big\|\partial_{x_2}\omega_3\big\|^2_{X_{ed}}\big)\mathcal{E}^2_2\\
&+\nu^{-1}\big(\big\|P_0\Delta v_3\big\|^2_{X_1}+\big\|\partial_{x_1}\Delta v_3\big\|^2_{X_1}+\big\|\partial_{x_2}\Delta v_3\big\|^2_{X_1}\big)\mathcal{E}^2_1.
\end{aligned}
$$
Thus, we have
\begin{equation}\label{eq:4.52}
\begin{aligned}
\big\|P_0\Delta v_3\big\|^2_{X_1}
\lesssim& \big\|\Delta v_3(0)\big\|^2_{L^2}+\nu^{-1}\big\|f\big\|^2_{L^2L^2}\\
\lesssim & \big\|V_0\big\|^2_{X_0}+\nu^{-2}\big(\big\|\partial_{x_1}\omega_3\big\|^2_{X_{ed}}
+\big\|\partial_{x_2}\omega_3\big\|^2_{X_{ed}}\\
&+\big\|P_0\Delta v_3\big\|^2_{X_1}+\big\|\partial_{x_1}\Delta v_3\big\|^2_{X_1}+\big\|\partial_{x_2}\Delta v_3\big\|^2_{X_1}\big)\mathcal{E}^2_1\\
&+\nu^{-3/2}\big(\big\|\partial_{x_1}(\Delta v_3)_{\neq}\big\|^2_{X_{ed}}+\big\|\partial_{x_2}(\Delta v_3)_{\neq}\big\|^2_{X_{ed}}
+\big\|\partial_{x_1}\omega_3\big\|^2_{X_{ed}}\\
&+\big\|\partial_{x_2}\omega_3\big\|^2_{X_{ed}}\big)\mathcal{E}^2_2
+\nu^{-2}\big(\big\|P_0\Delta v_3\big\|^2_{X_1}+\big\|\partial_{x_1}\Delta v_3\big\|^2_{X_1}+\big\|\partial_{x_2}\Delta v_3\big\|^2_{X_1}\big)\mathcal{E}^2_1.
\end{aligned}
\end{equation}
Since
$$
\big\|P_0\Delta v_3\big\|^2_{X_1}+\big\|\partial_{x_1}\Delta v_3\big\|^2_{X_1}+\big\|\partial_{x_2}\Delta v_3\big\|^2_{X_1}
\lesssim\big\|P_0\Delta v_3\big\|^2_{X_1}+\big\|\partial_{x_1}\Delta v_3\big\|^2_{X_{ed}}+\big\|\partial_{x_2}\Delta v_3\big\|^2_{X_{ed}},
$$
and combining \eqref{eq:4.51} and \eqref{eq:4.52}, we have
$$
\begin{aligned}
&\big\|P_0\Delta v_3\big\|^2_{X_1}+\big\|\partial_{x_1}\Delta v_3\big\|^2_{X_1}+\big\|\partial_{x_2}\Delta v_3\big\|^2_{X_1}+\big\|\partial_{x_1}(\Delta v_3)_{\neq}\big\|^2_{X_{ed}}\\
&+\big\|\partial_{x_2}(\Delta v_3)_{\neq}\big\|^2_{X_{ed}}+\big\|\partial_{x_1}\omega_3\big\|^2_{X_{ed}}
+\big\|\partial_{x_2}\omega_3\big\|^2_{X_{ed}}\\
&+\nu^{-1}\left\|e^{\epsilon\nu^{1/2}t}\partial_{x_1}\Delta p\right\|^2_{L^2 L^2}+\nu^{-1}\left\|e^{\epsilon\nu^{1/2}t}\partial_{x_2}\Delta p\right\|^2_{L^2 L^2}\\
\leq& C\big\|V_0\big\|^2_{X_0}+C\big(\nu^{-3/2}\mathcal{E}_2^2+\nu^{-2}\mathcal{E}_1^2\big)
\big(\big\|\partial_{x_1}\omega_3\big\|^2_{X_{ed}}
+\big\|\partial_{x_2}\omega_3\big\|^2_{X_{ed}}\big)\\
&+C\nu^{-2}\mathcal{E}_1^2\big(\big\|P_0\Delta v_3\big\|^2_{X_1}+\big\|\partial_{x_1}\Delta v_3\big\|^2_{X_1}+\big\|\partial_{x_2}\Delta v_3\big\|^2_{X_1}\big)\\
&+C\nu^{-3/2}\mathcal{E}_2^2\big(\big\|\partial_{x_1}(\Delta v_3)_{\neq}\big\|^2_{X_{ed}}+\big\|\partial_{x_2}(\Delta v_3)_{\neq}\big\|^2_{X_{ed}}\big)\\
\leq& C\big\|V_0\big\|^2_{X_0}+C\big(\nu^{-2}\mathcal{E}_1^2+\nu^{-3/2}\mathcal{E}_2^2\big)\big(\big\|P_0\Delta v_3\big\|^2_{X_1}+\big\|\partial_{x_1}\Delta v_3\big\|^2_{X_1}+\big\|\partial_{x_2}\Delta v_3\big\|^2_{X_1}\\
&+\big\|\partial_{x_1}(\Delta v_3)_{\neq}\big\|^2_{X_{ed}}+\big\|\partial_{x_2}(\Delta v_3)_{\neq}\big\|^2_{X_{ed}}+\big\|\partial_{x_1}\omega_3\big\|^2_{X_{ed}}
+\big\|\partial_{x_2}\omega_3\big\|^2_{X_{ed}}\big).
\end{aligned}
$$
Then there exists $\varepsilon_1\in (0,1)$, for $\mathcal{E}_1\leq\varepsilon_1\nu, \mathcal{E}_2\leq\varepsilon_1\nu^{3/4}$, such that
$$
C\big(\nu^{-2}\mathcal{E}_1^2+\nu^{-3/2}\mathcal{E}_2^2\big)\leq1/2,
$$
thus we get
$$
\begin{aligned}
&\big\|P_0\Delta v_3\big\|^2_{X_1}+\big\|\partial_{x_1}\Delta v_3\big\|^2_{X_1}+\big\|\partial_{x_2}\Delta v_3\big\|^2_{X_1}
+\nu^{-1}\left\|e^{\epsilon\nu^{1/2}t}\partial_{x_1}\Delta p\right\|^2_{L^2 L^2}\\
&+\nu^{-1}\left\|e^{\epsilon\nu^{1/2}t}\partial_{x_2}\Delta p\right\|^2_{L^2 L^2}
+\big\|\partial_{x_1}(\Delta v_3)_{\neq}\big\|^2_{X_{ed}}\\
&+\big\|\partial_{x_2}(\Delta v_3)_{\neq}\big\|^2_{X_{ed}}+\big\|\partial_{x_1}\omega_3\big\|^2_{X_{ed}}
+\big\|\partial_{x_2}\omega_3\big\|^2_{X_{ed}}\lesssim \big\|V_0\big\|^2_{X_0}.
\end{aligned}
$$

\vskip .05in

\noindent {\bf Step 2. The estimate of $\mathcal{M}_2$.} Here we mainly consider
$$
\big\|\partial_{x_1}(\Delta v_1)\big\|^2_{X_{ed}}+\big\|\partial_{x_2}(\Delta v_1)\big\|^2_{X_{ed}}
+\big\|\partial_{x_1}(\Delta v_2)\big\|^2_{X_{ed}}+\big\|\partial_{x_2}(\Delta v_2)\big\|^2_{X_{ed}}.
$$
From the first equation in \eqref{eq:1.3}, we obtain that
\begin{equation}\label{eq:4.53}
\left(\partial_t+\mathcal{H}\right)\Delta v_1=\div F,
\end{equation}
where
$$
F=-\nabla\left(\delta\cos(m_0 y)v_3\right)+(F_1,F_2,F_3)-\nabla(V\cdot\nabla v_1),
$$
and
$$
\begin{aligned}
F_1&=-\Delta p+\sin (m_0 y)v_1+2\delta\cos(m_0 y)\partial_{x_1}v_3-2\delta\sin(m_0 y)\partial_{x_2}v_3,\\
F_2&=\cos (m_0 y)v_1,\\
F_3&=2\delta\sin(m_0 y)\partial_{x_2}v_1-2\delta\cos (m_0 y)\partial_{x_1}v_1.
\end{aligned}
$$
Similarly, one has
\begin{equation}\label{eq:4.54}
\left(\partial_t+\mathcal{H}\right)\Delta v_2=\div G,
\end{equation}
where
$$
G=\nabla\left(\delta\sin(m_0 y)v_3\right)+(G_1,G_2,G_3)-\nabla(V\cdot\nabla v_2),
$$
and
$$
\begin{aligned}
G_1&=\sin (m_0 y)v_2,\\
G_2&=-\Delta p+\cos (m_0 y)v_2+2\delta\cos(m_0 y)\partial_{x_1}v_3-2\delta\sin(m_0 y)\partial_{x_2}v_3,\\
G_3&=2\delta\sin (m_0 y)\partial_{x_2}v_2-2\delta\cos (m_0 y)\partial_{x_1}v_2.
\end{aligned}
$$
Applying the operator $\partial_{x_1}$ to \eqref{eq:4.53} and \eqref{eq:4.54}, one has
$$
\left(\partial_t+\mathcal{H}\right)\partial_{x_1}\Delta v_1=\div \partial_{x_1}F,\ \ \ \left(\partial_t+\mathcal{H}\right)\partial_{x_1}\Delta v_2=\div \partial_{x_1}G.
$$
We deduce by Lemma \ref{lem:2.4} that
$$
\big\|\partial_{x_1}(\Delta v_1)\big\|^2_{X_{ed}}\lesssim \big\|\partial_{x_1}(\Delta v_1)(0)\big\|^2_{L^2}+\nu^{-1}\left\|e^{\epsilon\nu^{1/2}t}\partial_{x_1}F\right\|^2_{L^2L^2},
$$
and
$$
\big\|\partial_{x_1}(\Delta v_2)\big\|^2_{X_{ed}}\lesssim \big\|\partial_{x_1}(\Delta v_2)(0)\big\|^2_{L^2}+\nu^{-1}\left\|e^{\epsilon\nu^{1/2}t}\partial_{x_1}G\right\|^2_{L^2L^2}.
$$
Similarly, we also have
$$
\big\|\partial_{x_2}(\Delta v_1)\big\|^2_{X_{ed}}\lesssim \big\|\partial_{x_2}(\Delta v_1)(0)\big\|^2_{L^2}+\nu^{-1}\left\|e^{\epsilon\nu^{1/2}t}\partial_{x_2}F\right\|^2_{L^2L^2},
$$
and
$$
\big\|\partial_{x_2}(\Delta v_2)\big\|^2_{X_{ed}}\lesssim \big\|\partial_{x_2}(\Delta v_2)(0)\big\|^2_{L^2}+\nu^{-1}\left\|e^{\epsilon\nu^{1/2}t}\partial_{x_2}G\right\|^2_{L^2L^2}.
$$
Thus one gets
\begin{equation}\label{eq:4.55}
\begin{aligned}
&\big\|\partial_{x_1}(\Delta v_1)\big\|^2_{X_{ed}}+\big\|\partial_{x_2}(\Delta v_1)\big\|^2_{X_{ed}}
+\big\|\partial_{x_1}(\Delta v_2)\big\|^2_{X_{ed}}+\big\|\partial_{x_2}(\Delta v_2)\big\|^2_{X_{ed}}\\
\lesssim&\big\|\partial_{x_1}(\Delta v_1)(0)\big\|^2_{L^2}+\big\|\partial_{x_2}(\Delta v_1)(0)\big\|^2_{L^2}
+\big\|\partial_{x_1}(\Delta v_2)(0)\big\|^2_{L^2}+\big\|\partial_{x_2}(\Delta v_2)(0)\big\|^2_{L^2}\\
&+\nu^{-1}\left(\left\|e^{\epsilon\nu^{1/2}t}\partial_{x_1}F\right\|^2_{L^2L^2}
+\left\|e^{\epsilon\nu^{1/2}t}\partial_{x_2}F\right\|^2_{L^2L^2}\right)\\
&+\nu^{-1}\left(\left\|e^{\epsilon\nu^{1/2}t}\partial_{x_1}G\right\|^2_{L^2L^2}
+\left\|e^{\epsilon\nu^{1/2}t}\partial_{x_2}G\right\|^2_{L^2L^2}\right).
\end{aligned}
\end{equation}
Combining Lemma \ref{lem:4.3} and \ref{lem:4.4}, to obtain
$$
\begin{aligned}
\big\|\partial_{x_1}F\big\|_{L^2}
&\lesssim\big\| \nabla\left(\delta\cos(m_0 y)\partial_{x_1}v_3\right)\big\|_{L^2}
+\big\|\partial_{x_1}F_1\big\|_{L^2}\\
&\ \ \ +\big\|\partial_{x_1}F_2\big\|_{L^2}+\big\|\partial_{x_1}F_3\big\|_{L^2}
+\big\|\partial_{x_1}\nabla(V\cdot\nabla v_1)\big\|_{L^2}\\
&\lesssim\big\|\partial_{x_1}\nabla(V\cdot\nabla V)\big\|_{L^2}+\big\|\partial_{x_1}(\Delta p)\big\|_{L^2}\\
&\ \ \ +\big\|\partial^2_{x_1}v_1\big\|_{L^2}+\big\|\partial_{x_1}\partial_{x_2}v_1\big\|_{L^2}+\big\|\partial_{x_1}v_3\big\|_{H^1}\\
&\lesssim\big\|\partial_{x_1}V\big\|_{H^2}\big\|\Delta V_{\neq}\big\|_{L^2}+\big\|V\big\|_{H^2}\big\|\partial_{x_1}\Delta V_{\neq}\big\|_{L^2}+\big\|\partial_{x_1}\Delta p\big\|_{L^2}\\
&\  \ \ +\big\|\partial_{x_1}\omega_3\big\|_{L^2}+\big\|\partial_{x_2}\omega_3\big\|_{L^2}+\big\|(\Delta v_3)_{\neq}\big\|_{L^2},
\end{aligned}
$$
and
$$
\begin{aligned}
\big\|\partial_{x_1}G\big\|_{L^2}\lesssim&\big\| \nabla\left( \delta\sin(m_0 y)\partial_{x_1}v_3\right)_{\neq}\big\|_{L^2}
+\big\|\partial_{x_1}G_1\big\|_{L^2}\\
&+\big\|\partial_{x_1}G_2\big\|_{L^2}+\big\|\partial_{x_1}G_3\big\|_{L^2}+\big\|\partial_{x_1}\nabla(V\cdot\nabla v_2)\big\|_{L^2}\\
\lesssim&\big\|\partial_{x_1}V\big\|_{H^2}\big\|\Delta V_{\neq}\big\|_{L^2}+\big\|V\big\|_{H^2}\big\|\partial_{x_1}\Delta V_{\neq}\big\|_{L^2}+\big\|\partial_{x_1}\Delta p\big\|_{L^2}\\
&+\big\|\partial_{x_1}\omega_3\big\|_{L^2}+\big\|\partial_{x_2}\omega_3\big\|_{L^2}+\big\|(\Delta v_3)_{\neq}\big\|_{L^2}.
\end{aligned}
$$
Similarly, one has
$$
\begin{aligned}
\big\|\partial_{x_2}F\big\|_{L^2}
&\lesssim\big\|\partial_{x_2}V\big\|_{H^2}\big\|\Delta V_{\neq}\big\|_{L^2}+\big\|V\big\|_{H^2}\big\|\partial_{x_2}\Delta V_{\neq}\big\|_{L^2}+\big\|\partial_{x_2}\Delta p\big\|_{L^2}\\
&\ \ \ \ +\big\|\partial_{x_1}\omega_3\big\|_{L^2}+\big\|\partial_{x_2}\omega_3\big\|_{L^2}+\big\|(\Delta v_3)_{\neq}\big\|_{L^2},
\end{aligned}
$$
and
$$
\begin{aligned}
\big\|\partial_{x_2}G\big\|_{L^2}
\lesssim&\big\|\partial_{x_2}V\big\|_{H^2}\big\|\Delta V_{\neq}\big\|_{L^2}+\big\|V\big\|_{H^2}\big\|\partial_{x_2}\Delta V_{\neq}\big\|_{L^2}+\big\|\partial_{x_2}\Delta p\big\|_{L^2}\\
&+\big\|\partial_{x_1}\omega_3\big\|_{L^2}+\big\|\partial_{x_2}\omega_3\big\|_{L^2}+\big\|(\Delta v_3)_{\neq}\big\|_{L^2}.
\end{aligned}
$$
Then we have
\begin{equation}\label{eq:4.56}
\begin{aligned}
&\left\|e^{\epsilon\nu^{1/2}t}\partial_{x_1}F\right\|^2_{L^2L^2}
+\left\|e^{\epsilon\nu^{1/2}t}\partial_{x_2}F\right\|^2_{L^2L^2}
+\left\|e^{\epsilon\nu^{1/2}t}\partial_{x_1}G\right\|^2_{L^2L^2}
+\left\|e^{\epsilon\nu^{1/2}t}\partial_{x_2}G\right\|^2_{L^2L^2}\\
\lesssim&\left(\left\|e^{\epsilon\nu^{1/2}t}\partial_{x_1}\Delta V\right\|^2_{L^2L^2}
+\left\|e^{\epsilon\nu^{1/2}t}\partial_{x_2}\Delta V\right\|^2_{L^2L^2}\right)\big\|V\big\|^2_{L^\infty X_0}\\
&+\left\|e^{\epsilon\nu^{1/2}t}\partial_{x_1}\omega_3\right\|^2_{L^2 L^2}+\left\|e^{\epsilon\nu^{1/2}t}\partial_{x_2}\omega_3\right\|^2_{L^2 L^2}\\
&+\left\|e^{\epsilon\nu^{1/2}t}\partial_{x_1}(\Delta v_3)_{\neq}\right\|^2_{L^2 L^2}
+\left\|e^{\epsilon\nu^{1/2}t}\partial_{x_2}(\Delta v_3)_{\neq}\right\|^2_{L^2 L^2}\\
&+\left\|e^{\epsilon\nu^{1/2}t}\partial_{x_1}\Delta p\right\|^2_{L^2 L^2}+\left\|e^{\epsilon\nu^{1/2}t}\partial_{x_2}\Delta p\right\|^2_{L^2 L^2}\\
\lesssim& \nu^{-1/2}\left(\big\|\partial_{x_1}\Delta V\big\|^2_{X_{ed}}
+\big\|\partial_{x_2}\Delta V\big\|^2_{X_{ed}}\right)\mathcal{E}^2_2\\
&+\nu^{-1/2}
\left(\big\|\partial_{x_1}\omega_3\big\|^2_{X_{ed}}+\big\|\partial_{x_2}\omega_3\big\|^2_{X_{ed}}
+\big\|\partial_{x_1}(\Delta v_3)_{\neq}\big\|^2_{X_{ed}}
+\big\|\partial_{x_2}(\Delta v_3)_{\neq}\big\|^2_{X_{ed}}\right)\\
&+\left\|e^{\epsilon\nu^{1/2}t}\partial_{x_1}\Delta p\right\|^2_{L^2 L^2}+\left\|e^{\epsilon\nu^{1/2}t}\partial_{x_2}\Delta p\right\|^2_{L^2 L^2}.
\end{aligned}
\end{equation}
Combining the estimate of $\mathcal{M}_1$, \eqref{eq:4.55} and \eqref{eq:4.56}, we obtain
$$
\begin{aligned}
&\big\|\partial_{x_1}(\Delta V)\big\|^2_{X_{ed}}+\big\|\partial_{x_2}(\Delta V)\big\|^2_{X_{ed}}\\
\leq& C\big\|V_0\big\|^2_{X_0}+C\nu^{-3/2}\mathcal{E}^2_2\left(\big\|\partial_{x_1}(\Delta V)\big\|^2_{X_{ed}}+\big\|\partial_{x_2}(\Delta V)\big\|^2_{X_{ed}}\right)\\
&+C\nu^{-3/2}\big\|V_0\big\|^2_{X_0}+C\nu^{-1}\left\|e^{\epsilon\nu^{1/2}t}\partial_{x_1}\Delta p\right\|^2_{L^2 L^2}+C\nu^{-1}\left\|e^{\epsilon\nu^{1/2}t}\partial_{x_2}\Delta p\right\|^2_{L^2 L^2}\\
\leq&C\nu^{-3/2}\big\|V_0\big\|^2_{X_0}+C\nu^{-3/2}\mathcal{E}^2_2\left(\big\|\partial_{x_1}(\Delta V)\big\|^2_{X_{ed}}+\big\|\partial_{x_2}(\Delta V)\big\|^2_{X_{ed}}\right).
\end{aligned}
$$
Since $\mathcal{E}_2\leq \varepsilon_1\nu^{3/4}$ for some $\varepsilon_1\in (0,1)$, such that
$$
C\nu^{-3/2}\mathcal{E}^2_2\leq 1/2,
$$
then we have
$$
\big\|\partial_{x_1}\Delta V\big\|^2_{X_{ed}}+\big\|\partial_{x_2}\Delta V\big\|^2_{X_{ed}}
\lesssim \nu^{-3/2}\big\|V_0\big\|^2_{X_0}.
$$

\vskip .05in

\noindent {\bf Step 3. The estimate of $\mathcal{E}_1(T)$.} By the definition of $\|\cdot\|_{X_0}$, $\|\cdot\|_{X_1}$, $\|\cdot\|_{X_{ed}}$ and the estimate of $\mathcal{M}_1$, we imply that
$$
\begin{aligned}
\big\|v_3\big\|_{X_0}&=\big\|P_0v_3\big\|_{H^2}+\big\|\partial_{x_1}v_3\big\|_{H^2}+\big\|\partial_{x_2}v_3\big\|_{H^2}\\
&\lesssim\big\|P_0\Delta v_3\big\|_{X_1}+\big\|\partial_{x_1}\Delta v_3\big\|_{X_1}+\big\|\partial_{x_2}\Delta v_3\big\|_{X_1}\lesssim\big\|V_0\big\|^2_{X_0},
\end{aligned}
$$
and
$$
\begin{aligned}
&e^{\epsilon\nu^{1/2}t}\left(\big\|\partial_{x_1}(\Delta v_3)\big\|_{L^2}+\big\|\partial_{x_2}(\Delta v_3)\big\|_{L^2}\right)
+e^{\epsilon\nu^{1/2}t}\left(\big\|\partial_{x_1}\omega_3\big\|_{L^2}
+\big\|\partial_{x_2}\omega_3\big\|_{L^2}\right)\\
\lesssim&\big\|\partial_{x_1}(\Delta v_3)\big\|_{X_{ed}}+\big\|\partial_{x_2}(\Delta v_3)\big\|_{X_{ed}}+\big\|\partial_{x_1}\omega_3\big\|_{X_{ed}}
+\big\|\partial_{x_2}\omega_3\big\|_{X_{ed}}\lesssim \big\|V_0\big\|_{X_0},
\end{aligned}
$$
then we have
\begin{equation}\label{eq:4.57}
\mathcal{E}_1(T)\lesssim \big\|V_0\big\|_{X_0}.
\end{equation}

\vskip .05in

\noindent {\bf Step 4. The estimate of $\mathcal{E}_2(T)$.} Based on the estimate of $\mathcal{M}_1$ and $\mathcal{M}_2$,  we need only to estimate $\|\Delta P_0 v_1\|_{L^2}$ and $\|\Delta P_0 v_2\|_{L^2}$. Applying the operator $P_0$ to \eqref{eq:4.53}, one gets 
\begin{equation}\label{eq:4.58}
(\partial_t-\nu\Delta)\partial^2_yP_0v_1=\div \widetilde{F},
\end{equation}
where
$$
\widetilde{F}=-\nabla\left(\delta\cos(m_0 y)P_0v_3\right)-\nabla P_0(V\cdot\nabla v_1).
$$
By energy estimate, one has
$$
\big\|\Delta P_0v_1\big\|^2_{X_1}\lesssim\big\|\Delta P_0v_1(0)\big\|^2_{L^2}+\nu^{-1}\big\|\widetilde{F}\big\|^2_{L^2L^2}.
$$
Combining Lemma \ref{lem:2.5} and \ref{lem:4.4}, one gets
$$
\begin{aligned}
\big\|\widetilde{F}\big\|_{L^2}
&\lesssim
\big\|\nabla\left(\delta\cos(m_0 y)P_0v_3\right)\big\|_{L^2}+\big\|\nabla P_0(V\cdot\nabla v_1)\big\|_{L^2}\\
&\lesssim\big\|\partial^2_yP_0v_3\big\|_{L^2}+\big\|P_0\left(V_{\neq}\cdot\nabla(v_1)_{\neq}\right)\big\|_{H^1}
+\big\|P_0V\cdot \nabla P_0v_1\big\|_{H^1}\\
&\lesssim\big\|\nabla\Delta P_0v_3\big\|_{L^2}+\big\|V\big\|_{H^2}\big\|\Delta V_{\neq}\big\|_{L^2}+\big\|v_3\big\|_{H^2}\big\|\nabla\Delta P_0v_1\big\|_{L^2},
\end{aligned}
$$
so we can easily get
$$
\begin{aligned}
\big\|\widetilde{F}\big\|^2_{L^2L^2}
&\lesssim\big\|\nabla\Delta P_0v_3\big\|^2_{L^2L^2}+\big\|V\big\|^2_{L^\infty H^2}\left\|e^{\epsilon\nu^{1/2}t}\Delta V_{\neq}\right\|^2_{L^2 L^2}+\big\|v_3\big\|^2_{L^\infty H^2}\big\|\nabla\Delta P_0v_1\big\|^2_{L^2 L^2}\\
&\lesssim \nu^{-1}\big\|P_0\Delta v_3\big\|^2_{X_1}+\nu^{-1/2}\left(\big\|\Delta \partial_{x_1}V_{\neq}\big\|^2_{X_{ed}}+\big\|\Delta \partial_{x_2}V_{\neq}\big\|^2_{X_{ed}}\right)\big\|V\big\|^2_{L^\infty X_0}\\
&\ \ \ \ +\nu^{-1}\big\|\Delta P_0v_1\big\|^2_{X_1}\big\|v_3\big\|^2_{L^\infty X_0}.
\end{aligned}
$$
Thus, we have
$$
\begin{aligned}
\big\|\Delta P_0v_1\big\|^2_{X_1}&\leq C\big\|\Delta P_0v_1(0)\big\|^2_{L^2}+C\nu^{-1}\big\|\widetilde{F}\big\|^2_{L^2L^2}\\
&\leq C\big\|\Delta P_0v_1(0)\big\|^2_{L^2}+C\nu^{-2}\big\|P_0\Delta v_3\big\|^2_{X_1}+C\nu^{-2}\mathcal{E}^2_1\big\|\Delta P_0v_1\big\|^2_{X_1}\\
&\ \ \ +C\nu^{-3/2}\mathcal{E}^2_2\left(\big\|\Delta \partial_{x_1}V_{\neq}\big\|^2_{X_{ed}}+\big\|\Delta \partial_{x_2}V_{\neq}\big\|^2_{X_{ed}}\right).
\end{aligned}
$$
Since $\mathcal{E}_1\leq \varepsilon_1\nu, \mathcal{E}_2\leq \varepsilon_1\nu^{3/4}$ for some $\varepsilon_1\in (0,1)$, such that
$$
C\nu^{-2}\mathcal{E}^2_1\leq 1/2,
$$
and combining the estimate of $\mathcal{M}_1$ and $\mathcal{M}_2$, we obtain
\begin{equation}\label{eq:4.59}
\begin{aligned}
\big\|\Delta P_0v_1\big\|^2_{X_1}
&\lesssim \big\|\Delta P_0v_1(0)\big\|^2_{L^2}+\nu^{-2}\big\|P_0\Delta v_3\big\|^2_{X_1}\\
&\ \ \ +\left(\big\|\Delta \partial_{x_1}V_{\neq}\big\|^2_{X_{ed}}+\big\|\Delta \partial_{x_2}V_{\neq}\big\|^2_{X_{ed}}\right)\\
&\lesssim\big\|V_0\big\|^2_{X_0}+\nu^{-2}\big\|V_0\big\|^2_{X_0}+\nu^{-3/2}\big\|V_0\big\|^2_{X_0}\\
&\lesssim\nu^{-2}\big\|V_0\big\|^2_{X_0}.
\end{aligned}
\end{equation}
Applying the operator $P_0$ to \eqref{eq:4.54}, we obtain
$$
(\partial_t-\nu\Delta)\partial^2_yP_0v_2=\div \widetilde{G},
$$
where
$$
\widetilde{G}=-\nabla\left( \delta\sin(m_0 y)P_0v_3\right)-\nabla P_0(V\cdot\nabla v_2).
$$
Similar to \eqref{eq:4.58} and \eqref{eq:4.59}, we have
\begin{equation}\label{eq:4.60}
\big\|\Delta P_0v_2\big\|^2_{X_1}
\lesssim \nu^{-2}\big\|V_0\big\|^2_{X_0}.
\end{equation}
Then we deduce by the estimate of $\mathcal{M}_1$, $\mathcal{M}_2$, \eqref{eq:4.59} and \eqref{eq:4.60} that
$$
\begin{aligned}
\big\|V\big\|^2_{X_0}&\lesssim\big\|\Delta P_0 v_1\big\|^2_{L^2}+\big\|\Delta P_0 v_2\big\|^2_{L^2}+\big\|\Delta P_0 v_3\big\|^2_{L^2}
+\big\|\partial_{x_1}\Delta V\big\|^2_{L^2}+\big\|\partial_{x_2}\Delta V\big\|^2_{L^2}\\
&\lesssim\big\|\Delta P_0 v_1\big\|^2_{X_1}+\big\|\Delta P_0 v_2\big\|^2_{X_1}+\big\|\Delta P_0 v_3\big\|^2_{X_1}
+\big\|\partial_{x_1}\Delta V\big\|^2_{X_{ed}}+\big\|\partial_{x_2}\Delta V\big\|^2_{X_{ed}}\\
&\lesssim \big\|V_0\big\|^2_{X_0}+\nu^{-3/2}\big\|V_0\big\|^2_{X_0}+\nu^{-2}\big\|V_0\big\|^2_{X_0}\\
&\lesssim \nu^{-2}\big\|V_0\big\|^2_{X_0}.
\end{aligned}
$$
Thus, we have
\begin{equation}\label{eq:4.61}
\mathcal{E}_2(T)\lesssim\nu^{-1}\big\|V_0\big\|_{X_0}.
\end{equation}
Combining \eqref{eq:4.57} and \eqref{eq:4.61}, we know that there exists a positive constant $C_0$, such that
$$
\mathcal{E}_1(T)\leq C_0\big\|V_0\big\|_{X_0},\ \ \ \mathcal{E}_2(T)\leq C_0\nu^{-1}\big\|V_0\big\|_{X_0}.
$$
This completes the proof of Proposition \ref{prop:4.5}.
\end{proof}

\vskip .05in

\begin{remark}
In Proposition {\rm \ref{prop:4.5}}, we know that zero mode estimate
$$
\big\|\Delta P_0 v_1\big\|^2_{L^2}+\big\|\Delta P_0 v_2\big\|^2_{L^2}\lesssim\nu^{-2}\big\|V_0\big\|^2_{X_0},\ \ \
\big\|\Delta P_0 v_3\big\|^2_{L^2}\lesssim \big\|V_0\big\|^2_{X_0},
$$
and non-zero mode estimate
$$
\big\|\partial_{x_1}\Delta V\big\|^2_{L^2}+\big\|\partial_{x_2}\Delta V\big\|^2_{L^2}
\lesssim \nu^{-3/2}\big\|V_0\big\|^2_{X_0},
$$
that means the estimate of $\|V\|_{X_0}$ depends on the $\|\Delta P_0 v_1\|_{L^2}$ and $\|\Delta P_0 v_2\|_{L^2}$. Thus, the stability threshold of \eqref{eq:1.3}  is determined by the estimates of the $P_0 v_1$ and $P_0 v_2$. In other words, the unstable part of $\eqref{eq:1.3}$ is the zero mode of $v_1$ and $v_2$, which is from the lift-up effect of planar helical flow.
\end{remark}

Next, we prove the Theorem \ref{thm:1.1} based on the Proposition \ref{prop:4.5}, the detail is as follows.

\begin{proof}[The proof of Theorem {\rm \ref{thm:1.1}}]
Let $c_0=\varepsilon_1/2C_0$, and for $\|V_0\|_{X_0}\leq c_0\nu^{7/4}$, we can easily know that
$$
\mathcal{E}_1(0)\leq C_0\big\|V_0\big\|_{X_0}\leq \varepsilon_1\nu/2,\ \ \ \mathcal{E}_2(0)\leq C_0\big\|V_0\big\|_{X_0}\leq \varepsilon_1\nu^{3/4}/2.
$$
Therefore, there exists $T>0$, such that
$$
\mathcal{E}_1(T)\leq\varepsilon_1\nu,\ \ \ \ \mathcal{E}_2(T)\leq\varepsilon_1\nu^{3/4}.
$$
Then we imply by Proposition \ref{prop:4.5} that
$$
\mathcal{E}_1(T)\leq C_0\big\|V_0\big\|_{X_0}\leq \varepsilon_1\nu/2,
$$
and
$$
 \mathcal{E}_2(T)\leq C_0\nu^{-1}\big\|V_0\big\|_{X_0}\leq C_0c_0\nu^{3/4}\leq \varepsilon_1\nu^{3/4}/2.
$$
By continuity argument, we imply that
$$
T=\infty.
$$
This completes the proof of Theorem \ref{thm:1.1}.
\end{proof}

\begin{remark}
In this paper, we use continuity methods to extend $ \mathcal{E}_2(T)$ from local to global. The smallness of $V_0$ makes the assumption of $\mathcal{E}_1(T)$ and $\mathcal{E}_2(T)$ locally valid. And combining Proposition {\rm \ref{prop:4.5}} and the condition of $V_0$, the estimate of $\mathcal{E}_2(T)$ is global by continuity argument.
\end{remark}

\begin{remark}
For the constant $c_0,\epsilon,\nu_0$ and positive constant $C_0$ in the Theorem {\rm \ref{thm:1.1}}, we can choose them by the proofs of the Theorem {\rm \ref{thm:1.1}} and  Proposition {\rm \ref{prop:4.5}} that
$$
\varepsilon_1\Rightarrow C_0\Rightarrow c_0=\varepsilon_1/2C_0,
$$
the constant $\epsilon$ is fixed in \eqref{eq:1.15} and choose $\nu_0$, for any $\nu\in (0,\nu_0)$, one has $\nu\ll1$.
\end{remark}

\appendix

\vskip .05in

\section{Linearized Navier-Stokes equation in the case of $\delta=1$}\label{sec.A}

In this section, we study the linearized equation of \eqref{eq:1.5} in the case of $\delta=1$. We consider the linear equation
\begin{equation}\label{eq:A.1}
\begin{cases}
\partial_t u+\left(\nu(k_1^2+k_2^2)+\mathcal{L}_{k_1,k_2}  \right)u=0,\\
\partial_t w+\left(\nu(k_1^2+k_2^2)+\mathcal{H}_{k_1,k_2}  \right)w=-i|k|\sin (y+\alpha_k)(\alpha^2-\partial_y^2)^{-1}u,\\
u(0,y)=u_0(y),\ \ \ w(0,y)=w_0(y),
\end{cases}
\end{equation}
where
$$
\mathcal{L}_{k_1,k_2}=-\nu m^2_0\partial^2_{y}+i|k|\sin (y+\alpha_k)\left(1-(\alpha^2-\partial^2_{y})^{-1}\right),
$$

$$
\mathcal{H}_{k_1,k_2}
=-\nu m^2_0\partial^2_{y}+i|k|\sin (y+\alpha_k),
$$
and
$$
\alpha^2=(k_1^2+k_2^2),\ \ k=(k_1,k_2),\ \ |k|=\sqrt{k_1^2+k_2^2},\ \ \cos \alpha_k=\frac{k_1}{|k|},\ \ \sin \alpha_k=\frac{k_2}{|k|}.
$$
Since $k^2_1+k_2^2\geq 1$, and we know that the situation is similar to $\delta>1$ for the case of $\alpha^2=(k_1^2+k_2^2)>1$, see Section \ref{sec.3}. Thus, we consider the case of $\alpha^2=(k_1^2+k_2^2)=1$ here.

\vskip .05in

Let denote function space
$$
L_0^2(\mathbb{T}_{2\pi})=\left\{u\in L^2(\mathbb{T}_{2\pi})\big| \int_{0}^{2\pi}u(y)dy=0 \right\}.
$$
Definite the $Q_1$ is the orthogonal projection operator, it is as follows
$$
Q_1:L^2(\mathbb{T}_{2\pi}) \rightarrow L_0^2(\mathbb{T}_{2\pi}),
$$
and
$$
P_1=I-Q_1.
$$
Next, we state the following result.

\begin{proposition}\label{prop:A.1}
Let $u$ and $w$ are solutions of \eqref{eq:A.1}. If $|k|^2=k_1^2+k_2^2=1$ and $0<\nu\ll 1$, then we have

\begin{equation}\label{eq:A.2}
\big\|Q_1u(t,\cdot)\big\|_{L^2}\lesssim e^{-c|k|^{1/2}\nu^{1/2}t-\nu t}\big\|Q_1u_0\big\|_{L^2},
\end{equation}

\begin{equation}\label{eq:A.3}
\big\|P_1u(t,\cdot)\big\|_{L^2}\lesssim e^{-\nu t}\big\|P_1 u_0\big\|_{L^2}+e^{-\nu t}\nu^{-1/6}\big\|Q_1 u_0\big\|_{L^2},
\end{equation}

\begin{equation}\label{eq:A.4}
\big\|Q_1 w(t,\cdot)\big\|_{L^2}
\lesssim e^{-c|k|^{1/2}\nu^{1/2}t-\nu t}\big(\big\|w_0\big\|_{L^2}+\big\|u_0\big\|_{L^2}\big),
\end{equation}
and
\begin{equation}\label{eq:A.5}
\begin{aligned}
\big\|P_1 w(t,\cdot)\big\|_{L^2}\lesssim &e^{-c|k|^{1/2}\nu^{1/2}t-\nu t}\big(\big\|w_0\big\|_{L^2}+\big\|u_0\big\|_{L^2}\big)\\
&+e^{-\nu t}\big\|P_1 u_0\big\|_{L^2}+e^{-\nu t}\nu^{-1/6}\big\|Q_1 u_0\big\|_{L^2}.
\end{aligned}
\end{equation}
\end{proposition}

\begin{proof}
For $k_1^2+k_2^2=1$, we only have two case: $k_1=1, k_2=0$ and $k_1=0, k_2=1$. In the case of $k_1=1, k_2=0$, the first equation of \eqref{eq:A.1} is written as
\begin{equation}\label{eq:A.6}
\partial_t u+\nu u-\nu \partial^2_{y}u+ik_1\sin yu
+ik_1\sin y\varphi=0,
\end{equation}
where $(\partial^2_{y}-1)\varphi=u$. In this case, the \eqref{eq:A.6} is the same as the linearized equation of 3-D Kolmogorov flow. Thus, we can easily to obtain the estimates of \eqref{eq:A.2} and  \eqref{eq:A.3} by previous works for the 3-D Kolmogorov flow (see \cite{LWZ.2020}, Theorem 1.1). Next, we consider the estimate of $w$. The \eqref{eq:A.1} can be written as
$$
\begin{cases}
\partial_t u+\left(\nu+ik_1\sin y\left(1-(1-\partial^2_{y})^{-1}\right)-\nu\partial_y^2\right)u=0,\\
\partial_t w+\left(\nu+ik_1\sin y-\nu\partial_y^2\right)w=-ik_1\sin y(1-\partial_y^2)^{-1}u,\\
u(0,y)=u_0(y),\ \ \ w(0,y)=w_0(y).
\end{cases}
$$
Then we have
$$
\partial_t (w+u)+\left(\nu+\mathcal{H}_{k_1,k_2}\right)(w+u)=0,
$$
and by Semigroup theory, Lemma \ref{lem:2.1} and \ref{lem:2.3}, one has
\begin{equation}\label{eq:A.7}
\begin{aligned}
\big\|w+u\big\|_{L^2}&=\left\|e^{-\nu t}e^{-t\mathcal{H}_{k_1,k_2}}(u_0+w_0)\right\|_{L^2}\\
&\lesssim e^{-c|k|^{1/2}\nu^{1/2}t-\nu t}\big\|w_0+u_0\big\|_{L^2}\\
&\lesssim e^{-c\nu^{1/2}t-\nu t}\big(\big\|w_0\big\|_{L^2}+\big\|u_0\big\|_{L^2}\big).
\end{aligned}
\end{equation}
Combining the definition of $Q_1$, $P_1$ and  \eqref{eq:A.7}, one gets
$$
\big\|Q_1(w+u)\big\|_{L^2}\lesssim e^{-c\nu^{1/2}t-\nu t}\big(\big\|w_0\big\|_{L^2}+\big\|u_0\big\|_{L^2}\big),
$$
$$
\big\|P_1(w+u)\big\|_{L^2}\lesssim e^{-c\nu^{1/2}t-\nu t}\big(\big\|w_0\big\|_{L^2}+\big\|u_0\big\|_{L^2}).
$$
Thus, we have
$$
\begin{aligned}
\big\|Q_1 w\big\|_{L^2}&\lesssim\big\|Q_1(u+w)\big\|_{L^2}+\big\|Q_1u\big\|_{L^2}\\
&\lesssim e^{-c\nu^{1/2}t-\nu t}\big(\big\|w_0\big\|_{L^2}+\big\|u_0\big\|_{L^2}\big)+e^{-c\nu^{1/2}t-\nu t}\big\|Q_1u_0\big\|_{L^2}\\
&\lesssim e^{-c\nu^{1/2}t-\nu t}\big(\big\|w_0\big\|_{L^2}+\big\|u_0\big\|_{L^2}\big),
\end{aligned}
$$
and
$$
\begin{aligned}
\big\|P_1 w\big\|_{L^2}&\lesssim\big\|P_1(u+w)\big\|_{L^2}+\big\|P_1u\big\|_{L^2}\\
&\lesssim e^{-c\nu^{1/2}t-\nu t}\big(\big\|w_0\big\|_{L^2}+\big\|u_0\big\|_{L^2}\big)\\
&\ \ \ \ + e^{-\nu t}\big\|P_1 u_0\big\|_{L^2}+e^{-\nu t}\nu^{-1/6}\big\|Q_1 u_0\big\|_{L^2}.
\end{aligned}
$$
The case of $k_1=0,k_2=1$ is similar. This completes the proof of Proposition \ref{prop:A.1}.
\end{proof}

\vskip .06in

\noindent \textbf{Acknowledgement.} The authors would like to thank Professor Weike Wang, Professor Xi Zhang and Dr.Yuanyuan Feng for helpful discussions and useful comments.Y.Wang also would like to thank the School of Business Informatics and Mathematics, University of Mannheim for kindly host.~The work of B.Shi was partially supported by the Jiangsu Funding Program for Excellent Postdoctoral Talent (Grant No.2023ZB116). The work of Y.Wang was partially supported by the National Natural Science Foundation of China (Grant No.12271357), Shanghai Science and Technology Innovation Action Plan (Grant No.21JC1403600) and the CSC-DAAD Postdoc Scholarship.

\bibliographystyle{abbrv}
\bibliography{3D-NS-Helical-ref}

\begin{thebibliography}{10}

\bibitem{AHL.0003}
X.~An, T.~He, and T.~Li.
\newblock Nonlinear asymptotic stability and transition threshold for 2{D}
  {T}aylor-{C}ouette flows in sobolev spaces.
\newblock {\em {\rm arXiv:2306.13562}}.

\bibitem{Arnold.1965}
V.~Arnold.
\newblock Sur la topologie des \'{e}coulements stationnaires des fluides
  parfaits.
\newblock {\em C. R. Acad. Sci. Paris}, 261:17--20, 1965.

\bibitem{BGM.2017}
J.~Bedrossian, P.~Germain, and N.~Masmoudi.
\newblock On the stability threshold for the 3{D} {C}ouette flow in {S}obolev
  regularity.
\newblock {\em Ann. of Math. (2)}, 185(2):541--608, 2017.

\bibitem{BGM.2019}
J.~Bedrossian, P.~Germain, and N.~Masmoudi.
\newblock Stability of the {C}ouette flow at high {R}eynolds numbers in two
  dimensions and three dimensions.
\newblock {\em Bull. Amer. Math. Soc. (N.S.)}, 56(3):373--414, 2019.

\bibitem{BGM.2020}
J.~Bedrossian, P.~Germain, and N.~Masmoudi.
\newblock Dynamics near the subcritical transition of the 3{D} {C}ouette flow
  {I}: {B}elow threshold case.
\newblock {\em Mem. Amer. Math. Soc.}, 266(1294):v+158, 2020.

\bibitem{BHIW.0004}
J.~Bedrossian, S.~He, S.~Iyer, and F.~Wang.
\newblock Stability threshold of nearly-{C}ouette shear flows with navier
  boundary conditions in 2{D}.
\newblock {\em {\rm arXiv:2311.00141}}.

\bibitem{BMV.2016}
J.~Bedrossian, N.~Masmoudi, and V.~Vicol.
\newblock Enhanced dissipation and inviscid damping in the inviscid limit of
  the {N}avier-{S}tokes equations near the two dimensional {C}ouette flow.
\newblock {\em Arch. Ration. Mech. Anal.}, 219(3):1087--1159, 2016.

\bibitem{BVW.2018}
J.~Bedrossian, V.~Vicol, and F.~Wang.
\newblock The {S}obolev stability threshold for 2{D} shear flows near
  {C}ouette.
\newblock {\em J. Nonlinear Sci.}, 28(6):2051--2075, 2018.

\bibitem{Chapman.2002}
S.~J. Chapman.
\newblock Subcritical transition in channel flows.
\newblock {\em J. Fluid Mech.}, 451:35--97, 2002.

\bibitem{CDLZ.0002}
Q.~Chen, S.~Ding, Z.~Lin, and Z.~Zhang.
\newblock Nonlinear stability for 3-{D} plane {P}oiseuille flow in a finite
  channel.
\newblock {\em {\rm arXiv:2310.11694}}.

\bibitem{CLWZ.2020}
Q.~Chen, T.~Li, D.~Wei, and Z.~Zhang.
\newblock Transition threshold for the 2-{D} {C}ouette flow in a finite
  channel.
\newblock {\em Arch. Ration. Mech. Anal.}, 238(1):125--183, 2020.

\bibitem{CWZ.0001}
Q.~Chen, D.~Wei, and Z.~Zhang.
\newblock Transition threshold for the 3{D} {C}ouette flow in a finite channel.
\newblock {\em {\rm arXiv:2006.00721}, {\rm to appear in} Mem. Amer. Math.
  Soc}.

\bibitem{CWZ.2023}
Q.~Chen, D.~Wei, and Z.~Zhang.
\newblock Linear stability of pipe {P}oiseuille flow at high {R}eynolds number
  regime.
\newblock {\em Comm. Pure Appl. Math.}, 76(9):1868--1964, 2023.

\bibitem{Childress.1970}
S.~Childress.
\newblock New solutions of the kinematic dynamo problem.
\newblock {\em J. Mathematical Phys.}, 11:3063--3076, 1970.

\bibitem{CEW.2020}
M.~Coti~Zelati, T.~M. Elgindi, and K.~Widmayer.
\newblock Enhanced dissipation in the {N}avier-{S}tokes equations near the
  {P}oiseuille flow.
\newblock {\em Comm. Math. Phys.}, 378(2):987--1010, 2020.

\bibitem{DHB.1992}
F.~Daviaud, J.~Hegseth, and P.~Berg\'e.
\newblock Subcritical transition to turbulence in plane couette flow.
\newblock {\em Phys. Rev. Lett.}, 69:2511--2514, Oct 1992.

\bibitem{Zotto.2023}
A.~Del~Zotto.
\newblock Enhanced dissipation and transition threshold for the {P}oiseuille
  flow in a periodic strip.
\newblock {\em SIAM J. Math. Anal.}, 55(5):4410--4424, 2023.

\bibitem{DWZ.2021}
W.~Deng, J.~Wu, and P.~Zhang.
\newblock Stability of {C}ouette flow for 2{D} {B}oussinesq system with
  vertical dissipation.
\newblock {\em J. Funct. Anal.}, 281(12):Paper No. 109255, 40, 2021.

\bibitem{DU.2018}
A.~A. Didov and M.~Y. Uleysky.
\newblock Analysis of stationary points and their bifurcations in the
  {$ABC$}-flow.
\newblock {\em Appl. Math. Comput.}, 330:56--64, 2018.

\bibitem{DU.201802}
A.~A. Didov and M.~Y. Uleysky.
\newblock Nonlinear resonances in the {$ABC$}-flow.
\newblock {\em Chaos}, 28(1):013123, 8, 2018.

\bibitem{DL.2022}
S.~Ding and Z.~Lin.
\newblock Enhanced dissipation and transition threshold for the 2-{D} plane
  {P}oiseuille flow via resolvent estimate.
\newblock {\em J. Differential Equations}, 332:404--439, 2022.

\bibitem{DFGHMS.1986}
T.~Dombre, U.~Frisch, J.~M. Greene, M.~H\'{e}non, A.~Mehr, and A.~M. Soward.
\newblock Chaotic streamlines in the {ABC} flows.
\newblock {\em J. Fluid Mech.}, 167:353--391, 1986.

\bibitem{DR.1982}
P.~G. Drazin and W.~H. Reid.
\newblock {\em Hydrodynamic stability}.
\newblock Cambridge Monographs on Mechanics and Applied Mathematics. Cambridge
  University Press, Cambridge-New York, 1982.

\bibitem{Eckert.2010}
M.~Eckert.
\newblock The troublesome birth of hydrodynamic stability theory: Sommerfeld
  and the turbulence problem.
\newblock {\em European Physical Journal H}, 35(1):29--51, 2010.

\bibitem{EP.1975}
T.~Ellingsen and E.~Palm.
\newblock {Stability of linear flow}.
\newblock {\em The Physics of Fluids}, 18(4):487--488, 04 1975.

\bibitem{FBW.2022}
Y.~Feng, B.~Shi, and W.~Wang.
\newblock Dissipation enhancement of planar helical flows and applications to
  three-dimensional {K}uramoto-{S}ivashinsky and {K}eller-{S}egel equations.
\newblock {\em J. Differential Equations}, 313:420--449, 2022.

\bibitem{GGN.2009}
I.~Gallagher, T.~Gallay, and F.~Nier.
\newblock Spectral asymptotics for large skew-symmetric perturbations of the
  harmonic oscillator.
\newblock {\em Int. Math. Res. Not. IMRN}, (12):2147--2199, 2009.

\bibitem{GF.1987}
D.~Galloway and U.~Frisch.
\newblock A note on the stability of a family of space-periodic beltrami flows.
\newblock {\em Journal of Fluid Mechanics}, 180:557--564, 1987.

\bibitem{HJM.2003}
B.~Hof, A.~Juel, and T.~Mullin.
\newblock Scaling of the turbulence transition threshold in a pipe.
\newblock {\em Phys. Rev. Lett.}, 91:244502, Dec 2003.

\bibitem{HZD.1998}
D.-B. Huang, X.-H. Zhao, and H.-H. Dai.
\newblock Invariant tori and chaotic streamlines in the {ABC} flow.
\newblock {\em Phys. Lett. A}, 237(3):136--140, 1998.

\bibitem{Kato.1966}
T.~Kato.
\newblock {\em Perturbation theory for linear operators}.
\newblock Die Grundlehren der mathematischen Wissenschaften, Band 132.
  Springer-Verlag New York, Inc., New York, 1966.

\bibitem{Kelvin.1887}
L.~Kelvin.
\newblock Stability of fluid motion: rectilinear motion of viscous fluid
  between two parallel plates.
\newblock {\em Phil. Mag}, 24(5):188--196, 1887.

\bibitem{Kerswell.2005}
R.~R. Kerswell.
\newblock Recent progress in understanding the transition to turbulence in a
  pipe.
\newblock {\em Nonlinearity}, 18(6):R17--R44, 2005.

\bibitem{LWZ.202002}
T.~Li, D.~Wei, and Z.~Zhang.
\newblock Pseudospectral and spectral bounds for the {O}seen vortices operator.
\newblock {\em Ann. Sci. \'{E}c. Norm. Sup\'{e}r. (4)}, 53(4):993--1035, 2020.

\bibitem{LWZ.2020}
T.~Li, D.~Wei, and Z.~Zhang.
\newblock Pseudospectral bound and transition threshold for the 3{D}
  {K}olmogorov flow.
\newblock {\em Comm. Pure Appl. Math.}, 73(3):465--557, 2020.

\bibitem{LX.2019}
Z.~Lin and M.~Xu.
\newblock Metastability of {K}olmogorov flows and inviscid damping of shear
  flows.
\newblock {\em Arch. Ration. Mech. Anal.}, 231(3):1811--1852, 2019.

\bibitem{Liss.2020}
K.~Liss.
\newblock On the {S}obolev stability threshold of 3{D} {C}ouette flow in a
  uniform magnetic field.
\newblock {\em Comm. Math. Phys.}, 377(2):859--908, 2020.

\bibitem{LHR.1994}
A.~Lundbladh, D.~S. Henningson, and S.~C. Reddy.
\newblock Threshold amplitudes for transition in channel flows.
\newblock {\em Springer Netherlands}, 1994.

\bibitem{MZ.2022}
N.~Masmoudi and W.~Zhao.
\newblock Stability threshold of two-dimensional {C}ouette flow in {S}obolev
  spaces.
\newblock {\em Ann. Inst. H. Poincar\'{e} C Anal. Non Lin\'{e}aire},
  39(2):245--325, 2022.

\bibitem{MXYZ.2016}
T.~McMillen, J.~Xin, Y.~Yu, and A.~Zlato\v{s}.
\newblock Ballistic orbits and front speed enhancement for {ABC} flows.
\newblock {\em SIAM J. Appl. Dyn. Syst.}, 15(3):1753--1782, 2016.

\bibitem{MM.2007}
F.~Mellibovsky and A.~Meseguer.
\newblock {Pipe flow transition threshold following localized impulsive
  perturbations}.
\newblock {\em Physics of Fluids}, 19(4):044102, 04 2007.

\bibitem{QL.2023}
S.~Qin and S.~Liao.
\newblock A kind of {L}agrangian chaotic property of the
  {A}rnold-{B}eltrami-{C}hildress flow.
\newblock {\em J. Fluid Mech.}, 960:Paper No. A15, 23, 2023.

\bibitem{RSBH.1998}
S.~C. Reddy, P.~J. Schmid, J.~S. Baggett, and D.~S. Henningson.
\newblock On stability of streamwise streaks and transition thresholds in plane
  channel flows.
\newblock {\em J. Fluid Mech.}, 365:269--303, 1998.

\bibitem{Reynolds1883}
O.~Reynolds.
\newblock An experimental investigation of the circumstances which determine
  whether the motion of water shall be direct or sinuous, and of the law of
  resistance in parallel channels.
\newblock {\em Philosophical Transactions of the Royal Society of London},
  1883.

\bibitem{SH.2001}
P.~J. Schmid and D.~S. Henningson.
\newblock {\em Stability and transition in shear flows}, volume 142 of {\em
  Applied Mathematical Sciences}.
\newblock Springer-Verlag, New York, 2001.

\bibitem{Trefethen.1997}
L.~N. Trefethen.
\newblock Pseudospectra of linear operators.
\newblock {\em SIAM Rev.}, 39(3):383--406, 1997.

\bibitem{TTRD.1993}
L.~N. Trefethen, A.~E. Trefethen, S.~C. Reddy, and T.~A. Driscoll.
\newblock Hydrodynamic stability without eigenvalues.
\newblock {\em Science}, 261(5121):578--584, 1993.

\bibitem{Wei.2021}
D.~Wei.
\newblock Diffusion and mixing in fluid flow via the resolvent estimate.
\newblock {\em Sci. China Math.}, 64(3):507--518, 2021.

\bibitem{WZ.2021}
D.~Wei and Z.~Zhang.
\newblock Transition threshold for the 3{D} {C}ouette flow in {S}obolev space.
\newblock {\em Comm. Pure Appl. Math.}, 74(11):2398--2479, 2021.

\bibitem{WZ.2023}
D.~Wei and Z.~Zhang.
\newblock Nonlinear enhanced dissipation and inviscid damping for the 2{D}
  {C}ouette flow.
\newblock {\em Tunis. J. Math.}, 5(3):573--592, 2023.

\bibitem{WZZ.2020}
D.~Wei, Z.~Zhang, and W.~Zhao.
\newblock Linear inviscid damping and enhanced dissipation for the {K}olmogorov
  flow.
\newblock {\em Adv. Math.}, 362:106963, 103, 2020.

\bibitem{Yaglom.2012}
A.~M. Yaglom.
\newblock {\em Hydrodynamic instability and transition to turbulence}, volume
  100 of {\em Fluid Mechanics and its Applications}.
\newblock Springer, Dordrecht, 2012.
\newblock With a foreword by Uriel Frisch and a memorial note for Yaglom by
  Peter Bradshaw.

\bibitem{ZZ.2023}
Z.~Zhang and R.~Zi.
\newblock Stability threshold of {C}ouette flow for 2{D} {B}oussinesq equations
  in {S}obolev spaces.
\newblock {\em J. Math. Pures Appl. (9)}, 179:123--182, 2023.

\bibitem{ZKLH.1993}
X.~H. Zhao, K.~H. Kwek, J.~B. Li, and K.~L. Huang.
\newblock Chaotic and resonant streamlines in the {ABC} flow.
\newblock {\em SIAM J. Appl. Math.}, 53(1):71--77, 1993.

\end{thebibliography}

\end{document}